\documentclass[12pt]{amsart}
\usepackage{epsfig,color}
\usepackage{blindtext}
\usepackage{graphicx}
\usepackage{enumitem}
\usepackage{url}
\usepackage{amssymb}
\usepackage{graphicx,import}
\usepackage{comment}
\usepackage{esint}
\usepackage{xcolor}
\usepackage{mathtools}
\usepackage{comment}

\usepackage[margin = 1in] {geometry}

\usepackage{hyperref}
\usepackage{dsfont}

\newtheorem{theorem}{Theorem}[section]

\newtheorem{proposition}[theorem]{Proposition}
\newtheorem{lemma}[theorem]{Lemma}
\newtheorem{corollary}[theorem]{Corollary}
\newtheorem{question}[theorem]{Question}

\theoremstyle{definition}

\newtheorem{definition}[theorem]{Definition}
\newtheorem*{remark*}{Remark}

\newtheorem{remark}[theorem]{Remark}

\numberwithin{equation}{section}

\newcommand{\cA}{\mathcal{A}}
\newcommand{\cB}{\mathcal{B}}
\newcommand{\cH}{\mathcal{H}}
\newcommand{\cS}{\mathcal{S}}
\newcommand{\cY}{\mathcal{Y}}
\newcommand{\R}{\mathbb{R}}

\newcommand{\N}{\mathbb{N}}
\newcommand{\Z}{\mathbb{Z}}
\newcommand{\D}{\mathbb{D}}
\newcommand{\B}{\mathbb{B}}

\DeclareMathOperator{\Ric}{Ric}
\DeclareMathOperator{\Lip}{Lip}

\DeclareMathOperator{\dist}{dist}
\DeclareMathOperator{\intr}{int}

\newcommand{\eps}{\varepsilon}
\newcommand{\wto}{\overset{\ast}{\rightharpoonup}}

\title[Optimal regularity for minimizers of the PMC functional over isotopies]{Optimal regularity for minimizers of the prescribed mean curvature functional over isotopies}

\author{Lorenzo Sarnataro}
\address{Department of Mathematics, Princeton University, Princeton, NJ 08540, USA}
\email{lorenzos@princeton.edu}

\author{Douglas Stryker}
\address{Department of Mathematics, Princeton University, Princeton, NJ 08540, USA}
\email{dstryker@princeton.edu}

\begin{document}

\begin{abstract}
    We prove the optimal $C^{1,1}$ regularity for minimizers of the prescribed mean curvature functional over isotopy classes. As an application, we find an embedded sphere of prescribed mean curvature in the round 3-sphere for an open dense set of prescribing functions with $L^{\infty}$ norm at most 0.547.
\end{abstract}

\maketitle

% \tableofcontents

\section{Introduction}
Given a Riemannian manifold $(M^n, g)$ and a function $h: M \to \R$, a two-sided immersion $\Sigma^{n-1} \to M$ is said to have \emph{prescribed mean curvature} $h$ if the mean curvature of the immersion satisfies $H_{\Sigma} = h\vert_{\Sigma} \cdot \nu_{\Sigma}$, where $\nu_{\Sigma}$ is a continuous unit normal vector field along $\Sigma$. The problem of the existence of prescribed mean curvature hypersurfaces has been a rich source of motivation for work in geometric analysis.

An essential observation is that $h$-prescribed mean curvature hypersurfaces arise as critical points of the \emph{$h$-prescribed mean curvature functional}
\begin{equation}\label{eqn:ah_functional}
    \cA^h(\Omega) := \cH^{n-1}(\partial \Omega) - \int_{\Omega} h\ d\cH^n,
\end{equation}
where $\Omega \subset M$ is an open set with smooth boundary. As in the study of minimal hypersurfaces and the area functional, this observation enables the utilization of variational techniques to produce $h$-prescribed mean curvature hypersurfaces.

The simplest variational technique to produce critical points of a functional is to minimize the functional over some nontrivial class. In ambient dimensions 3 through 7, classical work of \cite{FedFlem}, \cite{Flem}, \cite{DG}, \cite{Alm66}, \cite{simons}, and \cite{hardtsimon79} guarantees that any minimizer of area in a Riemannian manifold over integral currents with a fixed boundary is a smooth embedded minimal hypersurface. In the same dimensions, these results have been adapted by \cite{morgan} (see \cite[Theorem 2.2]{ZZPMC}) to the prescribed mean curvature setting, where the minimization of $\cA^h$ occurs over Caccioppoli sets that are allowed to differ from a fixed initial Caccioppoli set in a sufficiently small region $U \subset M$. In this setting, the boundary of any minimizer in $U$ is a smooth embedded $h$-prescribed mean curvature hypersurface.

While this regularity theory is very powerful, it is important to remember that the classes over which the minimization occurs are very large. In particular, elements in the same class can have arbitrarily complicated topology. Hence, we cannot draw any general conclusions about the topology of minimizers, which would be a natural hope for any geometric problem.

In the case of the area functional (i.e.\ $h=0$) in 3-manifolds, there has been significant work studying the regularity of minimizers of area over smaller classes with controlled topology. A major breakthrough was the work of \cite{AS}, where minimizers of the area functional in $\R^3$ over embedded disks with certain fixed boundaries are shown to be smooth embedded minimal surfaces. By an ingenious procedure to simplify the topology of elements in a minimizing sequence, \cite{MSY} used the regularity of \cite{AS} to show that minimizers of the area functional in a closed 3-manifold over an isotopy class\footnote{We emphasize that isotopies preserve the topology of any submanifold.} are smooth embedded minimal surfaces.

Despite the successful application of the regularity theory of \cite{AS} and \cite{MSY} in producing minimal surfaces with controlled topology in 3-manifolds (see for example \cite{smith}, \cite{Colding_DeLellis}, \cite{DeLellis_Pell}, \cite{HaslKet}, \cite{Ketover}), there has been no generalization\footnote{A result along these lines is claimed in \cite{yaublackhole}; see later in the introduction for a discussion of this paper.} of these techniques to the functional $\cA^h$.

In this paper, we generalize of the ideas of \cite{AS} and \cite{MSY} to the $\cA^h$ functional. For an open set $U \subset M$, we let $\mathcal{I}(U)$ denote the set of isotopies $\phi : [0, 1] \times M \to M$ satisfying:
\begin{itemize}
    \item $\phi(t, \cdot) : M \to M$ is a smooth diffeomorphism,
    \item $\phi(0, \cdot) = \mathrm{id}$,
    \item and $\phi(t, x) = x$ for $x \notin U$.
\end{itemize}
Our goal is to deduce the regularity of minimizers for the problem
\[ \inf_{\phi \in \mathcal{I}(U)} \cA^h(\phi(1, \Omega_0)). \]
We deduce the following regularity statement.

\begin{theorem}\label{thm:main_isotopy_regularity}
Let $h: M \to \R$ be a smooth function, and let $c := \sup_M |h|$. Let $U \subset M$ be a sufficiently small (see \S\ref{sec:MSYintreg} for a precise definition) open set with $C^1$ boundary, and let $\Omega_0 \subset M$ be an open subset with smooth boundary having tranverse intersection with $\partial U$. Let $\{\phi_k\}_{k\in\N} \subset \mathcal{I}(U)$, $\Omega_k := \phi_k(1, \Omega_0)$, and $\Sigma_k := \partial \Omega_k \cap U$ be a sequence satisfying
\begin{equation}\label{eqn:main_isotopy_problem}
    \cA^h(\Omega_k) \leq \inf_{\phi \in \mathcal{I}(U)} \cA^h(\phi(1, \Omega_0)) + \eps_k
\end{equation}
for $\eps_k \to 0$. Then there is a varifold $V$, an open set $\Omega \subset U$, and a subsequence (not relabeled) so that
\[ \mathbf{1}_{\Omega_k \cap U} \xrightarrow{L^1} \mathbf{1}_{\Omega},\ \ D\mathbf{1}_{\Omega_k \cap U} \wto D\mathbf{1}_{\Omega},\ \ \mathbf{v}(\Sigma_k) \rightharpoonup V. \]
$V$ is an integer rectifiable varifold with $c$-bounded first variation. For every $x \in \mathrm{spt}\|V\|$, we have $\Theta^2(\|V\|, x) = n_x \in \N$. If $h(x) \neq 0$, then $n_x \in \{1\} \cup 2\N$. For every $x \in \mathrm{spt}\|V\|$, there is a neighborhood $W_x$ of $x$ so that the following hold.
\begin{itemize}
    \item $V \llcorner G(W_x,2) = \sum_{l=1}^{n_x}\mathbf{v}(N_l, 1)$, where $N_l \subset W_x$ is a $C^{1,1}$ surface with $c$-bounded first variation. Moreover, each $N_l$ is on one side of $N_{l'}$ intersecting tangentially at $x$ for any $l,\ l' \in \{1, \hdots, n_x\}$.
    \item If $n_x = 1$, then
    \begin{itemize}
        \item $\|V\| \llcorner W_x = |D\mathbf{1}_{\Omega}| \llcorner W_x$,
        \item $V \llcorner G(W_x,2) = \mathbf{v}(N, 1)$, where $N \subset W_x$ is a smooth stable surface with prescribed mean curvature $h$ with respect to the set $\Omega \cap W_x$.
    \end{itemize}
    \item If $h(x) \neq 0$ and $n_y = 2n$ for all $y \in \mathrm{spt}\|V\| \cap W_x'$ for an open set $W_x' \subset W_x$ containing $x$, then $V \llcorner G(W'_x,2) = \mathbf{v}(N, 2n)$, where $N\subset W_x'$ is a smooth stable minimal surface. Moreover, $\Omega \cap W_x' = W_x' \setminus N$ if $h(x) > 0$, and $\Omega \cap W_x' = \varnothing$ if $h(x) < 0$.
    \item If $h(x) \neq 0$, $n_x = 2n$, and there is a sequence $x_j \to x$ with $x_j \in \mathrm{spt}\|V\|$ and $n_{x_j} \neq 2n$, then $n = 1$ and the surfaces $N_1,\ N_2 \subset W_x$ additionally satisfy that the generalized mean curvature of $N_1$ points towards $N_2$ and vice versa. Moreover, the set $\Gamma=\partial(N_1 \cap N_2) \cap W_x$ is locally contained in a $C^{1,\alpha}$ curve and satisfies $\cH^1(\Gamma)<+\infty$.
\end{itemize}
Suppose additionally that $U$ is a sufficiently small geodesic ball. Then there is a neighborhood $Y$ of $\partial U$ with $V \llcorner G(Y, 2) = \mathbf{v}(N, 1)$, where $N \subset Y$ is a smooth stable surface of prescribed mean curvature $h$ with respect to $\Omega$ satisfying $\partial N = \partial \Sigma$.
\end{theorem}

In \S\ref{sec:examples}, we provide the details of examples that show that this regularity is sharp. We note that an adaptation of \cite{MSY} to the constant mean curvature setting with a stronger regularity conclusion than Theorem \ref{thm:main_isotopy_regularity} was claimed in \cite[Theorem 2.1]{yaublackhole}, but without a detailed proof. It appears that the example in \S\ref{sec:example_main} supplies a counterexample to this statement. In any case, this example illustrates an essential obstruction to a proof along the lines of \cite{AS} and \cite{MSY}.

\begin{remark}
In groundbreaking work, \cite{4spheres} prove the existence of 4 embedded minimal spheres in a generic metric on $S^3$, which solves an important and longstanding conjecture in the theory of minimal surfaces. The proof of \cite{4spheres} uses extensively Theorem \ref{thm:main_isotopy_regularity} as well as the new ingredients of the proof we develop in this work.
\end{remark}

\subsection{Application to existence in the 3-sphere}
Before unpacking the ideas of our proof of Theorem \ref{thm:main_isotopy_regularity}, we present a modest application of the theory to the existence of embedded $h$-prescribed mean curvature 2-spheres in some metrics on $S^3$. Since any embedded 2-sphere in $S^3$ can be moved by isotopies to a constant map, minimization techniques fail to produce nontrivial examples. Hence, we require a min-max variational approach.

To motivate our result, we recall the following question of Yau.
\begin{question}[{\cite[Problem 59]{yauQuestions}}]\label{quest:yau}
For which smooth functions $h : \R^3 \to \R$ does there exist an embedded sphere with prescribed mean curvature $h$ (with respect to the flat metric)?
\end{question}

For progress towards Question \ref{quest:yau} (and related problems) relying more on PDE techniques, we refer to \cite{BK}, \cite{TW}, \cite{gerhardt}, \cite{ye}, \cite{PX}, and \cite{smallspheres}. There is also partial progress outside the min-max framework in \cite{yauremark}. For progress towards \ref{quest:yau} without topological control, we refer to \cite{liam}.

Since noncompact ambient spaces present unique difficulties for min-max theory, we consider the analogous question in $S^3$.
\begin{question}\label{quest:sphere1}
For which smooth functions $h : S^3 \to \R$ does there exist an embedded sphere with prescribed mean curvature $h$ (with respect to the round metric)?
\end{question}

For the area functional, a min-max variational framework was developed by \cite{almgren}, \cite{pitts}, and \cite{SchoenSimon} to produce a closed embedded minimal hypersurface in any closed Riemannian manifold of dimension between 3 and 7. These techniques have been refined to produce an abundance of closed minimal hypersurfaces (see \cite{MNinfinitlymany}, \cite{LMNweyl}, \cite{IMNdensity}, \cite{MNSequidistribution}, and \cite{song}). This approach was successfully adapted to the functional $\cA^h$ in \cite{ZZCMC} and \cite{ZZPMC}. In these works, min-max is carried out over similarly large classes of hypersurfaces without controlled topology, and the regularity theory relies on the classical  regularity theory for minimizers discussed earlier in the introduction. Hence, we cannot draw any general conclusions about the topology of the resulting minimal hypersurfaces.

An alternative min-max approach in the minimal hypersurfaces case was developed by Guaraco (\cite{guaraco}) and Gaspar-Guaraco (\cite{GG_AC}, \cite{GG_Weyl}) using phase transitions for the Allen-Cahn equation, building on the regularity theory of Hutchinson-Tonegawa \cite{Hutchinson2000}, Tonegawa \cite{Tonegawa2005} and Tonegawa-Wickramasekera \cite{Tonegawa2012}. (We refer the reader to \cite{Dey2022} for a comparison between the Almgren-Pitts theory and the Allen-Cahn theory.) Similarly, \cite{BW} developed an analogue of the Allen-Cahn min-max approach for prescribed mean curvature hypersurfaces, building on the regularity theory developed in \cite{BWCMC} and \cite{BWPMC}. We emphasize that \cite{BW} only applies to nonnegative prescribing functions. As above, the topology of the produced hypersurfaces from these techniques is uncontrolled.

A min-max approach for the area functional in closed 3-manifolds that controls the genus of the produced minimal surface, called Simon-Smith min-max, was developed by \cite{smith} (see \cite{Colding_DeLellis} and \cite{DeLellis_Pell}), building on the regularity theory of \cite{AS} and \cite{MSY}. In particular, it follows from these works that every metric on $S^3$ admits an embedded minimal sphere. For progress towards the existence of more embedded minimal spheres using this approach, we refer the reader to \cite{HaslKet}.

As an application of our regularity theory, we develop a min-max approach for the $\cA^h$ functional along the lines of Simon-Smith. In particular, we give the following partial answer to Question \ref{quest:sphere1}.

\begin{theorem}\label{thm:PMC_in_sphere}
There is an open dense set $\mathfrak{P} \subset \{h \in C^{\infty}(S^3) \mid |h| \leq 0.547\}$ so that there is an embedded sphere with prescribed mean curvature $h$ with respect to the round metric for any $h \in \mathfrak{P}$. 
In particular, $\{h \in C^{\infty}(S^3) \mid 0 < h \leq 0.547\} \subset \mathfrak{P}$.
\end{theorem}

The restriction to an open dense set is the same restriction from \cite{ZZPMC}. We recall that unique continuation is an important ingredient in the regularity theory for min-max surfaces. It follows from \cite[Proposition 3.8 and Corollary 3.18]{ZZPMC} that unique continuation holds for an open dense set of prescribing functions that includes all positive functions.

The restriction to functions with $L^{\infty}$ norm bounded by 0.547 is the more significant assumption. This restriction is a consequence of the more delicate regularity theory for minimizers of the functional (\ref{eqn:ah_functional}) over isotopies, as compared to the cleaner regularity conclusion for the area functional. We emphasize that our regularity theory is sharp, so it is not clear how to extend the existence theory beyond this class of prescribing functions using the techniques developed in this paper. Moreover, we are not aware of any counterexamples for large prescribing functions in this setting, so the question remains open.

Outside the setting of the round metric on $S^3$, it is natural to consider the following related problem.

\begin{question}\label{quest:sphere2}
For which metrics $g$ on $S^3$ and which constants $c \in \R_{>0}$ does there exist an embedded sphere with constant mean curvature $c$ (with respect to the metric $g$)?\footnote{As a point of reference, it is conjectured in the introduction of \cite{RSdegree} that there is a positive answer to \ref{quest:sphere2} for any $c$ and any metric $g$ with positive sectional curvature.}
\end{question}

Outside the min-max framework, there has been considerable progress on the existence of constant mean curvature immersed spheres in homogeneous 3-manifolds (see \cite{DMsol}, \cite{meeksSol}, \cite{MMPRmanifolds}, and \cite{MMPRsphere}).

A min-max approach using the Dirichlet energy was developed in \cite{sacksuhlenbeck} to produce branched immersed minimal spheres. This approach was adapted to the constant mean curvature setting by \cite{chengzhou}, where they show that every metric on $S^3$ admits a branched immersed sphere of constant mean curvature $c$ for almost every $c \in \R$. However, it is unclear if the branched immersions produced by these techniques are genuine immersions, let alone embeddings.

Using our regularity theory, we provide the following partial answer to (\ref{quest:sphere2}).

\begin{theorem}\label{thm:CMC_in_sphere}
There are explicit constants $\eps_0 > 0$ and $c_0 > 0$ so that if $(S^3, g)$ is the induced metric of a graph over the round 3-sphere $\mathbf{S}^3 \subset \R^4 \subset \R^N$ with $C^2$ norm bounded by $\eps_0$\footnote{By this, we mean that $(S^3, g)$ is the induced metric of $\{x + f(x) \mid x_1^2 + x_2^2 + x_3^2 + x_4^2 = 1,\ x_5, \hdots, x_N = 0\}$ where $N \geq 4$ and $f : S^3 \to \R^N$ satisfies $\|f\|_{C^2} \leq \eps_0$.} and $0 \leq c \leq c_0$, then $(S^3, g)$ admits an embedded sphere with constant mean curvature $c$.
\end{theorem}

We remark that a key feature of Theorems \ref{thm:PMC_in_sphere} and \ref{thm:CMC_in_sphere} is that the restrictions on the prescription function and metric are explicit. Indeed, we expect that versions of these theorems hold with inexplicit constants by an implicit function theorem argument, using only the well-established existence theory for minimal spheres. Unlike such a result, our work provides actual testable criteria for existence.

\subsection{Idea of the proof of regularity}
The proof of \ref{thm:main_isotopy_regularity} follows the same strategy as \cite{AS} and \cite{MSY}. For clarity, we outline the approach of \cite{AS} and \cite{MSY}, and then highlight the key differences.

\begin{enumerate}
    \item \emph{Reduce the problem of minimizing area in an isotopy class to the problem of minimizing area over disks.} In \cite{MSY}, a procedure called $\gamma$-reduction is introduced. In essence, $\gamma$-reduction is used to delete small necks that either generate nontrivial topology or separate large components. By applying the $\gamma$-reduction procedure a maximal number of times to a minimizing sequence for area over isotopies (and checking that the limit surface does not change), \cite{MSY} use the $\gamma$-irreducibility property to show that the new sequence is a minimizing sequence for area over disks in any small ball.
    \item \emph{Reduce the problem of minimizing area over disks to the problem of minimizing area over ``stacked disks'' in a cylinder.} By zooming in along a minimizing sequence for area over disks at a point where the limit has a tangent plane, \cite{AS} use area comparison estimates to throw away components that are not close to the tangent plane in a small ball, without changing the limit. Therefore, it suffices to study a minimizing sequence for area over disjoint unions of disks with homotopically nontrivial boundary in a fixed cylinder. We call disks in this arrangement ``stacked disks''.
    \item \emph{Reduce the problem of minimizing area over many stacked disks in a cylinder to the problem of minimizing area over one stacked disk in a cylinder.} To isolate one disk from many in the stacked disks arrangement, \cite{AS} develop a simple procedure to disentangle intersecting disks without increasing area. It then suffices to study a minimizing sequence for area over stacked disks with only one disk.
    \item \emph{Conclude by Allard's regularity.} For a minimizing sequence for area over stacked disks with one disk, it is easy to obtain good lower and upper bounds for area. These estimates allow \cite{AS} to deduce the smooth regularity of the limit using the regularity theorem of \cite{allard}.
\end{enumerate}

Steps (1) and (2) work for $\cA^h$ with only technical modifications. For example, we make frequent use of the following basic idea. In the $\cA^h$ case, inequalities will have an extra volume term as compared to the area functional case. We absorb the volume term into the area terms using the isoperimetric inequality. The resulting inequalities resemble the analogous inequalities in the case of the area functional, only with slightly worse coefficients. 

The essential difficulty is adapting steps (3) and (4). To highlight the key difference from the area functional, consider the case illustrated in Figure \ref{fig:stack}. In this case, we have a stack of two disks in a cylinder which is the boundary of the disjoint union of the set above the top disk and the set below the bottom disk. If we replace the bottom disk with a disk that passes through the top disk, the volume between the two disks is now covered twice. So, no matter how we disentangle these intersecting disks, the volume that was covered twice can now only be covered once. Hence, any disentangling procedure will affect $\cA^h$ in a complicated way, unlike the area functional.

\begin{figure}
    \centering
    \includegraphics[width = 0.8 \textwidth]{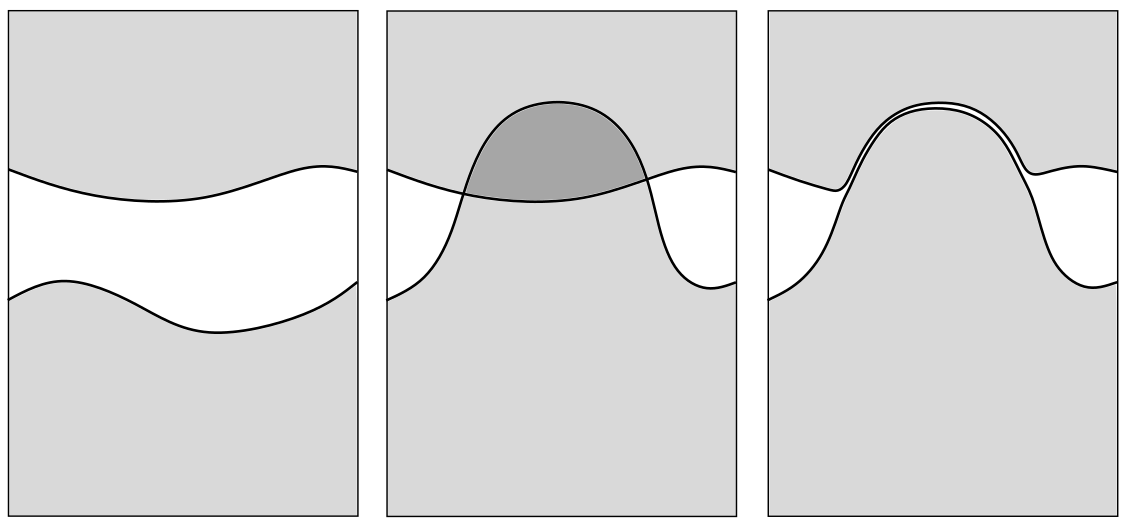}
    \caption{The left image depicts the initial stack of two disks. The middle image depicts a replacement for the bottom disk that intersects the top disk. The right image depicts a possible disentanglement of the disks.}
    \label{fig:stack}
\end{figure}

In the course of our proof, we develop two weaker versions of steps (3) and (4). Together, the weak versions of these steps can be used to deduce Theorem \ref{thm:main_isotopy_regularity} with sub-optimal $C^{1,\alpha}$ regularity.

\emph{First weak version}. In this version of steps (3) and (4), we deduce the $C^{1,\alpha}$ regularity of the limit of each disk in the stacked disk arrangement.

\begin{enumerate}
    \item[(3')] \emph{Each disk in a stacked disk arrangement minimizing $\cA^h$ is itself minimizing for $\cA^{\tilde{h}}$.} Using a delicate disentanglement procedure, we find that the limit of each disk is minimizing for $\cA^{\tilde{h}}$, where $\tilde{h}$ is a measurable function bounded by $\sup |h|$ that depends on the limit. We emphasize that $\tilde{h}$ cannot be chosen to be continuous in general.
    \item[(4')] \emph{Conclude by Allard's regularity.} Using the same area estimates as step (4), we can apply the theorem of \cite{allard}. However, the first variation of the limit is only known to be $L^{\infty}$ by step (3'), so the limit of each disk is $C^{1,\alpha}$.
\end{enumerate}

\emph{Second weak version}. In this version of steps (3) and (4), we determine where the limit is a smooth multiplicity one $h$-prescribed mean curvature surface and where is it a smooth even multiplicity minimal surface, in the case where $h$ is positive.

\begin{enumerate}
    \item[(3'')] \emph{Reduce the problem of minimizing $\cA^h$ over many stacked disks in a cylinder to the problem of minimizing $\cA^h$ over stacked disks with only one or two disks.} When $h$ has a sign, it is easier to keep track of volumes under disentanglement. Similar in spirit to step (3) from \cite{AS}, we use this observation in combination with our disentanglement procedure to decompose large stacks of disks into units consisting of only one or two disks (depending on orientations).
    \item[(4'')] \emph{Conclude by Allard's regularity.} The case of only one stacked disk now follows exactly as in step (4) from \cite{AS}, where we find that the limit is a smooth embedded $h$-prescribed mean curvature surface. In the two stack case, if the limits of the two disks agree on some open set, we see that the limit is stationary. We then follow step (4) from \cite{AS} to find that the limit is a smooth embedded minimal surface with multiplicity two.
\end{enumerate}

After deducing the $C^{1,\alpha}$ version of Theorem \ref{thm:main_isotopy_regularity}, we use some tools from the theory of free boundary problems to upgrade the regularity of our minimizers to $C^{1,1}$. Indeed, we can show that minimizers locally solve an $n$-membrane problem for a quasilinear elliptic operator (see \cite{VC74}, \cite{SilvestreTM} and \cite{SavinYu1} for more background), to which the recent results of \cite{WZ_multiple} apply.

\subsection{Idea of the proof of existence in the 3-sphere}
There are a few key difficulties in carrying over the min-max program of \cite{Colding_DeLellis} and \cite{DeLellis_Pell} to our setting, though much remains unchanged. Unique continuation and strong compactness are the two major obstacles.

In the classical min-max settings, unique continuation for surfaces with zero or prescribed mean curvature is used in an essential way to patch together iterated replacements to prove regularity. However, unique continuation does not hold for surfaces with the regularity of Theorem \ref{thm:main_isotopy_regularity} (e.g.\ a multiplicity two minimal surface can suddenly split into two prescribed mean curvature surfaces).

Ignoring the full power of our regularity theory, we overcome the unique continuation issue by finding situations in which min-max solutions and replacement varifolds have density one everywhere. In this case, the regularity of Theorem \ref{thm:main_isotopy_regularity} is the same as in the Almgren-Pitts prescribed mean curvature setting (e.g.\ smooth stable surfaces with prescribed mean curvature). Unique continuation then applies (assuming the prescribing function is sufficiently nice).

To achieve density one, we make two observations about varifolds in the round 3-sphere. First, a varifold limit of a min-max sequence for the $\cA^h$ functional has a mass upper bound (in terms of a bound on the prescribing functional) by comparing with the optimal sweepout for the area functional. Second, a monotonicity formula for the Willmore energy in Euclidean space due to \cite{Simon_Willmore} and \cite{Topping1998} (suitably modified for the varifold setting) gives a lower bound for the mass of a varifold in $S^3 \subset \R^4$ with a point of density at least two (in terms of a bound on the first variation). These two bounds contradict when the prescribing function is bounded in absolute value by $0.547$.

Another essential ingredient in both the regularity theory and genus bounds is a strong compactness theory for minimizers. In the more classical settings, this strong compactness follows from curvature estimates for stable minimal surfaces. Since we only expect $C^{1,1}$ regularity in general, this approach requires special attention. In fact, there is an important and subtle step in the standard point picking contradiction argument for the curvature estimates of stable minimal surfaces that uses the Schauder estimates to upgrade from $C^{1,\alpha}$ convergence to $C^2$ convergence. This approach manifestly fails in our setting, because minimizers are not even $C^2$. Instead, we use a $C^{1,1}$ estimate of \cite{WZ_multiple} for the $n$-membrane problem that allows us to close the usual point picking argument. This approach supplies a sufficiently strong compactness theorem for the rest of the regularity theory and genus bounds.

\subsection{Outline of the paper}
In \S\ref{sec:notation}, we establish notations and conventions for the rest of the paper. In \S\ref{sec:examples}, we carry out the details of two illustrative examples. In Part \ref{part:disks}, we modify the arguments of \cite{AS} (e.g.\ minimizing area over disks) to the the prescribed mean curvature setting. In Part \ref{part:isotopy}, we modify the arguments of \cite{MSY} (e.g.\ minimizing over isotopy classes) to the prescribed mean curvature setting, and prove Theorem \ref{thm:main_isotopy_regularity}. In Part \ref{part:compactness}, we develop a strong compactness theory for minimizers with only $C^{1,1}$ regularity. In Part \ref{part:min-max}, we apply the above regularity theory in the Simon-Smith min-max program to prove Theorems \ref{thm:PMC_in_sphere} and \ref{thm:CMC_in_sphere}.

\subsection{Acknowlegements}
The authors are indebted to their advisor Fernando Cod\'a Marques for suggesting this topic and conversing regularly about the problem. The authors would also like to thank Camillo De Lellis for answering several key questions during the completion of this work, as well as Ovidiu Savin and Hui Yu for providing useful insights about membrane problems. The authors are grateful to Xin Zhou, Daniel Ketover, and Costante Bellettini for enlightening discussions surrounding this work. The authors are grateful to Paul Minter, Zhihan Wang, and Thomas Massoni for helping work out various technical issues. Finally, the authors would like to thank Zhichao Wang and Xin Zhou for pointing out a gap in the original version of \S\ref{sec:freeboundary}, and for sharing their work \cite{WZ_multiple}, as well as the anonymous referee for several comments that improved the clarity of the exposition.

D.S.\ was supported by an NDSEG fellowship.

\section{Notation and Conventions}\label{sec:notation}
\begin{itemize}
    \item $\B$ denotes $B_1^{\R^3}(0)$.
    \item $\D$ denotes $\overline{B}_1^{\R^2}(0)$.
    \item $G(U, 2)$ denotes the Grassmannian of tangent 2-planes over the set $U \subset M$.
    \item $\mathcal{I}(U)$ is the set of isotopies $\phi : [0, 1] \times M \to M$ satisfying
    \[ \phi(0, \cdot) = \mathrm{id}\ \ \text{and}\ \ \phi(t, \cdot)\vert_{M \setminus U} = \mathrm{id}. \]
    \item $\cH^n$ denotes the $n$-dimensional Hausdorff measure on $(M,g)$.
    \item $\cA^h$ denots the $h$-prescribed mean curvature functional; namely, for $h : M \to \R$ and a set of finite perimeter $\Omega \subset M$, we define
    \[ \cA^h(\Omega) := \cH^2(\partial \Omega) - \int_{\Omega} h\ d\cH^3. \]
    \item $\mathbf{v}(N, m)$ denotes the varifold associated to a $C^1$ surface $N \subset M$ with multiplicity $m \in \N$.
    \item $\delta V(X)$ is the first variation of a varifold $V$ with respect to the vector field $X$. In the case $V$ has bounded first variation, we use the convention
    \[ \delta V(X) = \int_{G(M,2)} \mathrm{div}_{\pi}X\ dV(x, \pi) = \int_M X \cdot H_V\ d\|V\|, \]
    where $H_V$ is the generalized mean curvature vector of $V$. We note that with this convention, the unit round sphere $S^2 \subset \R^3$ has mean curvature pointing out of the unit ball with length $2$.
    \item $V_k \rightharpoonup V$ denotes convergence in the sense of varifolds.
    \item $\mu_k \wto \mu$ denotes weak star convergence in the sense of \cite[\S4.3]{maggi}.
    \item $[\Sigma]$ denotes the integer 2-current associated to an oriented surface $\Sigma$.
    \item $[\Omega]$ denotes the integer 3-current associated to a set of finite perimeter $\Omega$, oriented so that $\partial [\Omega]$ is the integer 2-current associated to $\partial \Omega$ with orientation given by the outward pointing unit normal.
    \item $\mathbb{M}(T)$ denotes the mass of an integral current.
    \item $d_U$ denotes the distance function from $U$ in $M$.
    \item $U(\theta)$ denotes the set $\{x \in M \mid d_U(x) < \theta\}$.
    \item $U_{\xi}$ denotes the set $\{x \in U \mid d(x, \partial U) > \xi\}$.
\end{itemize}

\section{Examples}\label{sec:examples}
\subsection{Failure of prescribed mean curvature}
Here we supply a general class of examples where the limit of any minimizing sequence does not have prescribed mean curvature.

Let $M^3 = S^1 \times \Sigma$ where $\Sigma$ is a closed surface. Let $g$ be any metric on $M$, and $h : M \to \R$ any positive function. Let $\Omega_0 := I \times \Sigma \subset M$, where $I$ is an open interval in the $S^1$ factor. Let $\Sigma_1$ and $\Sigma_2$ be the components of $\partial \Omega_0$, which are embedded surfaces in the same isotopy class as $\{0\} \times \Sigma \subset M$. We define
\[ m_0 := \inf_{\phi \in \mathcal{I}(M)} \cH^2(\phi(1, \Sigma_1)) = \inf_{\phi \in \mathcal{I}(M)} \cH^2(\phi(1, \Sigma_2)). \]
By treating the area and volume terms separately, we see that
\[ \inf_{\phi \in \mathcal{I}(M)} \cA^h(\Omega_0) \geq 2m_0 - \int_{M} h. \]
Moreover, this lower bound can be achieved by taking a sequence satisfying
\[ \mathbf{1}_{\phi_k(1, \Omega_0)} \xrightarrow{L^1} \mathbf{1}_{M} \]
and
\[ \lim_{k \to \infty} \cH^2(\phi_k(1, \Sigma_1)) = \lim_{k \to \infty} \cH^2(\phi_k(1, \Sigma_2)) = m_0. \]
By \cite{MSY}, the varifold limit of the boundaries is a multiplicity 2 stable minimal surface, and therefore does not have prescribed mean curvature $h > 0$.

\subsection{Failure of smoothness}\label{sec:example_main}
We give an example where the limit of any minimizing sequence is not $C^2$. This example is a prototype of the ``stacked disk'' minimization problem used to prove regularity, which is outlined in the introduction. The failure of smoothness in this example therefore illustrates the essential reason for the failure of smoothness in general.

Let $U := B_1^{\R^2}(0) \times (-1, 1) \subset \R^3$. Let
\[ \Omega_1 := B_1^{\R^2}(0) \times (-1, -\eps)\ \ \text{and}\ \ \Omega_2 := B_1^{\R^2}(0) \times (\eps, 1). \]

\begin{proposition}
The problem
\[ \inf_{\phi \in \mathcal{I}(U)} \cA^c(\phi(1, \Omega_1)) \]
has a unique minimizer whose boundary in $U$ is the upward curved spherical cap of radius $2/c$ with boundary $\partial U \cap \{x_3 = -\eps\}$.
\end{proposition}
\begin{proof}
We use the calibration technique, see \cite[Lemma 1.1.1]{CM}.

Let $\Sigma^*$ denote the surface of the upward curved spherical cap of radius $2/c$ with boundary $\partial U \cap \{x_3 = -\eps\}$. Let $\Omega_1^* \subset U$ be the set below $\Sigma^*$ in $U$. It is clear that $\Omega_1^* = \phi^*(1, \Omega_1)$ for some $\phi^* \in \mathcal{I}(U)$.

Note that $\Sigma^*$ is the graph of a function $u$ over $B_1^{\R^2}(0)$. Moreover, $u$ satisfies the mean curvature equation
\[ \mathrm{div}\left(\frac{\nabla u}{\sqrt{1 + |\nabla u|^2}}\right) = -c. \]
Let $\omega$ be the 2-form given by
\[ \omega(X, Y) = \mathrm{det}(X, Y, \nu), \]
where
\[ \nu := \frac{(-u_x, -u_y, 1)}{\sqrt{1 + |\nabla u|^2}}. \]
By the computation from \cite[Lemma 1.1.1]{CM}, we have
\[ \omega = \frac{dx \wedge dy - u_xdy \wedge dz - u_ydz\wedge dx}{\sqrt{1+|\nabla u|^2}} \]
and
\begin{align*}
d\omega
& = \left[\left(\frac{-u_x}{\sqrt{1+|\nabla u|^2}}\right)_x + \left(\frac{-u_y}{\sqrt{1+|\nabla u|^2}}\right)_y\right] dx\wedge dy \wedge dz = c\ dx\wedge dy \wedge dz.
\end{align*}
Note that $|\omega(X,Y)| \leq 1$ for any pair of orthogonal unit vectors, with equality if and only if $X, Y \in T\Sigma^*$.

Now let $\Omega_1' = \phi(1, \Omega_1)$ be any competitor with $\phi \in \mathcal{I}(U)$. Let $\Sigma' = \partial \Omega_1' \cap U$. Suppose $\Sigma' \neq \Sigma^*$. Then by Stokes' theorem, we have
\begin{align*}
    \cA^c(\Omega^*)
    & = \cH^2(\Sigma^*) - c\cH^3(\Omega^*)\\
    & = \int_{\Sigma^*} \omega - \int_{\Omega^*} d\omega\\
    & = \int_{\Sigma'} \omega - \int_{\Omega'} d\omega\\
    & < \cH^2(\Sigma') - c\cH^3(\Omega') = \cA^c(\Omega').
\end{align*}
The conclusion follows.
\end{proof}

By vertical reflection, the analogous statement holds for $\Omega_2$. Fix $c \leq 1$. Then for $\eps$ sufficiently small, we shall see that any minimizer for the problem
\[ \inf_{\phi \in \mathcal{I}(U)} \cA^c(\phi(1, \Omega_1 \cup \Omega_2)) \]
must have a minimal interface, and therefore cannot be $C^2$.

We give a more precise description of minimizers.

\begin{proposition}
Any minimizer for the problem
\begin{equation}\label{eqn:min_problem}
    \inf_{\phi \in \mathcal{I}(U)} \cA^c(\phi(1, \Omega_1 \cup \Omega_2))
\end{equation} 
is graphical, symmetric under vertical reflection, and rotationally symmetric.
\end{proposition}
\begin{proof}
We first show that minimizers must be graphical. Let $m$ denote the infimum for the above problem \eqref{eqn:min_problem}. Let $\Omega^* \subset U$ and $V^*$ be the Caccioppoli limit and varifold limit respectively of a minimizing sequence. Let $\Omega^*_1$ and $\Omega^*_2$ respectively be the limits of $\Omega_1$ and $\Omega_2$ under the isotopies of the minimizing sequence. Then $\Omega^* = \Omega_1^* \cup \Omega_2^*$, and
\[ m = \|V^*\|(U) - c\cH^3(\Omega^*) \geq P(\Omega_1^*; U) + P(\Omega_2^*; U) - c\cH^3(\Omega_1^*) - c\cH^3(\Omega_2^*).  \]
Suppose for contradiction that at least one of $\partial \Omega_1^*\cap U$ or $\partial \Omega_2^* \cap U$ is not graphical over $B_1^{\R^2}(0)$. For $\mathbf{x} \in B_1^{\R^2}(0)$, let $l_{\mathbf{x}}$ denote the vertical line $\{(\mathbf{x}, z) \mid z \in \R\}$. Consider new sets
\begin{align*}
    & \Omega_1' := \{(\mathbf{x}, z) \mid -1 < z < \cH^1(\overline{\Omega^*} \cap l_{\mathbf{x}})/2 - 1\},\\
    & \Omega_2' := \{(\mathbf{x}, z) \mid 1 - \cH^1(\overline{\Omega^*} \cap l_{\mathbf{x}})/2 < z < 1\}.
\end{align*}
By Steiner symmetrization (see \cite[Theorem 14.4]{maggi}), we have
\[ \cH^3(\Omega_1' \cup \Omega_2') = \cH^3(\Omega^*) \ \ \text{and} \ \ P(\Omega_1'; U) + P(\Omega_2'; U) \leq P(\Omega_1^*; U) + P(\Omega_2^*; U), \]
where equality holds for the perimeters if and only if $\Omega_i^* \cap l_{\mathbf{x}}$ is connected for all $\mathbf{x}$. Hence, we have
\[ P(\Omega_1'; U) + P(\Omega_2'; U) - c\cH^3(\Omega_1') - c\cH^3(\Omega_2') < m. \]
Since $\partial \Omega_i' \cap U$ is graphical over $B_1^{\R^2}(0)$ and $\partial \Omega_i' \setminus U = \partial \Omega_i^* \setminus U$ by construction, $\Omega' := \Omega_1' \cup \Omega_2'$ is in the closure of the isotopy class of $\Omega$ (e.g.\ use a smooth kernel to approximate the boundary of $\Omega_i'$ by smooth disks with the same boundary), which yields a contradiction.

The proofs of symmetry under reflection and rotation follow similarly by the natural symmetrization procedures. Indeed, these steps are easier, as one can exploit the fact that minimizers are graphical and use the convexity of the function $\sqrt{1 + |x|^2}$.
\end{proof}

In MATLAB, we programmed a discrete version of the $\cA^1$ functional for rotationally and vertically symmetric graphs. Using MATLAB's fmincon, we computed minima for different values of $\eps$, and observed that the minimizers were robust under random initial conditions for fmincon. In Figures \ref{fig:example1}, \ref{fig:example2}, and \ref{fig:example3}, we provide graphics of the solution for values of $\eps$ where the two sheets are disjoint ($\eps = 0.4$), where the two sheets touch at one point ($\eps = 2 -\sqrt{3}$), and where a minimal interface forms ($\eps = 0.1$).

\begin{figure}
    \centering
    \begin{minipage}{0.32\textwidth}
        \includegraphics[width = 0.9 \textwidth]{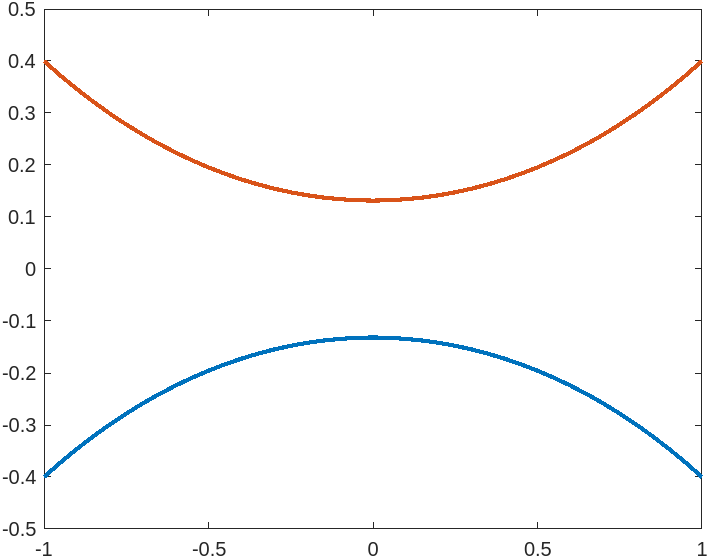}
        \caption{\newline $\eps = 0.4$}
        \label{fig:example1}
    \end{minipage}
    \begin{minipage}{0.32\textwidth}
        \includegraphics[width = 0.9 \textwidth]{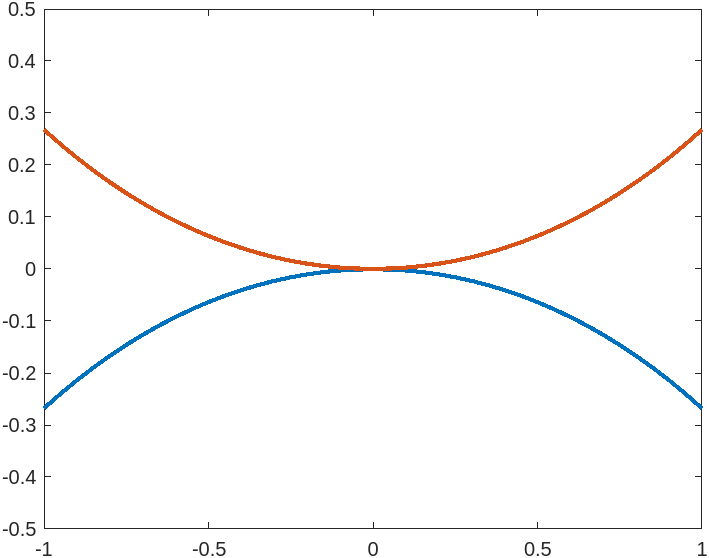}
        \caption{\newline $\eps = 2-\sqrt{3}$}
        \label{fig:example2}
    \end{minipage}
    \begin{minipage}{0.32\textwidth}
        \includegraphics[width = 0.9 \textwidth]{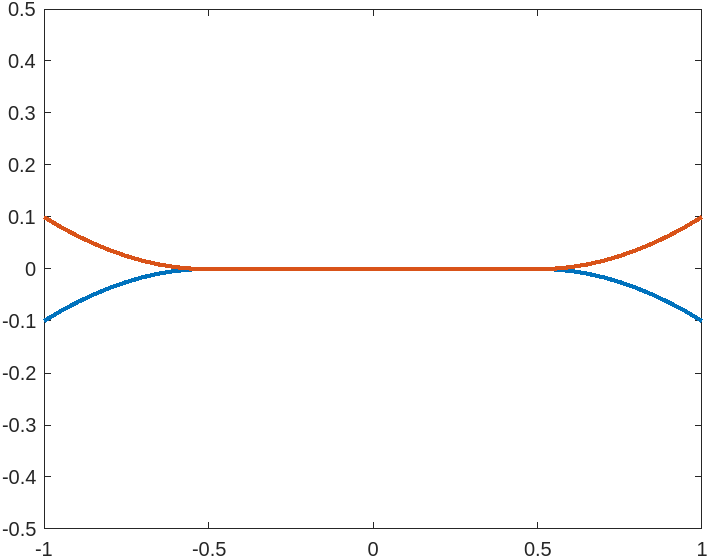}
        \caption{\newline $\eps = 0.1$}
        \label{fig:example3}
    \end{minipage}
\end{figure}

\part{Minimization over disks}\label{part:disks}
We follow the argument of \cite{AS}, with  appropriate modifications to the manifold setting as in \cite{MSY}.

\section{Preliminaries}
Let $(M^3, g)$ be a connected closed\footnote{As in \cite{MSY}, we only need to assume that $M$ is homogeneously regular. In this case, we should also assume that $h$ is in $L^{\infty}$.} Riemannian 3-manifold. Let $h : M \to \R$ be a smooth function with $c \coloneqq \sup_M |h|$. Before we formulate the minimization problem, we need some preliminary facts about the prescribed mean curvature functional and disks in manifolds.

\subsection{Local control}\label{sec:local_control}
Since $M$ is compact, there are constants $\rho_0 > 0$, $\mu > 0$, and $\beta_0 \geq 1$ so that the following hold:
\begin{itemize}
    \item For each $x \in M$, the restriction of the domain and range of $\exp_x$ gives a diffeomorphism
    \[ \phi_x : \overline{B}_{\rho_0}^{T_xM}(0) \to \overline{B}_{\rho_0}^{M}(x). \]
    \item For each $\xi \in \overline{B}_{\rho_0}^{T_xM}(0)$ and $y \in \overline{B}_{\rho_0}^M(x)$, we have
    \begin{equation}\label{eqn:prelim_dexp_small}
        \|d_\xi\phi_x\|,\ \|d_y\phi_x^{-1}\| \leq 2.
    \end{equation}
    \item In geodesic normal coordinates $x^1, x^2, x^3$ for $B_{\rho_0}^M(x)$, we have
    \begin{equation}\label{eqn:prelim_dg_small}
        \sup_{\overline{B}_{\rho_0}^{T_xM}(0)} \left|\frac{\partial g_{ij}}{\partial x^k}\right| \leq \mu/\rho_0,\ \ \sup_{\overline{B}_{\rho_0}^{T_xM}(0)} \left|\frac{\partial^2 g_{ij}}{\partial x^k\partial x^l}\right| \leq \mu/\rho_0^2.
    \end{equation}
    \item For any $\rho \in (0,\rho_0)$ and $x \in M$, we have \begin{equation}\label{eqn:prelim_convex}
        d_{B_{\rho}(x)} \coloneqq \mathrm{dist}(B_{\rho}(x), \cdot) \ \text{is\ convex}
    \end{equation}
    on $B_{\rho_0}(x)\setminus \overline{B_\rho(x)}$.
    \item For any $\rho \in (0,\rho_0)$ and $x \in M$, we have
    \begin{equation}\label{eqn:prelim_mean_convex}
        \Delta d_{B_{\rho}(x)} \geq c
    \end{equation}
    on $B_{\rho_0}(x)\setminus \overline{B_\rho(x)}$.
    \item For any $\rho \in (0, \rho_0)$ and $x \in M$, we have
    \begin{equation}\label{eqn:prelim_isoperimetric}
        \min\{\cH^2(E),\ \cH^2(\partial B_{\rho}(x) \setminus E)\} \leq \beta_0(\cH^1(\partial E))^2
    \end{equation}
    for any set of finite perimeter $E \subset \partial B_{\rho}(x)$.
    \item For any $0 < \rho_1 < \rho_2 \leq \rho_0$ and $x \in M$, we have
    \begin{equation}\label{eqn:prelim_boundary_area}
        \cH^2(\partial B_{\rho_1}(x)) \leq \cH^2(\partial B_{\rho_2}(x)).
    \end{equation}
\end{itemize}

\begin{remark}
By rescaling the metric and prescribing function, we can assume $\rho_0 = 1$. We henceforth make this assumption for ease of notation.
\end{remark}

\subsection{Thin tube isoperimetric inequality}
We make frequent use of the following Lemma from \cite{MSY}.
\begin{lemma}[{\cite[Lemma 1]{MSY}}]\label{lem:thin_isoperimetric_MSY}
There is a constant $\delta \in (0, 1)$ independent of $M$ and $g$ so that if $\Sigma$ is a smooth embedded surface in $M$ and satisfies
\begin{equation}\label{eqn:isoperimetric_smallness_assumption}
    \cH^2(\Sigma \cap B_1(x)) < \delta^2
\end{equation}
for all $x \in M$, then there is a unique compact $K_{\Sigma} \subset M$ with $\partial K_{\Sigma} = \Sigma$ and
\begin{equation}\label{eqn:thin_isoperimetric_smallness_conclusion}
    \cH^3(K_{\Sigma} \cap B_1(x)) \leq \delta^3
\end{equation}
for all $x\in M$, and
\begin{equation}\label{eqn:thin_isoperimetric_MSY}
    \cH^3(K_{\Sigma}) \leq c_1\cH^2(\Sigma)^{3/2},
\end{equation}
where $c_1 = c_1(\mu)$. Moreover, if $\Sigma$ is diffeomorphic to $S^2$, then $K_{\Sigma}$ is diffeomorphic to $\overline{\B}$.\footnote{The exponent of $\delta$ in (\ref{eqn:thin_isoperimetric_smallness_conclusion}) is written as 2 in the statement of \cite[Lemma 1]{MSY}, but a quick examination of the proof reveals that it should be 3.}
\end{lemma}

\subsection{Area comparison}
The following lemma is the appropriate modification of \cite[Lemma 3]{MSY} for our setting. Essentially, we use the divergence theorem to show that moving pieces of a set inside a $c$-mean convex domain decreases the $\cA^c$ functional.

\begin{lemma}\label{lem:area_comparison}
Let $U \subset B_{1/2}(x_0)$ and $\theta \in (0, 1/2)$ satisfy
\begin{itemize}
    \item $x_0 \in U$,
    \item $\overline{U}$ is diffeomorphic to $\overline{\B}$,
    \item $d_U$ is convex on $U(\theta)\setminus \overline{U}$,
    \item $\Delta d_U \geq c$ on $U(\theta)\setminus \overline{U}$.
\end{itemize}
Let $\beta \geq 1$ be a constant so that, for each $s \in (\theta /2, \theta)$,
\begin{equation}\label{eqn:area_comparison_isoperimetric_assumption}
    \min\{\cH^2(E),\ \cH^2(\partial U(s) \setminus E)\} \leq \beta(\cH^1(\partial E))^2
\end{equation}
whenever $E$ is a surface in $\partial U(s)$. Let
\begin{equation}\label{eqn:area_comparison_delta_assumption}
    \delta_1 = \min\left\{\delta,\ \frac{\theta}{(1+128c_1)\beta^{1/2}},\ \frac{\theta}{9\beta^{1/2}(1/32 + cc_1)^{1/2}},\ \frac{\cH^3(M)^{1/3}}{(4c_1)^{1/3}},\ 1\right\},
\end{equation}
(where $c_1$ and $\delta$ are as in Lemma \ref{lem:thin_isoperimetric_MSY}). Let $\Sigma$ be a smooth embedded disk intersecting $\partial U$ transversally with $\partial \Sigma \subset M \setminus U$. Suppose
\begin{equation}\label{eqn:area_comparison_smallness_assumption}
    \cH^2(\partial U) + \cH^2(\Sigma\setminus U) \leq \delta_1^2/32.
\end{equation}

If $\Lambda$ is any component of $\Sigma \setminus U$ with $\partial \Sigma \cap \Lambda = \varnothing$, then there is a unique compact $K_{\Lambda} \subset M \setminus U$ so that
\begin{equation}\label{eqn:area_comparison_smallness_conclusion}
    \cH^3(K_{\Lambda}) \leq c_1\delta_1^3,\ \ \partial K_{\Lambda} = \Lambda \cup F,
\end{equation}
where $F \subset \partial U$ is a surface with $\partial F = \partial \Lambda$ satisfying
\begin{equation}\label{eqn:area_comparison}
    \cH^2(F) < \cH^2(\Lambda \cap U(\theta)) - c\cH^3(K_{\Lambda}).
\end{equation}
\end{lemma}
\begin{proof}
Let $F_0 \subset \partial U$ satisfy $\partial F_0 = \partial \Lambda$, which exists because $\partial U$ is diffeomorphic to a 2-sphere. Then by (\ref{eqn:area_comparison_delta_assumption}) and (\ref{eqn:area_comparison_smallness_assumption}), we have
\[ \cH^2(F_0) + \cH^2(\Lambda) < \delta^2. \]
By Lemma \ref{lem:thin_isoperimetric_MSY}, there is a compact set $W \subset M$ with
\[ \cH^3(W) \leq c_1\delta_1^3,\ \ \partial W = F_0 \cup \Lambda. \]
Let $K_{\Lambda} = W \setminus U$, in which case $F$ is either $F_0$ (if $W \cap U = \varnothing$) or $\partial U \setminus F_0$ (if $U \subset W$). Since $K_{\Lambda}$ must satisfy $\cH^3(K_{\Lambda}) \leq \cH^3(M)/4$ by the definition of $\delta_1$, $K_{\Lambda}$ is unique.

Define $F_t = K_{\Lambda} \cap \partial U(t)$, $E_t = \Lambda \cap U(t)$, and $K_t = K_{\Lambda} \cap U(t)$. We compute
\begin{align}\label{eqn:area_comparison_computation}
    c\cH^3(K_{t_2} \setminus \overline{K}_{t_1})
    & \leq \int_{K_{t_2} \setminus \overline{K}_{t_1}} \Delta d_U\notag\\
    & = \cH^2(F_{t_2}) - \cH^2(F_{t_1}) + \int_{\overline{E}_{t_2} \setminus E_{t_1}} \langle \nu, \nabla d_U\rangle\\
    & \leq \cH^2(F_{t_2}) - \cH^2(F_{t_1}) + \cH^2(\overline{E}_{t_2}) - \cH^2(E_{t_1}).\notag
\end{align}
Equality in (\ref{eqn:area_comparison_computation}) requires $E_{t_1} = E_{t_2}$, because the identity $\nu(y) = \nabla d_U(y)$ for almost every $y \in \overline{E}_{t_2}\setminus E_{t_1}$ implies $\overline{E}_{t_2} \setminus E_{t_1} \subset \partial U(s)$ for some $s$.

By setting $t_1 = 0$ and $t_2 = t > 0$ in (\ref{eqn:area_comparison_computation}), we get
\begin{equation}\label{eqn:area_comparison_application}
    \cH^2(F) < \cH^2(F_t) + \cH^2(\overline{E}_t) - c\cH^3(K_t).
\end{equation}
If $K_{\theta} = K_{\Lambda}$, then we obtain (\ref{eqn:area_comparison}) by setting $t = \theta$. Henceforth, we assume $K_{\theta} \neq K_{\Lambda}$.

By the coarea formula, we have
\[ \int_0^{\theta} \cH^2(F_s)\ ds \leq \cH^3(K_{\theta}) \leq c_1\delta_1^3. \]
Hence,
\[ \cH^2(F_t) \leq 4c_1\theta^{-1}\delta_1^3 \]
for a set of $t \in [0, \theta]$ of measure at least $\frac{3}{4}\theta$. Since $\delta_1 < (128c_1)^{-1}\theta$, we have
\begin{equation}\label{eqn:area_comparison_smallslice1}
    \cH^2(F_t) \leq \delta_1^2/32
\end{equation}
on a set of measure at least $\frac{3}{4}\theta$. Then there is a $t^* \in [\frac{3}{4}\theta, \theta]$ so that (\ref{eqn:area_comparison_smallslice1}) holds at $t^*$. Then (\ref{eqn:area_comparison_computation}) and (\ref{eqn:area_comparison_smallness_assumption}) imply
\begin{equation}\label{eqn:area_comparison_smallslice2}
    \cH^2(F_t) \leq \delta_1^2/32 + \delta_1^2/32 = \delta_1^2/16
\end{equation}
for all $t \leq \frac{3}{4}\theta$.

Since $B_{\theta/2}(x_0) \subset U(\theta/2)$ and $U(\theta) \subset B_1(x_0)$, (\ref{eqn:prelim_dexp_small}) and the isoperimetric inequality in $\R^3$ (since the $\R^3$ isoperimetric constant is $\sim 4.8 \geq 4$) imply
\[ \cH^2(\partial U(t)) \geq \frac{1}{4}\cH^2(\partial(\phi_{x_0}^{-1}(U(t)))) \geq (\cH^3(\phi_{x_0}^{-1}(B_{\theta/2}(x_0))))^{2/3} \geq \theta^2/8 \]
for $t \in [\theta/2, \theta]$. Since $\delta_1\leq \theta$,  (\ref{eqn:area_comparison_smallslice2}) implies
\[ \cH^2(F_t) \leq \frac{1}{2}\cH^2(\partial U(t)) \]
for all $t \in [\theta/2, 3\theta/4]$. Then by (\ref{eqn:area_comparison_isoperimetric_assumption}) and the coarea formula, we have for almost every $t \in [\theta/2, 3\theta/4]$
\[ \cH^2(F_t) \leq \beta(\cH^1(\partial F_t))^2 \leq \beta \left(\frac{d}{dt} \cH^2(E_t)\right)^2. \]
Hence, (\ref{eqn:area_comparison_application}) implies
\[ \cH^2(F) - \cH^2(E_t) + c\cH^3(K_{\Lambda}) \leq \beta\left(\frac{d}{dt}(\cH^2(F)-\cH^2(E_t) + c\cH^3(K_{\Lambda}))\right)^2
\]
for almost every $t \in [\theta/2, 3\theta/4]$. By integration, we have
\begin{align*}
    \sqrt{\cH^2(F) - \cH^2(E_{\theta/2}) + c\cH^3(K_{\Lambda})} - & \sqrt{\cH^2(F) - \cH^2(E_{3\theta/4}) + c\cH^3(K_{\Lambda})}
    \geq \beta^{-1/2}\theta/8
\end{align*}
provided $\cH^2(F) - \cH^2(E_{3\theta/4}) + c\cH^3(K_{\Lambda}) > 0$. By (\ref{eqn:area_comparison_smallness_assumption}), (\ref{eqn:area_comparison_smallness_conclusion}), and (\ref{eqn:area_comparison_delta_assumption}), we have
\begin{align*}
    \sqrt{\cH^2(F) + c\cH^3(K_{\Lambda})}
    & \leq \sqrt{\delta_1^2/32 + cc_1\delta_1^3}\\
    & \leq \delta_1\sqrt{1/32 + cc_1}\\
    & \leq \beta^{-1/2}\theta/9,
\end{align*}
so we reach a contradiction. Hence, we have
\[ \cH^2(F) \leq \cH^2(E_{3\theta/4}) - c\cH^3(K_{\Lambda}) < \cH^2(E_{\theta}) - c\cH^3(K_{\Lambda}), \]
as desired.
\end{proof}

\section{Minimization Problem}
In this section, we rigorously formulate the problem of minimizing the prescribed mean curvature functional over disks in a manifold.

\subsection{Setup}\label{sec:disk_notation}
Let $c_1$ and $\delta$ as in Lemma \ref{lem:thin_isoperimetric_MSY}. As in the statement of Lemma \ref{lem:area_comparison}, we define
\begin{equation}\label{eqn:delta_assumption}
    \delta_2 \coloneqq \min\left\{\delta,\ \frac{\rho_0/2}{(1+128c_1)\beta_0^{1/2}},\ \frac{\rho_0/2}{9\beta_0^{1/2}(1/32 + cc_1)^{1/2}},\ \frac{\cH^3(M)^{1/3}}{(4c_1)^{1/3}},\ 1\right\},
\end{equation}
where $\rho_0$ and $\beta_0$ are defined in \S\ref{sec:local_control}.

Let $U \subset B_{1/2}(x_0)$ and $\theta \in (0, 1/2)$ satisfy
\begin{itemize}
    \item $x_0 \in U$,
    \item $\overline{U}$ is diffeomorphic to $\overline{\B}$ (denoted by $\overline{U}\approx \overline{\B}$),
    \item $d_U$ is convex on $U(\theta)\setminus \overline{U}$,
    \item $\Delta d_U \geq c$ on $U(\theta)\setminus \overline{U}$,
    \item $\cH^2(\partial U) < \delta_2^2/32$.
\end{itemize}
Let $\beta$ and $\delta_1$ be defined in terms of $U$ and $\theta$ as in Lemma \ref{lem:area_comparison}. Note that by (\ref{eqn:prelim_convex}), (\ref{eqn:prelim_mean_convex}), (\ref{eqn:prelim_isoperimetric}), there are valid choices $U$ and $\theta$ so that $\delta_2 \leq \delta_1$ (e.g. we can take $U = B_\rho(x_0)$ for $\rho \ll \rho_0/2$ and $\theta = \rho_0/2$).

We define the following notation.
\begin{itemize}
    \item $\mathcal{D}$ is the set of oriented smooth embedded 2-disks $\Sigma \subset M$ with smooth boundary.
    \item $[\Sigma]$ is the integral $2$-current corresponding to $\Sigma \in \mathcal{D}$.
    \item $\mathcal{D}[\Sigma]$ is the set of disks $\Sigma' \in \mathcal{D}$ with $\partial[\Sigma'] = \partial[\Sigma]$.
    \item $\mathcal{D}(U)$ is the set of disks $\Sigma \in \mathcal{D}$ so that
    \begin{itemize}
        \item $\partial \Sigma \cap U = \varnothing$,
        \item $\Sigma$ has transverse (or empty) intersection with $\partial U$,
        \item there is an open set $\Omega \subset U$ so that $\partial [\Omega] \llcorner U = [\Sigma] \llcorner U$.
    \end{itemize}
    \item $\mathcal{D}^*(U)$ is the set of disks $\Sigma \in \mathcal{D}(U)$ so that
    \begin{itemize}
        \item $\cH^2(\partial U) + \cH^2(\Sigma \setminus U) \leq \delta_1^2/32$,
        \item $\partial \Sigma \setminus \partial U$ is not contained in $K_{\Lambda}$ for any component $\Lambda$ of $\Sigma \setminus U$ with $\partial \Sigma \cap \Lambda = \varnothing$, where $K_{\Lambda}$ is as in Lemma \ref{lem:area_comparison}.
    \end{itemize}
    \item $\mathcal{D}(U)[\Sigma]$ is the set of $\Sigma' \in \mathcal{D}(U) \cap \mathcal{D}[\Sigma]$ satisfying
    \begin{itemize}
        \item $\Sigma' \setminus U \subset \Sigma \setminus U$,
        \item if $x \in \Sigma' \setminus U$, then $\Sigma$ and $\Sigma'$ have the same orientation at $x$.
    \end{itemize}
\end{itemize}

We note that if $\Sigma \in \mathcal{D}^*(U)$ then $\mathcal{D}(U)[\Sigma] \subset \mathcal{D}^*(U)$.

\subsection{Fill-ins}
The prescribed mean curvature functional requires integration over a region. We formulate these regions here.

\begin{proposition}\label{prop:fill-in}
Let $\Sigma \in \mathcal{D}^*(U)$ and $\Sigma_1,\ \Sigma_2 \in \mathcal{D}(U)[\Sigma]$. There is a unique integral 3-current $T(\Sigma_1, \Sigma_2)$ satisfying
\begin{equation}\label{eqn:rel_current}
    \partial T(\Sigma_1, \Sigma_2) = [\Sigma_1] - [\Sigma_2]\ \ \text{and}\ \ \mathrm{spt}(T(\Sigma_1, \Sigma_2)) \subset \overline{U} \cup \bigcup_{\Lambda} K_{\Lambda},
\end{equation}
where $\Lambda$ are the components of $\Sigma \setminus U$ with $\Lambda \cap \partial \Sigma = \varnothing$, and $K_{\Lambda}$ is the corresponding set from Lemma \ref{lem:area_comparison}.
\end{proposition}
\begin{proof}
Let $K \coloneqq \overline{U} \cup \bigcup_{\Lambda} K_{\Lambda}$. Note that there is a 3-current $S$ with $\mathrm{spt}(S) \subset K$ obtained by adding $\pm[K_{\Lambda}]$ for each $\Lambda \subset \Sigma_i$ for some $i$ with the appropriate sign so that
\[ \mathrm{spt}([\Sigma_1] - [\Sigma_2] - \partial S) \subset \overline{U}. \]
Since $\overline{U}$ is diffeomorphic to $\overline{\B}$ and $H_2(\overline{\B}, \Z) = 0$, there is a current $T$ with support in $\overline{\B}$ so that
\[ [\Sigma_1] - [\Sigma_2] = \partial(S + T). \]
We take $T(\Sigma_1, \Sigma_2) := S + T$. The uniqueness follows by the fact there are no nonzero closed 3-currents in $M$ with support in $K$.
\end{proof}

\begin{definition}\label{def:fill-in}
We call $T(\Sigma_1, \Sigma_2)$ from Proposition \ref{prop:fill-in} the \emph{fill-in} of $\Sigma_1$ and $\Sigma_2$. We let $F: I_3(M, \Z) \to L^1(M, \Z)$ denote the canonical isomorphism of normed linear spaces, and we define the \emph{fill-in integral} as
\begin{equation}\label{eqn:fill-in_h}
    T_h(\Sigma_1, \Sigma_2) \coloneqq \int_{M} h\ F(T(\Sigma_1, \Sigma_2))\ d\cH^3.
\end{equation}
\end{definition}

\begin{proposition}\label{prop:fill-in_properties}
Let $\Sigma \in \mathcal{D}^*(U)$. For $\Sigma_1,\ \Sigma_2,\ \Sigma_3 \in \mathcal{D}(U)[\Sigma]$, we have
\begin{align}\label{eqn:fill-in_antisym}
    T(\Sigma_1, \Sigma_2) & = -T(\Sigma_2, \Sigma_1),\\
    T_h(\Sigma_1, \Sigma_2) & = -T_h(\Sigma_2, \Sigma_1),\notag
\end{align}
\begin{align}\label{eqn:fill-in_cyclic}
    T(\Sigma_1, \Sigma_2) + T(\Sigma_2, \Sigma_3) & = T(\Sigma_1, \Sigma_3),\\
    T_h(\Sigma_1, \Sigma_2) + T_h(\Sigma_2, \Sigma_3) & = T_h(\Sigma_1, \Sigma_3).\notag
\end{align}
\end{proposition}
\begin{proof}
We compute
\[ \partial (-T(\Sigma_2, \Sigma_1)) = -[\Sigma_2] + [\Sigma_1]. \]
By the uniqueness of $T(\Sigma_1, \Sigma_2)$, (\ref{eqn:fill-in_antisym}) follows.

We compute
\[ \partial (T(\Sigma_1, \Sigma_2) + T(\Sigma_2, \Sigma_3)) = [\Sigma_1] - [\Sigma_2] + [\Sigma_2] - [\Sigma_3] = [\Sigma_1] - [\Sigma_3]. \]
By the uniqueness of $T(\Sigma_1, \Sigma_3)$, (\ref{eqn:fill-in_cyclic}) follows.
\end{proof}

\subsection{Problem formulation}
In this part of the paper, we consider the following question.
\begin{question}\label{question:main}
If $\{\Sigma_k\}_{k \in \N} \subset \mathcal{D}^*(U)$ satisfies
\begin{equation}\label{eqn:min_seq_setup_disk}
    \cH^2(\Sigma_k) \leq \cH^2(\Sigma_k') + T_h(\Sigma_k, \Sigma_k') + \eps_k
\end{equation}
for all $\Sigma_k' \in \mathcal{D}(U)[\Sigma_k]$ with $\eps_k \to 0$, does (a subsequence of) $\Sigma_k$ converge to a regular limit with prescribed mean curvature $h$?
\end{question}

\section{Limits of Minimizing Sequences}
We apply the standard theory for sets of finite perimeter to deduce some basic preliminary conclusions about the limit in Question \ref{question:main}. In later sections we upgrade these basic observations to stronger regularity results.

\begin{theorem}\label{thm:compactness}
Suppose $\{\Sigma_k\}_{k \in \N} \subset \mathcal{D}^*(U)$ satisfies
\begin{equation}\label{eqn:min_seq_compactness}
    \cH^2(\Sigma_k) \leq \cH^2(\Sigma_k') + T_h(\Sigma_k, \Sigma_k') + \eps_k
\end{equation}
for all $\Sigma_k' \in \mathcal{D}(U)[\Sigma_k]$ with $\eps_k \to 0$. Let $\Omega_k \subset U$ be the open set satisfying
\[ \partial [\Omega_k] \llcorner U =  [\Sigma_k] \llcorner U. \]

If $\sup_k \cH^2(\Sigma_k \cap U) < \infty$, then there is a subsequence (not relabeled), a set of finite perimeter $\Omega \subset U$, and a varifold $V$ in $U$ so that
\[ \mathbf{1}_{\Omega_k} \xrightarrow{L^1} \mathbf{1}_{\Omega},\ \ D\mathbf{1}_{\Omega_k} \wto D\mathbf{1}_{\Omega},\ \ \mathbf{v}(\Sigma_k \cap U) \rightharpoonup V. \]

Furthermore,
\begin{itemize}
    \item $\Omega$ and $V$ satisfy
    \begin{equation}\label{eqn:variational}
        \|V\|(U) - \int_{\Omega} h\ d\cH^3 \leq \|\phi_{\#}V\|(U) - \int_{\phi(\Omega)} h\ d\cH^3
    \end{equation}
    for any diffeomorphism $\phi$ of $U$ equal to the identity outside a compact subset of $U$.
    \item the first variation of $V$ satisfies
    \begin{equation}\label{eqn:first_variation_formula}
        \delta V(X) = \int_{\partial^*\Omega} X \cdot h\nu_{\partial^*\Omega}\ d|D\mathbf{1}_{\Omega}|
    \end{equation}
    for any $C^1$ vector field $X$ with $\mathrm{spt}(X) \subset\subset U$.
    \item $V$ has $c$-bounded first variation, in the sense that
    \begin{equation}\label{eqn:bounded_first_variation}
        \|\delta V\|(B_r(x)) \leq c\|V\|(B_r(x))
    \end{equation}
    for all $B_r(x) \subset U$.
    \item The function
    \begin{equation}\label{eqn:monotonicity}
        r \mapsto f(r)\frac{\|V\|(B_r(x))}{\pi r^2}
    \end{equation}
    is monotone increasing in $r$, where $f$ is a continuous function depending only on $c$ and $(M,g)$ satisfying $f(r) \to 1$ as $r \to 0$.
\end{itemize}
\end{theorem}
\begin{proof}
The subsequential convergence follows by the standard compactness theorems (see \cite[Theorem 12.26]{maggi} and \cite[Chapter 1, Theorem 4.4]{simon_old}).

Let $\phi$ be a diffeomorphism of $U$ equal to the identity outside a compact subset. The assumption of transverse intersection implies
\begin{enumerate}
    \item $\lim_k \cH^2(\Sigma_k \cap U) = \|V\|(U)$,
    \item $\lim_k \int_{\Omega_k} h\ d\cH^3 = \int_{\Omega} h\ d\cH^3$,
    \item $\lim_k \cH^2(\phi(\Sigma_k) \cap U) = \|\phi_{\#}V\|(U)$,
    \item $\lim_k \int_{\phi(\Omega_k)} h\ d\cH^3 = \int_{\phi(\Omega)}h\ d\cH^3$.
\end{enumerate}
Let $\delta > 0$. Then there exists $k$ sufficiently large so that
\begin{align*}
    \|V\|(U) - \int_{\Omega} h\ d\cH^3
    & \leq \cH^2(\Sigma_k \cap U) - \int_{\Omega_k} h\ d\cH^3 + 2\delta & (1) + (2)\\
    & \leq \cH^2(\phi(\Sigma_k) \cap U) - \int_{\phi(\Omega_k)} h\ d\cH^3 + 3\delta & (\ref{eqn:min_seq_compactness})\\
    & \leq \|\phi_{\#}V\|(U) - \int_{\phi(\Omega)} h\ d\cH^3 + 5\delta & (3) + (4).
\end{align*}
Since $\delta$ is arbitrary, (\ref{eqn:variational}) follows.

For a $C^1$ vector field $X$ with $\mathrm{spt}(X) \subset \subset U$, (\ref{eqn:variational}) and \cite[Proposition 17.8]{maggi} imply
\begin{align*}
    \delta V(X)
    & = \frac{d}{dt}\Big|_{t=0} \|(\phi_t)_{\#}V\|(U)\\
    & = \frac{d}{dt}\Big|_{t=0} \int_{\phi_t(\Omega)} h\ d\cH^3\\
    & = \int_{\partial^* \Omega} X\cdot h\nu_{\partial^* \Omega} \ d|D\mathbf{1}_{\Omega}|,
\end{align*}
where $\phi_t$ is the flow of $X$. Hence, (\ref{eqn:first_variation_formula}) holds.

(\ref{eqn:bounded_first_variation}) follows from (\ref{eqn:first_variation_formula}) and \cite[Proposition 4.30]{maggi}.

Finally, the monotonicity of (\ref{eqn:monotonicity}) follows from (\ref{eqn:bounded_first_variation}) by \cite[Chapter 8, Section 40]{simon_old}\footnote{We note that \cite{simon_old} only finds a monotonicity formula for varifolds in $\R^N$. We can isometrically embed $(M, g)$ in some $\R^N$ and use this monotonicity formula, in which case the function $f$ will depend on the norm of the second fundamental form of the embedding.}.
\end{proof}

\section{Resolution of Overlaps}\label{sec:resolution}

We prove the appropriate modification of \cite[Lemma 2 and Corollary 1]{AS} in our setting, to disentangle overlapping disks. The key difference is that we must keep track of volumes in addition to area.

\subsection{Lipschitz disks}
As in \cite{AS}, we define a specific class of Lipschitz disks. Fix an open set $U \subset M$.

\begin{definition}\label{def:Lip_disks}
Consider $n$ disjoint copies $\{F_i\}_{i=1}^n$ of $\D$ and define $\cY$ to be the collection of Lipschitz maps $\chi:\bigsqcup_{i=1}^n F_i\to U$ such that:
\begin{enumerate}
	\item there are finitely many (or zero) pairwise disjoint $C^2$ Jordan curves $\Gamma_1,\dots,\Gamma_m\subset U$ such that for each $k=1,\dots, m$, \[\chi^{-1}(\Gamma_k)=\gamma_k^1\cup\gamma_k^2,\] 
	where $\{\gamma_k^l: k=1,\dots,m;\ l=1,2\}$ is a collection of $2m$ pairwise disjoint $C^2$ Jordan curves in $\bigsqcup_{i=1}^n (F_i\setminus\partial F_i)$, and $\chi\vert_{\gamma_k^l}$ is a $C^2$ parametrization of $\Gamma_k$.
	\item the restriction of $\chi$ to $\left(\bigsqcup_{i=1}^n F_i\right)\setminus\bigcup_{k=1}^m(\gamma_k^1\cup\gamma_k^2)$ is a $C^2$ embedding into $U$.
	\item for each $k=1,\dots, m$, there exist disjoint tubular neighborhoods $W_k^1$, $W_k^2$ of $\gamma_k^1$, $\gamma_k^2$ respectively, such that $\chi(W_k^1)\cap\chi(W_k^2)=\Gamma_k$ and exactly one of the following holds: 
	\begin{enumerate}[label=(3\alph*)]
		\item  $\chi(W_k^1)$ lies on one side of $\chi(W_k^2)$,
		\item  $\chi(W_k^1)$ intersects $\chi(W_k^2)$ transversely along $\Gamma_k$.
	\end{enumerate}
\end{enumerate}
For future convenience, let us introduce the following notation. For each $k=1,\dots, m$ and $l=1,2$, let $W_k^{l,+}=(W_k^l\cap\intr(\gamma_k^l))\setminus\gamma_k^l$ and $W_k^{l,-}=W_k^l\setminus(W_k^{l,+}\cup\gamma_k^l)$, so that $W_k^l=W_k^{l,+}\cup\gamma_k^l\cup W_k^{l,-}$.

Finally, define $\cY_0$ to be the collection of Lipschitz maps $\chi\in\cY$ such that only (3a) occurs for all $k=1,\dots,m$. These are the ``disentangled'' disks, so our goal will be to change a map in $\cY$ into a map in $\cY_0$.
\end{definition}

\begin{remark}\label{rmk:submanifold}
As in \cite{AS}, notice that $\chi\left(\bigsqcup_{i=1}^n F_i\setminus\partial F_i\right)\setminus\bigcup_{i=1}^m\Gamma_i$ is a 2-dimensional open submanifold of $M$, each component of which has closure of the form $\chi(\overline{R})$, where $R$ is a component of $\left(\bigsqcup_{i=1}^n F_i\right)\setminus\bigcup_{k=1}^m(\gamma_k^1\cup\gamma_k^2)$.
\end{remark}

We will say a finite collection $\mathcal{S}$ of Lipschitz disks is in $\cY$ (or $\cY_0$), and write $\cS\in\cY$ (or $\cS\in\cY_0$), if it is the image of some map $\chi\in\cY$ (or $\cY_0$, respectively).

\subsection{Resolution procedure}
We develop a procedure to disentangle intersecting disks by swapping pieces without changing the image of the union of the disks.

\begin{lemma}[Resolution of Overlaps]\label{lem:resolution}
Let $U\subset M$ be a bounded connected open set with $C^1$ boundary. Let $\cS = \{D_i\}_{i=1}^n \in \cY$. Then there are Lipschitz disks $\tilde{\cS} = \{\tilde{D}_i\}_{i=1}^n \in \cY_0$ with $\tilde{D}_i \subset U$ and $\partial \tilde{D}_i = \partial D_i$ satisfying
\[ \bigcup_{i=1}^n D_i = \bigcup_{i=1}^n \tilde{D}_i. \]
\end{lemma}
\begin{proof}
We use notation as in Definition \ref{def:Lip_disks}. 

For each intersection curve $\Gamma_k$, exactly one of (3a) and (3b) in Definition \ref{def:Lip_disks} holds. Let 
\[\cB(\cS)=\{1\leq k\leq m\mid \text{(3b) occurs}\}.\]

If $\cB(\cS)=\varnothing$, then $\cS\in\cY_0$ and we are done. Hence, it suffices by induction to furnish a procedure which replaces the collection $\mathcal{S}$ of disks with another collection $\tilde{\mathcal{S}}\in\cY$ such that:
\begin{itemize}
	\item $\bigcup_i D_i = \bigcup_{i} \tilde{D}_i$,
	\item $\partial \tilde{D}_i = \partial D_i$,
	\item $|\cB(\tilde{\cS})|<|\cB(\cS)|$.
\end{itemize}

Let $q \in \cB(\cS)$. There are (without loss of generality) two cases to consider.

\emph{Case 1}: $\text{int}\gamma_q^1\cap \text{int} \gamma_q^2 = \varnothing$. Fix a $C^2$ diffeomorphism
\[ \psi_q: \overline{\text{int}\gamma_q^1} \to \overline{\text{int}\gamma_q^2} \]
so that $(\chi \circ \psi_q)\vert_{\gamma_q^1} = \chi\vert_{\gamma_q^2}$.
Then we define a bijection $\Psi_q : \bigsqcup_i F_i \to \bigsqcup_i F_i$ by
\[ \Psi_q(\xi) \coloneqq \begin{cases}
\psi_q(\xi) & \xi \in \overline{\text{int}\gamma_q^1}\\
\psi_q^{-1}(\xi) & \xi \in \overline{\text{int}\gamma_q^2}\\
\xi & \text{otherwise}.
\end{cases} \]

\emph{Case 2}: $\gamma_q^2 \subset \text{int}(\gamma_q^1)$. Fix a $C^2$ diffeomorphism
\[ \psi_q: \overline{\text{int}\gamma_q^1} \setminus \text{int}\gamma_q^2 \to \overline{\text{int}\gamma_q^1} \setminus \text{int}\gamma_q^2 \]
so that $(\chi \circ \psi_q)\vert_{\gamma_q^l} = \chi\vert_{\gamma_q^{l+1}}$ (where the superscripts are taken mod 2 for convenience). Then we define a bijection $\Psi_q : \bigsqcup_i F_i \to \bigsqcup_i F_i$ by
\[ \Psi_q(\xi) \coloneqq \begin{cases}
\psi_q(\xi) & \xi \in \overline{\text{int}\gamma_q^1} \setminus \text{int}\gamma_q^2\\
\xi & \text{otherwise}.
\end{cases} \]

We define $\tilde{\chi} \coloneqq \chi \circ \Psi_q : \bigsqcup_i F_i \to M$. By the boundary assumption, in either case, $\tilde{\chi}$ is a Lipschitz map. Since $\Psi_q$ is a bijection, $\tilde{\chi}$ has the same image as $\chi$. Moreover, $\tilde{\chi}$ agrees with $\chi$ in a neighborhood of the boundaries of the $F_i$. Note that the restriction
\[ \Psi_q: \left(\bigsqcup_i F_i\right) \setminus (\gamma_q^1 \cup \gamma_q^2) \to \left(\bigsqcup_i F_i\right) \setminus (\gamma_q^1 \cup \gamma_q^2) \]
is a $C^2$ diffeomorphism. We check each component of Definition \ref{def:Lip_disks}:
\begin{enumerate}
    \item Since $\Psi_q$ is a bijection, we have $\tilde{\chi}^{-1}(\Gamma_k) = \tilde{\gamma}_k^1 \cup \tilde{\gamma}_k^2$, where $\tilde{\gamma}_k^l \coloneqq \Psi_q^{-1}(\gamma_k^l)$ are disjoint sets. Since $\Psi_q$ is a $C^2$ diffeomorphism away from $\gamma_q^1 \cup \gamma_q^2$, $\tilde{\gamma}_k^l$ is a $C^2$ Jordan curve for all $k \neq q$ and $\tilde{\chi}\vert_{\tilde{\gamma}_k^l}$ is a $C^2$ parametrization. Finally, it follows from the construction that $\tilde{\gamma}_q^l = \gamma_q^{l+1}$ and $\tilde{\chi}\vert_{\tilde{\gamma}_q^l}$ is a $C^2$ parametrization.
    \item Since $\Psi_q$ is a $C^2$ diffeomorphism away from $\tilde{\gamma}_q^1 \cup \tilde{\gamma}_q^2$, the restriction of $\tilde{\chi}$ to any component of $(\bigsqcup_i F_i)\setminus (\bigsqcup_k \tilde{\gamma}_k^1 \cup \tilde{\gamma}_k^2)$ is a $C^2$ embedding.
    \item For $k \neq q$, the desired tubular neighborhoods are $\tilde{W}_k^l = \Psi_q^{-1}(W_k^l)$. Since $\tilde{\chi}(\tilde{W}_k^l) = \chi(W_k^l)$, (3a) (resp. (3b)) holds at $k \neq q$ for $\tilde{\chi}$ if and only if (3a) (resp. (3b)) holds at $k$ for $\chi$. We now define $\tilde{W}_q^l$, which differ in the above cases.
    \begin{itemize}
        \item In Case 1, we have
        \[ \tilde{W}_q^l = \Psi_q^{-1}(W_q^{l,-} \cup \gamma_q^{l+1} \cup W_q^{l+1,+}). \]
        Then
        \[ \tilde{\chi}(\tilde{W}_q^l) = \chi(W_q^{l,-}) \cup \Gamma_q \cup \chi(W_q^{l+1,+}) \]
        \item In Case 2, we have
        \begin{align*}
            \tilde{W}_q^1 & = \Psi_q^{-1}(W_q^{1,-} \cup \gamma_q^2 \cup W_q^{2,-})\\
            \tilde{W}_q^2 & = \Psi_q^{-1}(W_q^{2,+} \cup \gamma_q^1 \cup W_q^{1,+}).
        \end{align*}
        Then
        \begin{align*}
            \tilde{\chi}(\tilde{W}_q^1) & = \chi(W_q^{1,-}) \cup \Gamma_q \cup \chi(W_q^{2,-})\\
            \tilde{\chi}(\tilde{W}_q^2) & = \chi(W_q^{2,+}) \cup \Gamma_q \cup \chi(W_q^{1,+}).
        \end{align*}
    \end{itemize}
    Since (3b) holds at $q$ for $\chi$, we see that (3a) holds at $q$ for $\tilde{\chi}$.
\end{enumerate}
Hence, $\tilde{\chi} \in \cY$. Moreover, it follows from (3) above that $|\cB(\tilde{\cS})| < |\cB(\cS)|$.
\end{proof}

\begin{remark}
The procedure described in the proof of Lemma \ref{lem:resolution} is not uniquely determined by the initial collection $\cS$. However, we will see that this is not an issue in our applications. 
\end{remark}

After applying the resolution procedure, we will make small perturbations to produce genuinely embedded surfaces. We collect the required statements in the following lemma.

\begin{lemma}\label{lem:y0}
Let $U \subset M$ be an open set with $\overline{U} \approx  \overline{\B}$. Let $\Sigma \in \mathcal{D}(U)$. Let $\{D_i\}_{i=1}^n$ be a set of smooth embedded disks with $D_i \subset U$, $\partial D_i \subset \partial U$, and 
\[ \Sigma \cap U = \bigcup_{i=1}^{n} D_i. \]
Let $\{\tilde{D}_i\} \in \cY_0$ satisfy $\tilde{D}_i \subset U$ and $\partial [\tilde{D}_i] = \partial [D_i]$. Then there is an open set $\tilde{\Omega} \subset U$ satisfying
\begin{equation}\label{eqn:y0_conclusion}
\partial [\tilde{\Omega}]\llcorner U = \sum_{i=1}^n [\tilde{D}_i].
\end{equation}
Moreover, for any $\eps > 0$ there is a $\Sigma' \in \mathcal{D}(U)[\Sigma]$ satisfying
\begin{equation}\label{eqn:y0_subset}
    \Sigma \triangle \Sigma' \subset U,
\end{equation}
\begin{equation}\label{eqn:y0_area}
    \cH^2(\Sigma' \cap U) \leq \sum_{i=1}^n \cH^2(\tilde{D}_i) + \eps,
\end{equation}
and
\begin{equation}\label{eqn:y0_vol}
    \mathbb{M}(T(\Sigma, \Sigma') - ([\Omega] - [\tilde{\Omega}])) \leq \eps.
\end{equation}
\end{lemma}
\begin{proof}
The existence of $\tilde{\Omega}$ follows from the existence of $\Omega$ and the fact that the boundary orientations of the disks agree.

The remaining conclusions follow from making small perturbations along the intersection curves of the disks $\{\tilde{D}_i\}$.
\end{proof}

\subsection{Application 1: Collapsing to the boundary}
As a first application of the resolution procedure, we would like to ignore some components of a minimizing sequence in a ball by collapsing those components to the boundary of the ball.

\begin{lemma}\label{lem:resolution_iso_collapse}
Let $U \subset M$ be an open set with $\overline{U} \approx  \overline{\B}$. Let $\Sigma \in \mathcal{D}(U)$. Let $\mathcal{S}$ be a set of smooth embedded disks $D$ with $D \subset U$, $\partial D \subset \partial U$, and
\[ \Sigma \cap U = \bigcup_{D \in \mathcal{S}} D. \]
Let $\mathcal{S}_0 \subset \mathcal{S}$. For $D \in \mathcal{S}_0$, let $F_D \subset \partial U$ be the isoperimetric region for $\partial D$, and let $\Lambda_D$ be the region in $U$ bounded by $D \cup F_D$. For any $\eps > 0$, there is a $\Sigma' \in \mathcal{D}(U)[\Sigma]$ satisfying
\begin{equation}\label{eqn:resolution_iso_collapse_subset}
    \Sigma \triangle \Sigma' \subset U,
\end{equation}
\begin{equation}\label{eqn:resolution_iso_collapse_area}
    \cH^2(\Sigma' \cap U) \leq \sum_{D \in \mathcal{S}_0} \cH^2(F_D) + \sum_{D \in \mathcal{S} \setminus \mathcal{S}_0} \cH^2(D) + \eps,
\end{equation}
and
\begin{equation}\label{eqn:resolution_iso_collapse_vol}
    \mathbb{M}(T(\Sigma, \Sigma')) \leq \sum_{D \in \mathcal{S}_0} \cH^3(\Lambda_D) + \eps.
\end{equation}
\end{lemma}
\begin{proof}
To apply Lemma \ref{lem:resolution}, we must first perturb the disks $F_D$. Let $\eta > 0$ so that $\cH^3(B_{\eta}(\partial U)) < \eps/2$. We define disks $F_D'$ for $D \in \cS_0$ so that the following hold:
\begin{itemize}
    \item $F_{D_1}' \cap F_{D_2}' = \varnothing$ for $D_1 \neq D_2 \in \cS_0$,
    \item $F_{D_1}$ has transverse (or empty) intersection with $D_2$ for any $D_1 \in \cS_0$ and $D_2 \in \cS \setminus \cS_0$,
    \item $F_D' \subset U$ and $\partial F_D' = \partial D$,
    \item $|\cH^2(F_D) - \cH^2(F_D')| < \eps/(2|\cS_0|)$,
    \item $F_D' \subset B_{\eta}(\partial U) \cap \Lambda_D$.
\end{itemize}
We can then apply Lemma \ref{lem:resolution} to the disks
\[ \cS' \coloneqq (\cS \setminus \cS_0) \cup \{F_D'\}_{D \in \cS_0}, \]
producing a collection $\tilde{\cS}$.

By Lemma \ref{lem:y0} and the choice of $\eta$, it then suffices show (up to a set of $\cH^3$ measure zero)
\[ (\Omega \triangle \tilde{\Omega}) \setminus B_{\eta}(\partial U) \subset \bigcup_{D \in \cS_0} \Lambda_D. \]

Let $x \in (\Omega \triangle \tilde{\Omega}) \setminus B_{\eta}(\partial U)$. Take a point
\[ x_0 \in \partial U \setminus \left(\bigcup_{D \in \cS_0} F_D \cup \bigcup_{D \in \cS} \partial D\right). \]
We also assume $x$ is not in any $D \in \cS$ or $\tilde{D} \in \tilde{\cS}$, which is a set of $\cH^3$ measure zero. Let $\gamma : (0, 1] \to U$ be a smooth path having transverse intersection with any $D \in \cS$ or $\tilde{D} \in \tilde{\cS}$, satisfying $\lim_{t \to 0} \gamma(t) = x_0$ and $\gamma(1) = x$. Since $x_0$ and $x$ are in the complement of $B_{\eta}(\partial U) \cap (\bigcup_{D \in \cS_0} \overline{\Lambda}_D)$, which contains the disks $F_D'$, we conclude that $\gamma$ intersects each $F_D'$ for $D \in \cS_0$ an even number of times. Since $x \in \Omega \triangle \tilde{\Omega}$ and $x_0 \notin \overline{\Omega \triangle \tilde{\Omega}}$, we conclude that $\gamma$ intersects some $D \in \cS_0$ an odd number of times. Since $x_0 \notin \overline{\Lambda}_D$, we have $x \in \Lambda_D$ for some $D \in \cS_0$.
\end{proof}

\subsection{Application 2: Isolating one stacked disk}\label{subsec:isolate_one}
Our goal in this application is to isolate individual disks in a stacked disk minimization problem to deduce some (albeit weaker) regularity for each sheet of the limit varifold.

We first establish some notation.

Let $U \approx  \B$ satisfy $\partial U = \overline{C} \cup \overline{D^+} \cup \overline{D^-}$, where $C$ is a smoothly embedded annulus and $D^+$ and $D^-$ are smoothly embedded disks. Let $\Sigma \in \mathcal{D}(U)$, in the sense that $\Sigma \cap \partial C = \varnothing$. Moreover, assume $\Sigma \cap U$ is a disjoint union of smooth embedded oriented disks $\{D_1, \hdots, D_n\}$, each of which has homotopically nontrivial boundary in $C$.

For $i \in \{1, \hdots, n\}$, let $B_i \subset U$ be the open subset satisfying
\[ D^- \subset \partial B_i \setminus D_i \subset \partial U. \]
If $B_i \subset B_j$ for $i \neq j$, we say $D_i < D_j$. Without loss of generality, we suppose
\[ D_1 > D_2 > \hdots > D_n. \]

Let $\Omega_i \subset U$ be the open subset satisfying
\[ \partial [\Omega_i]\llcorner U =  [D_i] \llcorner U. \]
If $B_i = \Omega_i$, we say $D_i$ is oriented up. Otherwise, we say $D_i$ is oriented down. By the ordering of the $D_i$ and the fact that $\Sigma \in \mathcal{D}(U)$, $D_i$ is oriented up if and only if $D_{i-1}$ and $D_{i+1}$ are oriented down (i.e.\ the orientation alternates with respect to the ordering).

\begin{lemma}\label{lem:resolution_stacked_one}
Fix some $1 \leq j \leq n$. Let $D_j' \in \mathcal{D}(U)[D_j]$ so that $D_j'$ has transverse intersections with each $D_i$ for $i \neq j$. For any $\eps > 0$, there is a $\Sigma' \in \mathcal{D}(U)[\Sigma]$ satisfying
\begin{equation}\label{eqn:resolution_stacked_one_subset}
    \Sigma \triangle \Sigma' \subset U,
\end{equation}
\begin{equation}\label{eqn:resolution_stacked_one_area}
    \cH^2(\Sigma' \cap U) \leq \cH^2(D_j') + \sum_{i\neq j} \cH^2(D_i) + \eps,
\end{equation}
and
\begin{equation}\label{eqn:resolution_stacked_one_vol1}
    \mathbb{M}(T(\Sigma, \Sigma') - [\Omega^+] + [\Omega^-]) \leq \eps,
\end{equation}
where $\Omega^+$ and $\Omega^-$ are disjoint sets contained in $\Omega_j \triangle \Omega_j'$.

Moreover, instead of (\ref{eqn:resolution_stacked_one_vol1}), we can construct $\Sigma'$ satisfying
\begin{equation}\label{eqn:resolution_stacked_one_vol2}
    \mathbb{M}(T(\Sigma, \Sigma') - ([\Omega_j] - [\Omega_j']) - [(\Omega_j' \setminus \Omega_j) \cap \Omega] +G) \leq \eps.
\end{equation}
where $G$ is a nonnegative integral current, in the sense that $\int F(G)f\ d\cH^3\geq 0$ for all nonnegative $f$.
\end{lemma}
\begin{proof}
\emph{Step 1}: We apply Lemma \ref{lem:resolution} to the disks $\cS' \coloneqq \{D_j'\} \cup \{D_i\}_{i\neq j}$, producing new disks $\tilde{\cS} = \{\tilde{D}_i\}_{i=1}^n$.

We claim
\begin{equation}\label{eqn:resolution_stacked_one_intermediate}
    [\Omega] - [\tilde{\Omega}] = [\Omega_j] - [\Omega_j'] + 2[(\Omega_j' \setminus \Omega_j) \cap \Omega] - 2[(\Omega_j \setminus \Omega_j') \setminus \Omega],
\end{equation}
which implies (\ref{eqn:resolution_stacked_one_vol1}) by Lemma \ref{lem:y0} and taking
\begin{align*}
    \Omega^+ := (\Omega_j \triangle \Omega_j') \cap \Omega,\ \
    \Omega^- := (\Omega_j \triangle \Omega_j') \setminus \Omega.
\end{align*}
We first observe
\begin{align*}
    [\Omega] - [\tilde{\Omega}]
    & = \sum_{i=1}^n [\Omega_i] - \sum_{i=1}^n [\tilde{\Omega}_i] = ([\Omega_j] - [\Omega_j']) + \left([\Omega_j'] + \sum_{i\neq j} [\Omega_i] - \sum_{i=1}^n [\tilde{\Omega}_i]\right).
\end{align*}
Let
\[
    \tilde{N} \coloneqq \sum_{i=1}^n \mathbf{1}_{\tilde{\Omega}_i}, \ \
    N' \coloneqq \mathbf{1}_{\Omega_j'} + \sum_{i\neq j} \mathbf{1}_{\Omega_i}.
\]
To show (\ref{eqn:resolution_stacked_one_intermediate}), it suffices to show for $\cH^3$-a.e. $x \in U$ that
\begin{equation}\label{eqn:resolution_stacked_one_cases}
    N'(x) - \tilde{N}(x) = \begin{cases}
    2 & x \in (\Omega_j' \setminus \Omega_j) \cap \Omega\\
    -2 & x \in (\Omega_j \setminus \Omega_j') \setminus \Omega\\
    0 & \text{otherwise}.
    \end{cases}
\end{equation}
Let $R$ be the component of $U \setminus \bigcup_i D_i$ containing $x$. Let $x_0 \in C \cap \overline{R}$. Let $\gamma : (0, 1] \to U$ be a path in $R$ so that
\begin{itemize}
    \item $\gamma$ avoids the intersection curves $\{\Gamma_k\}$,
    \item $\gamma$ has transverse intersection with $D_j'$,
    \item $\lim_{t \to 0} \gamma(t) = x_0$ and $\gamma(1) = x$.
\end{itemize}
We have $N'(x_0) = \tilde{N}(x_0)$. Along subsequent intersections of $\gamma$ with $D_j'$, each of $N'$ and $\tilde{N}$ will increase and decrease by 1 in alternating order.

If $\gamma$ has an even number of intersections with $D_j'$, then we see $N'(x) = \tilde{N}(x)$. Hence, we only need to consider the case of an odd number of intersections, which is equivalent to $x \in \Omega_j \triangle \Omega_j'$. In this case, $N'$ increases at the first intersection if and only if $x \in \Omega_j'$. Note that $\tilde{N}$ decreases at the first intersection if and only if $R \subset \Omega$. Hence, (\ref{eqn:resolution_stacked_one_cases}) follows.

\emph{Step 2}: We further modify the disks $\tilde{D}_i$ to $\tilde{D}_i'$ so that (up to a current with arbitrarily small mass)
\begin{equation}\label{eqn:step2}
    [\tilde{\Omega}] - [\tilde{\Omega}'] = -[(\Omega_j'\setminus \Omega_j) \cap \Omega] - G,
\end{equation}
where $G$ is a nonnegative integral current, which implies (\ref{eqn:resolution_stacked_one_vol2}) by (\ref{eqn:resolution_stacked_one_intermediate}) and Lemma \ref{lem:y0}.

Let $R$ be a component of $(\Omega_j' \setminus \Omega_j) \cap \Omega$. We collect the following facts.
\begin{enumerate}
    \item $R \cap \tilde{\Omega} = \varnothing$: a direct consequence of (\ref{eqn:resolution_stacked_one_intermediate}).
    \item $\partial R \cap \partial U = \varnothing$: follows from $\partial \Omega_j \cap \partial U = \partial \Omega_j' \cap \partial U$.
    \item $R$ is a component of $U \setminus \bigcup_i \tilde{D}_i$: Let $x,\ y \in R$. By definition, there is a path $\alpha$ from $x$ to $y$ contained in both $\Omega$ and in $\Omega_j'$, so $\alpha$ does not intersect any of the disks $\{D_i\}$ nor $D_j'$. Moreover, since $D_j \cap \overline{\Omega \setminus \Omega_j} = \varnothing$, we have $\partial R \subset \bigcup_i \tilde{D}_i$.
\end{enumerate}

Let $\tilde{\chi}: \bigcup_{i=1}^n F_i \to M$ be the Lipschitz map corresponding to the resolved disks $\tilde{\cS}$. Let $\{\sigma_l\}_l \subset \{\Gamma_k\}_k$ be the components of $\partial (\tilde{\chi}^{-1}(\partial R))$. Choose some $\sigma \in \{\sigma_l\}$ so that $\sigma \not\subset \text{int}(\sigma_l)$ for any $l$. By Definition \ref{def:Lip_disks}, there is a unique distinct curve $\sigma' \in \{\Gamma_k\}_k$ so that $\tilde{\chi}(\sigma) = \tilde{\chi}(\sigma')$. There are two cases.

\emph{Case 1}: $\text{int}\sigma \cap \text{int}\sigma' = \varnothing$. Since $d_1\coloneqq \tilde{\chi}(\overline{\text{int}\sigma})$ and $d_2\coloneqq \tilde{\chi}(\overline{\text{int}\sigma'})$ are Lipschitz disks with a common boundary and that do not cross through each other, there is a unique open set $R' \subset U$ with $\partial R' = d_1 \cup d_2$.

We claim that the disks $\{\tilde{D}_i\}$ do not intersect $R'$. Suppose otherwise for contradiction. Then for some $i$ there is an $x \in F_i \setminus (\overline{\text{int}\sigma \cup \text{int}\sigma'})$ satisfying $\tilde{\chi}(x) \in R'$. Let $\alpha$ be a path in $F_i \setminus (\overline{\text{int}\sigma \cup \text{int}\sigma'})$ from $x$ to $\partial F_i$. Then we have $\tilde{\chi}(\alpha(0)) \in R'$ and $\tilde{\chi}(\alpha(1)) \notin R'$. Since $\partial R' = d_1 \cup d_2$, we conclude that $\tilde{D}_i$ crosses through one of $d_1$ or $d_2$, which contradicts $\tilde{\chi} \in \cY_0$.

We claim that $R \subset R'$. We consider the local picture along $\tilde{\chi}(\sigma)$, where the 4 pieces $\tilde{\chi}(\text{int}(\sigma))$, $\tilde{\chi}(\text{ext}(\sigma))$, $\tilde{\chi}(\text{int}(\sigma'))$, and $\tilde{\chi}(\text{ext}(\sigma'))$ intersect with an X cross-section. We have
\begin{itemize}
    \item $\tilde{\chi}(\text{int}(\sigma))$ and  $\tilde{\chi}(\text{ext}(\sigma))$ are adjacent: follows from $\tilde{\chi} \in \cY_0$.
    \item $\tilde{\chi}(\text{int}(\sigma))$ and  $\tilde{\chi}(\text{int}(\sigma'))$ are adjacent: follows from the definition of $R'$ and the fact that $\{\tilde{D}_i\}$ do not intersect $R'$.
\end{itemize}
$R'$ contains the component of $U \setminus \bigcup_i \tilde{D}_i$ contained between $\tilde{\chi}(\text{int}(\sigma))$ and  $\tilde{\chi}(\text{int}(\sigma'))$ along this intersection. Since $R$ touches $\tilde{\chi}(\text{int}(\sigma))$ and does not touch $\tilde{\chi}(\text{ext}(\sigma))$, fact (3) above implies that $R$ is the component of $U \setminus \bigcup_i \tilde{D}_i$ contained between $\tilde{\chi}(\text{int}(\sigma))$ and  $\tilde{\chi}(\text{int}(\sigma'))$ along this intersection. Hence, $R \subset R'$.

We claim that $R' \cap \tilde{\Omega} = \varnothing$. Since $\tilde{\chi} \in \cY_0$, each component of $R'$ touches the same side of each of $d_1$ and $d_2$. By the definition of $\tilde{\Omega}$, we have either $R' \subset \tilde{\Omega}$ or $R' \cap \tilde{\Omega} = \varnothing$. By fact (1) above, we have $R \cap \tilde{\Omega} = \varnothing$, so $R' \cap \tilde{\Omega}^c \neq \varnothing$. Hence, $R' \cap \tilde{\Omega} = \varnothing$.

Collapse procedure: Relabel $\sigma$ and $\sigma'$ as $\sigma_1$ and $\sigma_2$ so that 
\[ \cH^2(\tilde{\chi}(\text{int}\sigma_1)) \geq \cH^2(\tilde{\chi}(\text{int}\sigma_2)). \]
Fix a $C^2$ diffeomorphism
\[ \psi : \overline{\text{int} \sigma_1} \to \overline{\text{int} \sigma_2} \]
so that $(\tilde{\chi} \circ \psi)\vert_{\sigma_1} = \tilde{\chi}\vert_{\sigma_2}$. Then we define the map
\[ \tilde{\chi}'(\xi) \coloneqq \begin{cases}
\tilde{\chi}(\psi(\xi)) & \xi \in \overline{\text{int}\sigma_1}\\
\tilde{\chi}(\xi) & \text{otherwise}.
\end{cases} \]

We make an arbitrarily small perturbation of the map $\tilde{\chi}'$ on $\text{int}\sigma_1$ so that $\tilde{\chi}' \in \cY_0$.

\emph{Case 2}: $\sigma' \subset \text{int}\sigma$. Since $A\coloneqq \tilde{\chi}(\overline{\text{int}\sigma \setminus \text{int}\sigma'})$ is a Lipschitz annulus with a common boundary and that does not cross through itself, there is a unique open set $R' \subset U$ with $\partial R' = A$.

We claim that the disks $\{\tilde{D}_i\}$ do not intersection $R'$. Suppose otherwise for contradiction. By the argument from Case 1, we need only suppose that $\tilde{\chi}(\text{int}\sigma') \cap R' \neq \varnothing$. Since $\tilde{\chi} \in \cY_0$, we must have $\tilde{\chi}(\text{int}\sigma') \subset R'$ in this case. Yet we observe that $\tilde{\chi}(\sigma')$ is a nontrivial element in $H_1(\overline{R'})$, a contradiction.

We claim that $R \subset R'$. We consider the local picture along $\tilde{\chi}(\sigma)$, where the 4 pieces $\tilde{\chi}(\text{int}(\sigma))$, $\tilde{\chi}(\text{ext}(\sigma))$, $\tilde{\chi}(\text{int}(\sigma'))$, and $\tilde{\chi}(\text{ext}(\sigma'))$ intersect with an X cross-section. We have
\begin{itemize}
    \item $\tilde{\chi}(\text{int}(\sigma))$ and  $\tilde{\chi}(\text{ext}(\sigma))$ are adjacent: follows from $\tilde{\chi} \in \cY_0$.
    \item $\tilde{\chi}(\text{int}(\sigma))$ and  $\tilde{\chi}(\text{ext}(\sigma'))$ are adjacent: follows from the definition of $R'$ and the fact that $\{\tilde{D}_i\}$ do not intersect $R'$.
\end{itemize}
$R'$ contains the component of $U \setminus \bigcup_i \tilde{D}_i$ contained between $\tilde{\chi}(\text{int}(\sigma))$ and  $\tilde{\chi}(\text{ext}(\sigma'))$ along this intersection. Since $R$ touches $\tilde{\chi}(\text{int}(\sigma))$ and does not touch $\tilde{\chi}(\text{ext}(\sigma))$, fact (3) above implies that $R$ is the component of $U \setminus \bigcup_i \tilde{D}_i$ contained between $\tilde{\chi}(\text{int}(\sigma))$ and  $\tilde{\chi}(\text{ext}(\sigma'))$ along this intersection. Hence, $R \subset R'$.

We claim that $R' \cap \tilde{\Omega} = \varnothing$. Since $\tilde{\chi} \in \cY_0$, each component of $R'$ touches the same side of $A$. By the definition of $\tilde{\Omega}$, we have either $R' \subset \tilde{\Omega}$ or $R' \cap \tilde{\Omega} = \varnothing$. By fact (1) above, we have $R \cap \tilde{\Omega} = \varnothing$, so $R' \cap \tilde{\Omega}^c \neq \varnothing$. Hence, $R' \cap \tilde{\Omega} = \varnothing$.

Collapse procedure: Fix a $C^2$ diffeomorphism
\[ \psi: \overline{\text{int}\sigma} \to \overline{\text{int}\sigma'} \]
so that $(\tilde{\chi} \circ \psi)\vert_{\sigma} = \tilde{\chi}\vert_{\sigma'}$. Then we define the map
\[ \tilde{\chi}'(\xi) \coloneqq \begin{cases}
\tilde{\chi}(\psi(\xi)) & \xi \in \overline{\text{int}\sigma}\\
\tilde{\chi}(\xi) & \text{otherwise}.
\end{cases} \]
We make an arbitrarily small perturbation of the map $\tilde{\chi}'$ in a neighborhood of $\sigma$ so that $\tilde{\chi}' \in \cY_0$.

\emph{Conclusion}. In either case, we have
\[ \cH^2(\text{im}\tilde{\chi}') \leq \cH^2(\text{im}\tilde{\chi}) + \eps \]
and
\[ \cH^3((\tilde{\Omega} \cup R')\setminus \tilde{\Omega}') \leq \eps. \]

Let $\tilde{R}$ be a different component of $(\Omega_j'\setminus \Omega_j) \cap \Omega$. Then we claim $\partial\tilde{R} \cap \overline{R}' = \varnothing$. Indeed, if not $\tilde{R}$ would be on the opposite side of a disk in $\tilde{\cS}$ from $R'$. Since $R' \cap \tilde{\Omega} = \varnothing$ and (by fact (1) above) $\tilde{R} \cap \tilde{\Omega} = \varnothing$, we reach a contradiction. Hence, we can iterate the above procedure and conclude by induction.
\end{proof}

\subsection{Application 3: Decomposing stacks}\label{subsec:decompose}
In the case of positive prescribing function, we will decompose stacked disk minimization problems into small stacked disk minimization problems each with 1 or 2 disks.

We first separate stacks consisting of 1 disk.

\begin{lemma}\label{lem:resolution_decompose_one}
Suppose $D_1$ is oriented up. Let $D_1' \in \mathcal{D}(U)[D_1]$ so that $D_1'$ has transverse intersections with each $D_i$. For any $\eps > 0$, there is a $\Sigma' \in \mathcal{D}(U)[\Sigma]$ satisfying
\begin{equation}\label{eqn:resolution_1-stack_subset}
    \Sigma \triangle \Sigma' \subset U,
\end{equation}
\begin{equation}\label{eqn:resolution_1-stack_area}
    \cH^2(\Sigma' \cap U) \leq \cH^2(D_1') + \sum_{i=2}^n \cH^2(D_i) + \eps,
\end{equation}
and
\begin{equation}\label{eqn:resolution_1-stack_vol}
    \mathbb{M}(T(\Sigma, \Sigma') - ([\Omega_1] - [\Omega_1']) + 2[\Omega_1 \setminus (\Omega \cup \Omega_1')]) \leq \eps.
\end{equation}
\end{lemma}
\begin{proof}
We apply Lemma \ref{lem:resolution} to $\{D_1'\} \cup \{D_i\}_{i=2}^n$. The resulting disks satisfy (\ref{eqn:resolution_stacked_one_intermediate}) as in the first step of the proof of Lemma \ref{lem:resolution_stacked_one}. However, we have $[(\Omega_1' \setminus \Omega_1) \cap \Omega] = 0$ because $D_1$ is oriented up. Hence, the result follows from Lemma \ref{lem:y0}.
\end{proof}

Now we separate stacks consisting of 2 disks.

\begin{lemma}\label{lem:resolution_decompose_two}
Suppose $D_j$ is oriented down and $D_{j+1}$ is oriented up. Let $D_j' \in \mathcal{D}(U)[D_j]$ and $D_{j+1}' \in \mathcal{D}(U)[D_{j+1}]$ be mutually disjoint and have transverse intersections with each $D_i$. For any $\eps > 0$, there is a $\Sigma' \in \mathcal{D}(U)[\Sigma]$ satisfying
\begin{equation}\label{eqn:resolution_2-stacks_subset}
    \Sigma \triangle \Sigma' \subset U,
\end{equation}
\begin{equation}\label{eqn:resolution_2-stacks_area}
    \cH^2(\Sigma' \cap U) \leq \cH^2(D_j') + \cH^2(D_{j+1}') + \sum_{i\notin \{j,j+1\}} \cH^2(D_i) + \eps,
\end{equation}
and
\begin{equation}\label{eqn:resolution_2-stacks_vol}
    \mathbb{M}(T(\Sigma, \Sigma') - (([\Omega_j] + [\Omega_{j+1}]) - ([\Omega_j'] + [\Omega_{j+1}'])) + 2[\Omega_j \cup \Omega_{j+1} \setminus (\Omega \cup \Omega_j' \cup \Omega_{j+1}')]) \leq \eps.
\end{equation}
\end{lemma}
\begin{proof}
We apply Lemma \ref{lem:resolution} to the disks $\{D_j', D_{j+1}'\} \cup \{D_i\}_{i\neq j, j+1}$, producing new disks $\{\tilde{D}_i\}_{i=1}^n$.

We compute
\begin{align*}
    [\Omega] - [\tilde{\Omega}]
    & = \sum_{i=1}^n [\Omega_i] - \sum_{i=1}^n [\tilde{\Omega}_i]\\
    & = ([\Omega_j] + [\Omega_{j+1}]) - ([\Omega_j'] - [\Omega_{j+1}'])\\
    & \hspace{0.4cm} + \left([\Omega_j'] + [\Omega_{j+1}'] + \sum_{i \neq j, j+1} [\Omega_i] - \sum_{i=1}^n [\tilde{\Omega}_i]\right).
\end{align*}
Let
\[
    \tilde{N} \coloneqq \sum_{i=1}^n \mathbf{1}_{\tilde{\Omega}_i}, \ \
    N' \coloneqq \mathbf{1}_{\Omega_j'} + \mathbf{1}_{\Omega_{j+1}'} + \sum_{i\neq j, j+1} \mathbf{1}_{\Omega_i}.
\]
To show (\ref{eqn:resolution_2-stacks_vol}), it suffices to show for $x \in U \setminus \bigcup_{i=1}^n \tilde{D}_i$ that
\begin{equation}\label{eqn:resolution_decompose_two_cases}
    N'(x) - \tilde{N}(x) = \begin{cases}
    -2 & x \in (\Omega_j \cup \Omega_{j+1}) \setminus (\Omega \cup \Omega_j' \cup \Omega_{j+1}')\\
    0 & \text{otherwise}.
    \end{cases}
\end{equation}
Let $R$ be the component of $U \setminus \bigcup_i D_i$ containing $x$. Let $x_0 \in C \cap \overline{R}$. Let $\gamma : (0, 1] \to U$ be a path in $R$ so that
\begin{itemize}
    \item $\gamma$ avoids $\Gamma_k$,
    \item $\gamma$ has transverse intersection with $D_j'$ and $D_{j+1}'$,
    \item $\lim_{t \to 0} \gamma(t) = x_0$ and $\gamma(1) = x$.
\end{itemize}
We have $N'(x_0) = \tilde{N}(x_0)$. Along subsequent intersections of $\gamma$ with $D_j' \cup D_{j+1}'$, each of $N'$ and $\tilde{N}$ will increase and decrease by 1 in alternating order.

If $\gamma$ has an even number of intersections with $D_j' \cup D_{j+1}'$, then we see $N'(x) = \tilde{N}(x)$. Hence, we only need to consider the case of an odd number of intersections, which is equivalent to $x \in (\Omega_j \cup \Omega_{j+1}) \triangle (\Omega_j' \cup \Omega_{j+1}')$. In this case, $N'$ increases at the first intersection if and only if $x \in \Omega_j' \cup \Omega_{j+1}'$. Note that $\tilde{N}$ decreases at the first intersection if and only if $R \subset \Omega$. We also note that if $x \notin \Omega_j \cup \Omega_{j+1}$, then $x \notin \Omega$ (by the specification of the orientations of $D_j$ and $D_{j+1}$). Hence, (\ref{eqn:resolution_decompose_two_cases}) follows.

The result now follows from Lemma \ref{lem:y0}.
\end{proof}

\section{Replacement}
We adapt the replacement theorem \cite[Theorem 1]{AS} to the prescribed mean curvature setting. This lemma allows us to bring components of a minimizing surface outside a small ball into the ball.

To prove the replacement theorem, we first require an isoperimetric inequality.

Let $\delta$ and $c_1$ be the constants from Lemma \ref{lem:thin_isoperimetric_MSY}, and let 
\begin{equation}\label{eqn:eta_assumption}
    \eta \coloneqq \min\left\{\frac{\cH^3(M)}{2},\ \frac{1}{8c^3c_1^2},\ c_1\delta^3\right\}.
\end{equation}

\begin{proposition}\label{prop:iso+PMC_controls_area}
If $\Omega \subset M$ is a set of finite perimeter with $\cH^3(\Omega) \leq \eta$, then
\begin{equation}\label{eqn:iso}
    \cH^3(\Omega) \leq c_1\cH^2(\partial \Omega)^{3/2},
\end{equation}
and
\begin{equation}\label{eqn:PMC_controls_area}
    \cA^h(\Omega) \geq \cA^c(\Omega) \geq \frac{1}{2}\cH^2(\partial \Omega).
\end{equation}
\end{proposition}

\begin{proof}
If $\cH^2(\partial \Omega) \geq \delta^2$, then \eqref{eqn:iso} holds trivially by the definition of $\eta$. On the other hand, if $\cH^2(\partial \Omega) < \delta^2$, Lemma \ref{lem:thin_isoperimetric_MSY} directly implies (\ref{eqn:iso}). Then (\ref{eqn:PMC_controls_area}) follows, as we have
\begin{align*}
    \cA^c(\Omega) = \cH^2(\partial \Omega) - c\cH^3(\Omega)
    &\geq \cH^2(\partial \Omega)(1 - cc_1^{2/3}\cH^3(\Omega)^{1/3}) \geq \frac{1}{2}\cH^2(\partial \Omega).
\end{align*}
The bound $\cA^h(\Omega) \geq \cA^c(\Omega)$ holds because $h \leq c$.
\end{proof}

We are now equipped to prove the replacement theorem.

\begin{lemma}[Replacement]\label{lem:replacement}
Let $\xi > 0$. Suppose $\cH^3(U) \leq \eta$. Let $\Sigma \in \mathcal{D}^*(U)$. Then there exists $\tilde{\Sigma} \in \mathcal{D}(U)[\Sigma]$ and an open subset $B \subset U$ with $\overline{B} \cap \tilde{\Sigma} = \varnothing$ such that
\begin{equation}\label{eqn:replace_subset}
    (\tilde{\Sigma} \cup \partial B) \cap U_{\xi} = \Sigma \cap U_{\xi},
\end{equation}
\begin{equation}\label{eqn:replace_inequality}
    \cH^2(\tilde{\Sigma}) + T_h(\Sigma, \tilde{\Sigma}) \leq \cH^2(\tilde{\Sigma}) + T_h(\Sigma, \tilde{\Sigma}) + \cH^2(\partial B) - c\cH^3(B) \leq \cH^2(\Sigma),
\end{equation}
and there is a set of disjoint oriented disks $\mathcal{S}$ with
\begin{equation}\label{eqn:replace_disks}
    [\tilde{\Sigma}] \llcorner U = \sum_{D \in \mathcal{S}} [D].
\end{equation}

Suppose additionally that
\begin{equation}\label{eqn:almost_min_prereplace}
    \cH^2(\Sigma) \leq \cH^2(\Sigma') + T_h(\Sigma, \Sigma') + \eps
\end{equation}
for all $\Sigma' \in \mathcal{D}(U)[\Sigma]$. Then
\begin{equation}\label{eqn:almost_min_postreplace}
    \cH^2(\tilde{\Sigma}) \leq \cH^2(\Sigma') + T_h(\tilde{\Sigma}, \Sigma') + \eps
\end{equation}
for all $\Sigma' \in \mathcal{D}(U)[\tilde{\Sigma}]$, and
\begin{equation}\label{eqn:replace_diff}
    \cH^2((\Sigma \triangle \tilde{\Sigma}) \cap U_{\xi}) \leq 2\eps.
\end{equation}

Let $P \subset \Sigma \cap U$ be a union of components of $\Sigma \cap U$. Let $\mathcal{S}_P \subset \mathcal{S}$ be the collection of disks $D \in \mathcal{S}$ with $D \cap U_{\xi} \subset P$, and let $\tilde{P} \coloneqq \bigcup_{D \in \mathcal{S}_P} D$. Then
\begin{equation}\label{eqn:replace_length}
    \cH^1\Big(\overline{\tilde{P}} \cap \partial U\Big) \leq \cH^1(\overline{P} \cap \partial U)
\end{equation}
and
\begin{equation}\label{eqn:replace_area}
    \cH^2(P) \leq \cH^2(\tilde{P}) + 2\eps.
\end{equation}
\end{lemma}

\begin{proof}
We use the same replacement procedure as in the proof of \cite[Theorem 1]{AS}. However, we have to keep track of relative volumes, so we include the details of the procedure before the smoothing, perturbation, and induction steps.

Let $\Lambda$ be any component of $\Sigma \setminus U$ such that $\Lambda \cap \partial \Sigma = \varnothing$. By Lemma \ref{lem:area_comparison}, there is a unique compact set $K_\Lambda \subset M \setminus U$ satisfying (\ref{eqn:area_comparison_smallness_conclusion}) and (\ref{eqn:area_comparison}. Given any component $\Lambda'$ of $\Sigma \setminus U$ with $\Lambda' \subset K_{\Lambda}$, we have $\Lambda' \cap \partial \Sigma = \varnothing$ and $K_{\Lambda'} \subset K_{\Lambda}$ (by the assumption on $U$ and $\Sigma$). Hence, we can select a component $\Lambda$ of $\Sigma \setminus U$ such that $\Lambda \cap \partial \Sigma = \varnothing$ and $\mathrm{int}(K_{\Lambda}) \cap \Sigma = \varnothing$.

Let $\chi$ be a diffeomorphism of $\D$ onto $\Sigma$, and consider the set $H\coloneqq \chi^{-1}(\Lambda)$. $H$ is a compact connected subset of $\D$ bounded by a finite collection of pairwise disjoint smooth Jordan curves $\gamma_1, \hdots, \gamma_n$. Suppose without loss of generality that $\gamma_1$ is the outermost of the curves. By the connectedness of $H$, we can write
\[ H = \overline{\mathrm{int} \gamma_1} \setminus \bigcup_{j=2}^n \mathrm{int} \gamma_j. \]

Recall that $F = \partial K_\Lambda \cap \partial U$. Define $F_1 \subset F$ to be the component of $F$ which contains the Jordan curve $\chi(\gamma_1)$, and $F_2 \coloneqq F \setminus F_1$. Define $I \coloneqq \cup \mathrm{int} \gamma_j$, where the union is over those $j$ such that $\chi(\gamma_j) \subset F_2$. Let $\Lambda_2 \coloneqq \chi(I)$. Then $\Lambda_1 \coloneqq \Lambda \cup \Lambda_2$ is diffeomorphic to the connected subset $H \cup I$ of $\D$, and $\partial \Lambda_1 = \partial F_1$. By \cite[Lemma 1]{AS}, we can construct a diffeomorphism $\alpha$ of $\Lambda_1$ onto $F_1$ such that $\alpha$ coincides with the identity on $\partial \Lambda_1 = \partial F_1$. We define a Lipschitz map $\hat{\chi} : \D \to M$ by
\[ \hat{\chi}(x) \coloneqq \begin{cases}
(\alpha \circ \chi)(x) & x \in H \cup I\\
\chi(x) & \text{otherwise}.
\end{cases} \]

We consider the following two cases:

\emph{Case 1: $\partial \Sigma \cap (F_1 \setminus \partial F_1) = \varnothing$}. Since $\mathrm{int}(K_{\Lambda}) \cap \Sigma = \varnothing$, $\hat{\chi}$ is injective, and we define
\[ \tilde{\Sigma} \coloneqq \hat{\chi}(\D) = (\Sigma \setminus \Lambda_1) \cup F_1. \]
Let $V_{\Lambda} \subset U \cup \bigcup_{\Lambda'} K_{\Lambda'}$ be the unique open set bounded by $\Lambda_1 \cup F_1$. Note that $\mathrm{int}(K_{\Lambda}) \subset V_{\Lambda}$, so $V_{\Lambda}$ is on the outside of $F_1$. Let $B \coloneqq V_{\Lambda} \cap U$, $V_{\Lambda}^o \coloneqq V_{\Lambda} \setminus \overline{U}$, $\Lambda_1^i \coloneqq \Lambda_1 \cap U$, $\Lambda_1^o = \Lambda_1 \setminus \overline{U}$, and $G \coloneqq V_\Lambda \cap \partial U$. Since $V_\Lambda$ is on the outside of $F_1$, we have $\partial B = \Lambda_1^i \cup G$. Our goal is to show
\begin{equation}\label{eqn:replace_case_1}
    \cH^2(\tilde{\Sigma}) + c\cH^3(V_\Lambda) + \cH^2(\partial B) - c\cH^3(B) < \cH^2(\Sigma).
\end{equation}
We have $\partial V_{\Lambda}^o = \Lambda_1^o \cup F_1 \cup G$, so (\ref{eqn:area_comparison}) implies
\[ \cH^2(F_1) + \cH^2(G) < \cH^2(\Lambda_1^o) - c\cH^3(V_\Lambda^o). \]
By the definition of $\tilde{\Sigma}$, we have
\[ \cH^2(\tilde{\Sigma}) = \cH^2(\Sigma) - \cH^2(\Lambda_1) + \cH^2(F_1). \]
Hence, we have
\begin{align*}
\cH^2(\tilde{\Sigma}) + c\cH^3(V_\Lambda) + \cH^2(\partial B) - c\cH^3(B)
& = \cH^2(\Sigma) - \cH^2(\Lambda_1^o) + \cH^2(F_1) + c\cH^3(V_\Lambda^o) + \cH^2(G)\\
& < \cH^2(\Sigma).
\end{align*}

\emph{Case 2: $\partial \Sigma \cap (F_1 \setminus \partial F_1) \neq \varnothing$}. Since $\mathrm{int}(K_{\Lambda}) \cap \Sigma = \varnothing$, we have $\partial \Sigma \subset F_1 \setminus \partial F_1$ in this case. Hence, $\hat{\chi}^{-1}(\partial \Sigma) = \partial \D \cup \gamma$, where $\gamma$ is a smooth Jordan curve in $\D \setminus \partial \D$. We define
\[ \tilde{\Sigma} \coloneqq \hat{\chi}(\mathrm{int} \gamma)\subset F_1. \]
Let $V_\Lambda \subset U \cup \bigcup_{\Lambda'} K_{\Lambda'}$ be the unique open set bounded by $\Sigma$ and $\tilde{\Sigma}$. Note that $\mathrm{int}(K_\Lambda) \subset V_\Lambda$, so $V_\Lambda$ is on the outside of $\tilde{\Sigma}$. Let $B \coloneqq V_\Lambda \cap U$, $V_\Lambda^o \coloneqq V_\Lambda \setminus \overline{U}$, $\Sigma^i \coloneqq \Sigma \cap U$, $\Sigma^o \coloneqq \Sigma \setminus \overline{U}$, and $G \coloneqq V_\Lambda \cap \partial U$. Since $V_\Lambda$ is on the outside of $\tilde{\Sigma}$, we have $\partial B = \Sigma^i \cup G$. Our goal is to show
\begin{equation}\label{eqn:replace_case_2}
    \cH^2(\tilde{\Sigma}) + c\cH^3(V_\Lambda) + \cH^2(\partial B) - c\cH^3(B) < \cH^2(\Sigma).
\end{equation}
We have $\partial V_\Lambda^o = \Sigma^o \cup \tilde{\Sigma} \cup G$, so (\ref{eqn:area_comparison}) implies
\[ \cH^2(\tilde{\Sigma}) + \cH^2(G) < \cH^2(\Sigma^o) - c\cH^3(V_\Lambda^o). \]
Hence, we have
\begin{align*}
    \cH^2(\tilde{\Sigma}) + c\cH^3(V_\Lambda)  + \cH^2(\partial B) - c\cH^3(B)
    & = \cH^2(\tilde{\Sigma}) + c\cH^3(V_\Lambda^o) + \cH^2(\Sigma^i) + \cH^2(G)\\
    & < \cH^2(\Sigma).
\end{align*}

The smoothing, slight perturbation, and induction steps now follow exactly as in the proof of \cite[Theorem 1]{AS} (with the additional observation that small perturbations make small changes to relative volumes). Taking $B$ to be the union of the perturbations of the sets $B$ from each iteration applied to the perturbed $\tilde{\Sigma}$ from the previous step, we obtain (\ref{eqn:replace_subset}) and (\ref{eqn:replace_disks}) immediately. Moreover, by (\ref{eqn:replace_case_1}) and (\ref{eqn:replace_case_2}), we have
\[ \cH^2(\tilde{\Sigma}) + c|T(\Sigma, \tilde{\Sigma})| + \cH^2(\partial B) - c\cH^3(B) \leq \cH^2(\Sigma). \]
Since
\[ T_h(\Sigma, \tilde{\Sigma}) \leq c|T(\Sigma, \tilde{\Sigma})|, \]
the second inequality in (\ref{eqn:replace_inequality}) holds. The first inequality in (\ref{eqn:replace_inequality}) follows from (\ref{eqn:PMC_controls_area}) and the fact that $B \subset U$.

(\ref{eqn:almost_min_postreplace}) follows from (\ref{eqn:replace_inequality}), (\ref{eqn:almost_min_prereplace}), Proposition \ref{prop:fill-in_properties}, and the fact that $\mathcal{D}(U)[\tilde{\Sigma}] \subset \mathcal{D}(U)[\Sigma]$.

To prove (\ref{eqn:replace_diff}), we note that
\[ \Sigma \triangle \tilde{\Sigma} \cap U_{\xi} = \partial B \cap U_{\xi}. \]
By (\ref{eqn:replace_inequality}) and (\ref{eqn:almost_min_prereplace}) with $\Sigma' = \tilde{\Sigma}$, we have
\[ \cH^2(\partial B) - c\cH^3(B) \leq \eps. \]
By (\ref{eqn:PMC_controls_area}), we have
\[ \cH^2(\partial B) \leq 2\eps, \]
which implies (\ref{eqn:replace_diff}).

(\ref{eqn:replace_length}) follows from the construction of $\tilde{\Sigma}$; namely, for each component of $\Sigma \cap U$, either (1) it is thrown away in $\partial B$, or (2) each component of its intersection with $\partial U$ is either (a) kept unchanged, or (b) pushed away from $\partial U$ in the perturbation step.

(\ref{eqn:replace_area}) follows from the construction of $\tilde{\Sigma}$ and (\ref{eqn:replace_diff}); namely, each component of $\Sigma \cap U$ is either (1) thrown away in $\partial B$, whose area is at most $2\eps$ by (\ref{eqn:replace_diff}), or (2) kept and given potentially larger area (by adding pieces of $\partial U$).
\end{proof}

\section{Filigree}
We adapt the filigree lemma \cite[Lemma 3]{AS} to the prescribed mean curvature setting. This lemma allows us to ignore small area components of a minimizing sequence in small balls. Let $U$ and $\theta$ be chosen (and $\beta$ and $\delta_1$ be thereby defined) as in \S\ref{sec:disk_notation}, and let $\eta$ be defined as in (\ref{eqn:eta_assumption}).

\begin{lemma}[Filigree]\label{lem:filigree}
Suppose $\cH^3(U(\theta)) \leq \eta$. Let $\Sigma \in \mathcal{D}^*(U(\theta))$ satisfy
\begin{equation}\label{eqn:filigree_area_assumption}
    \cH^2(\partial U(\theta)) + \cH^2(\Sigma \setminus U) \leq \delta_1^2/32
\end{equation}
and
\begin{equation}\label{eqn:almost_min_filigree}
    \cH^2(\Sigma) \leq \cH^2(\Sigma') + T_h(\Sigma, \Sigma') + \eps
\end{equation}
for all $\Sigma' \in \mathcal{D}(U)[\Sigma]$. Let $P$ be a union of some components of $\Sigma \cap U(\theta)$. If
\[ \cH^2(P) \leq \frac{\theta^2}{48\beta} \]
then
\[ \cH^2(P \cap U) \leq 6\eps. \]
\end{lemma}
\begin{proof}
For almost every $\sigma \in [0, \theta]$, we have $\Sigma \in \mathcal{D}^*(U(\sigma))$. Choose any such $\sigma = t\theta$ for $t \in [0, 1]$. We apply Lemma \ref{lem:replacement} to $\Sigma$ in $U(\sigma)$ with some small constant $\xi$, yielding $\tilde{\Sigma}$.

Let $\mathcal{S}$ be the set of disks composing $\tilde{\Sigma}$ in $U(\sigma)$. Let $\mathcal{S}_P \subset \mathcal{S}$ be the set of disks $D \in \mathcal{S}$ so that $D \cap U(\sigma)_{\xi} \subset P$. Let $\tilde{P} \coloneqq \bigcup_{D \in \mathcal{S}_P} D$. Take a disk $D \in \mathcal{S}_P$. Let $F_D$ denote the isoperimetric region in $\partial U(\sigma)$ for $\partial D$. Let $\Lambda_D$ denote the region in $U(\sigma)$ bounded by $D \cup F_D$.

Consider the competitor $\Sigma'$ for $\tilde{\Sigma}$ given by replacing each disk $D \in \mathcal{S}_P$ with $F_D$, and then resolving the overlaps and perturbing using Lemma \ref{lem:resolution_iso_collapse}. Then
\[ T_h(\tilde{\Sigma}, \Sigma') \leq c\sum_{D \in \mathcal{S}_P} \cH^3(\Lambda_D) + \eps/2. \]
By (\ref{eqn:iso}), we have
\[ \cH^3(\Lambda_D)^{2/3} \leq c_1^{2/3}(\cH^2(D) + \cH^3(F_D)). \]
Hence, we have
\begin{align*}
T_h(\tilde{\Sigma}, \Sigma')
& \leq cc_1^{2/3}(\cH^3(U(\sigma)))^{1/3}\sum_{D \in \mathcal{S}_P} (\cH^2(D) + \cH^2(F_D)) + \eps/2\\
& \leq \frac{1}{2}\sum_{D \in \mathcal{S}_P} (\cH^2(D) + \cH^2(F_D)) + \eps/2.
\end{align*}
By Lemma \ref{lem:replacement}(\ref{eqn:almost_min_postreplace}), we have
\begin{align*}
    \sum_{D \in \mathcal{S}} \cH^2(D)
    & = \cH^2(\tilde{\Sigma} \cap U(\sigma))\\
    & \leq \cH^2(\Sigma' \cap U(\sigma)) + T_h(\tilde{\Sigma}, \Sigma') + \eps\\
    & \leq \sum_{D \in \mathcal{S} \setminus \mathcal{S}_P} \cH^2(D) + \sum_{D \in \mathcal{S}_P} \cH^2(F_D) + \frac{1}{2}\sum_{D \in \mathcal{S}_P} (\cH^2(D) + \cH^2(F_D)) + 2\eps,
\end{align*}
so
\[ \cH^2(\tilde{P}) = \sum_{D \in \mathcal{S}_P} \cH^2(D) \leq 3\sum_{D \in \mathcal{S}_P} \cH^2(F_D) + 4\eps. \]
By the isoperimetric inequality on $\partial U(\sigma)$, we have
\[ \cH^2(\tilde{P}) \leq 3\beta(\cH^1(\partial \tilde{P}))^2 + 4\eps. \]
By Lemma \ref{lem:replacement}(\ref{eqn:replace_length}), we have
\[ \cH^1(\partial \tilde{P}) \leq \cH^1(P \cap \partial U(\sigma)). \]
By Lemma \ref{lem:replacement}(\ref{eqn:replace_area}), we have
\[ \cH^2(P \cap U(\sigma)) \leq \cH^2(\tilde{P}) + 2\eps. \]
Hence, we have
\[
\cH^2(P \cap U(t\theta)) \leq 3\beta(\cH^1(P \cap \partial U(t\theta)))^2 + 6\eps.
\]

Let $f(t) \coloneqq \cH^2(P \cap U(t\theta)) - 6\eps$. Then by the coarea formula we have
\begin{equation}\label{eqn:filigree}
f \leq 3\beta\theta^{-2}(f')^2.
\end{equation}
Suppose without loss of generality that $f(1) > 0$. Let
\[ t_0 \coloneqq \inf\{t \in [0, 1] \mid f(t) > 0\}. \]
Integrating (\ref{eqn:filigree}) gives
\[ (1-t_0)\frac{\theta}{2\sqrt{3\beta}} \leq \sqrt{f(1)} - \sqrt{f(t_0)} \leq \sqrt{\cH^2(P)}. \]
Then our area assumption gives
\[ t_0 \geq 1 - \frac{2\sqrt{3\beta}}{\theta}\sqrt{\cH^2(P)} \geq \frac{1}{2}. \]
Hence, $\cH^2(P \cap U) \leq \cH^2(P \cap U(\theta/2)) \leq 6\eps$.
\end{proof}

\section{Interior Regularity}
In this section we prove the interior regularity for $\cA^h$ minimizing sequences of disks, following the strategy of \cite{AS}.

We first apply Lemma \ref{lem:filigree} to deduce rectifiability of the limit varifold.

\begin{lemma}[Rectifiability]\label{lem:rectifiability}
Let $\{\Sigma_k\}_{k\in\N} \subset \mathcal{D}^*(U)$ satisfy $\cH^2(\Sigma_k) \leq \delta_2^2/64$ and
\begin{equation}\label{eqn:min_seq_rectifiability}
    \cH^2(\Sigma_k) \leq \cH^2(\Sigma_k') + T_h(\Sigma_k, \Sigma_k') + \eps_k
\end{equation}
for all $\Sigma_k' \in \mathcal{D}(U)[\Sigma_k]$ with $\eps_k \to 0$. Let $V$ be the subsequential varifold limit of $\Sigma_k$ in $U$ provided by Theorem \ref{thm:compactness}. Then
\[ \Theta^2(\|V\|, x) \geq \frac{1}{1000\beta} > 0 \]
for all $x \in \mathrm{spt}\|V\|$. Moreover, $V$ is rectifiable.
\end{lemma}
\begin{proof}
Rectifiability follows from the positive density lower bound by \cite[Theorem 5.5]{allard}.

Let $x \in \mathrm{spt}\|V\|$. Choose $r_0 \leq 1/2$ small enough so that $\cH^3(B_{r_0}(x)) \leq \eta$ and $\cH^2(\partial B_r(x)) \leq \delta_2^2/64$ for all $r \leq r_0$. Let $r_i \to 0$ with $r_i \leq r_0$, $B_{r_i}(x) \subset U$, and $\Sigma_k \in \mathcal{D}^*(B_{r_i}(x))$ for all $k$ and $i$, which is guaranteed by Sard's theorem. By Lemma \ref{lem:filigree} with $\theta = r_i/2$, if
\[ \liminf_k \cH^2(\Sigma_k \cap B_{r_i}(x)) \leq \frac{r_i^2}{192\beta}, \]
then
\[ \|V\|(B_{r_i/2}(x)) \leq \liminf_k \cH^2(\Sigma_k \cap B_{r_i/2}(x)) \leq \liminf_k 6\eps_k = 0, \]
which contradicts $x \in \mathrm{spt}\|V\|$. Hence, by the assumption of transversality, we have
\[ \|V\|(B_{r_i}(x)) = \lim_k \cH^2(\Sigma_k \cap B_{r_i}(x)) \geq \frac{r_i^2}{192\beta}. \]
By the monotonicity of (\ref{eqn:monotonicity}) for $V$, we have
\[ \Theta^2(\|V\|, x) \geq \frac{1}{1000\beta}, \]
as desired.
\end{proof}

We now have all the tools to use the stategy of \cite{AS} to deduce interior regularity.

\begin{theorem}\label{thm:local_regularity}
Let $\{\Sigma_k\}_{k\in\N} \subset \mathcal{D}^*(U)$ satisfy $\sup_k \cH^2(\Sigma_k \cap U) \leq \delta_2^2/64$ and
\begin{equation}\label{eqn:min_seq_regularity}
    \cH^2(\Sigma_k) \leq \cH^2(\Sigma_k') + T_h(\Sigma_k, \Sigma_k') + \eps_k
\end{equation}
for all $\Sigma_k' \in \mathcal{D}(U)[\Sigma_k]$ with $\eps_k \to 0$. Let $V$ and $\Omega$ respectively be the subsequential varifold limit of $\Sigma_k \cap U$ and the subsequential Caccioppoli set limit of $\Omega_k$ provided by Theorem \ref{thm:compactness}. Let $\alpha\in (0,1)$.

$V$ is an integer rectifiable varifold with $c$-bounded first variation. For every $x \in \mathrm{spt}\|V\|$, we have $\Theta^2(\|V\|, x) = n_x \in \N$. If $h(x) \neq 0$, then $n_x \in \{1\} \cup 2\N$. For every $x \in \mathrm{spt}\|V\|$, there is a neighborhood $W_x$ of $x$ so that the following hold.

\begin{itemize}
    \item $V \llcorner G(W_x,2) = \sum_{l=1}^{n_x}\mathbf{v}(N_l, 1)$, where $N_l \subset W_x$ is a $C^{1,\alpha}$ surface with $c$-bounded first variation. Moreover, each $N_l$ is on one side of $N_{l'}$ intersecting tangentially at $x$ for any $l,\ l' \in \{1, \hdots, n_x\}$.
    \item If $n_x = 1$, then
    \begin{itemize}
        \item $\|V\| \llcorner W_x = |D\mathbf{1}_{\Omega}| \llcorner W_x$,
        \item $V \llcorner G(W_x,2) = \mathbf{v}(N, 1)$, where $N \subset W_x$ is a smooth stable surface with prescribed mean curvature $h$ with respect to the set $\Omega \cap W_x$.
    \end{itemize}
    \item If $h(x) \neq 0$ and $n_y = 2n$ for all $y \in \mathrm{spt}\|V\| \cap W_x'$ for an open set $W_x' \subset W_x$ containing $x$, then $V \llcorner G(W'_x,2) = \mathbf{v}(N, 2n)$, where $N\subset W_x'$ is a smooth stable minimal surface. Moreover, $\Omega \cap W_x' = W_x' \setminus N$ if $h(x) > 0$, and $\Omega \cap W_x' = \varnothing$ if $h(x) < 0$.
    \item If $h(x) \neq 0$, $n_x = 2n$, and there is a sequence $x_j \to x$ with $x_j \in \mathrm{spt}\|V\|$ and $n_{x_j} \neq 2n$, then $n = 1$ and the surfaces $N_1,\ N_2 \subset W_x$ additionally satisfy that the generalized mean curvature of $N_1$ points towards $N_2$ and vice versa.
\end{itemize}
\end{theorem}
\begin{proof}
By Lemma \ref{lem:rectifiability}, $V$ is rectifiable and has a uniform positive density lower bound in its support. By Theorem \ref{thm:compactness}(\ref{eqn:bounded_first_variation}), $V$ has $c$-bounded first variation.

\emph{Integer Density}. We show that the density of $V$ is an integer at every $x \in \mathrm{spt}\|V\|$. Since this result concerns the infinitesimal behavior of $V$ near $x$, we push objects to $T_xM$ using the exponential map $\phi_x$. In $T_xM$, we use the standard Euclidean metric in the geodesic normal coordinates, which is the $C^{\infty}$ limit of the pullback metric under a blow up rescaling.

First, let $x \in \mathrm{spt}\|V\|$ be such that there is a tangent cone $C \in \mathrm{VarTan}(V,x)$ with $\mathrm{spt}\|C\| \subset P_C$, where $P_C$ is a 2-plane in $T_xM$ (by rectifiability, this holds $\|V\|$-almost everywhere).

Fix any $\xi > 0$.

Without loss of generality, assume $P_C \supset \mathrm{spt}\|C\|$ is the $x_1x_2$-plane in $T_xM$.

Let $\mu_t$ denote the transformation of $T_xM$ given by $\mu_t(y) = ty$. By the definition of $C$, there is a sequence $\{r_j\} \to \infty$ so that
\begin{equation}\label{eqn:var_tan}
    (\mu_{r_j}\circ \phi_x^{-1})_{\#}V \rightharpoonup C
\end{equation}
as $j \to \infty$. We always assume $r_j \geq 1$.

By the lower density bound from Lemma \ref{lem:rectifiability} and the monotonicity of (\ref{eqn:monotonicity}), there is a sequence $\{\sigma_j\} \to 0$ so that
\begin{equation}\label{eqn:flat}
    (B_{1}^{\R^2}(0) \times (-2, 2)) \cap \mathrm{spt}\|(\mu_{r_j}\circ \phi_x^{-1})_{\#}V\| \subset B_{1}^{\R^2}(0) \times (-\sigma_j/2, \sigma_j/2).
\end{equation}

We establish some notation. To distinguish objects in $T_xM$ from objects in $M$, we use lower case letters here. Let
\begin{itemize}
    \item $v^j \coloneqq (\mu_{r_j} \circ \phi_x^{-1})_{\#}V$,
    \item $s_{k,j} \coloneqq (\mu_{r_j} \circ \phi_x^{-1})(\Sigma_k)$,
    \item $b_{\rho} \coloneqq B_{\rho}^{T_xM}(0)$,
    \item $c_{\rho, \sigma} \coloneqq \partial b_{\rho} \cap \{x_3 \in [-\sigma, \sigma]\}$,
    \item $d_{\rho, \sigma}^+ \coloneqq \partial b_{\rho} \cap \{x_3 > \sigma\}$,
    \item $d_{\rho, \sigma}^- \coloneqq \partial b_{\rho} \cap \{x_3 < -\sigma\}$.
\end{itemize}

By (\ref{eqn:flat}), we have
\begin{equation}\label{eqn:noarea}
    \limsup_k \cH^2(s_{k,j} \cap (B_1^{\R^2}(0) \times ([-1,1] \setminus (-\sigma_j/2, \sigma_j/2)))) = 0.
\end{equation}

By (\ref{eqn:noarea}) and the coarea formula, for almost every $\rho < 1$ we have
\begin{equation}\label{eqn:nolength}
\limsup_k \cH^1(s_{k,j} \cap (d_{\rho, \sigma_j/2}^+ \cup d_{\rho,\sigma_j/2}^-)) = 0
\end{equation}
for every $j$. Fix some $\rho_1 \in (3/4, 1)$ satisfying (\ref{eqn:nolength}) so that $s_{k,j}$ has transverse intersection with $\partial b_{\rho_1}$ for every $k$ and $j$, which is guaranteed by Sard's theorem.

Fix some $\zeta > 0$. For $k\geq k_1(j, \zeta, \rho_1)$, we have
\begin{equation}\label{eqn:length_in_caps}
    \cH^1(s_{k,j} \cap (d_{\rho_1, \sigma_j/2}^+ \cup d_{\rho_1, \sigma_j/2}^-)) < \zeta.
\end{equation}
For $\zeta \leq \zeta_1(M)$, the coarea formula gives
\begin{equation}\label{eqn:int_in_edge}
    s_{k,j} \cap (\partial d_{\rho_1, \sigma_{k,j}}^+ \cup \partial d_{\rho_1, \sigma_{k,j}}^-) = \varnothing,
\end{equation}
for a sequence $\sigma_{k,j} \in (\sigma_j/2, \sigma_j)$.

We apply Lemma \ref{lem:replacement} to $s_{k,j} \in \mathcal{D}^*(b_{\rho_1})$ with constant $\xi$. We obtain the replacement $\tilde{s}_{k,j} \in \mathcal{D}(b_{\rho_1})[s_{k,j}]$ and a collection of disks $\mathfrak{S}_{k,j}$ composing $\tilde{s}_{k,j}$ in $b_{\rho_1}$.

Let $\xi' \in [\xi, 2\xi]$ so that
\begin{equation}\label{eqn:transverse}
    \|v^j\|(\partial b_{\rho_1 - \xi'}) = 0
\end{equation}
for all $j$.

By Lemma \ref{lem:replacement}(\ref{eqn:replace_subset}) and (\ref{eqn:replace_diff}) (also using (\ref{eqn:transverse})), we have
\begin{equation}\label{eqn:regularity_replace}
v^j \llcorner G(b_{\rho_1-\xi'},2)
= \lim_k \sum_{d \in \mathfrak{S}_{k,j}} \mathbf{v}(d \cap b_{\rho_1-\xi'}).
\end{equation}

Step 1: We show that we can throw away disks $d \in \mathfrak{S}_{k,j}$ with boundary in the caps $d_{\rho_1, \sigma_{k,j}}^{\pm}$. Let $\mathfrak{C}_{k,j}$ be the set of all such disks in $\mathfrak{S}_{k,j}$. For every $d \in \mathfrak{C}_{k,j}$, let $f_d$ be the isoperimetric region for $\partial d$ in $d_{\rho_1, \sigma_{k,j}}^+ \cup d_{\rho_1, \sigma_{k,j}}^-$. By (\ref{eqn:length_in_caps}) and Lemma \ref{lem:replacement}(\ref{eqn:replace_length}), we have
\[ \sum_{d \in \mathfrak{C}_{k,j}} \cH^1(\partial d) \leq \zeta \]
for $k \geq k_1(j, \zeta, \rho_1)$. By the isoperimetric inequality in the round sphere, we have
\[ \cH^2(f_d) \leq \beta_1 \cH^1(\partial d)^2 \]
for all $d \in \mathfrak{C}_{k,j}$. Hence, we have
\[ \sum_{d \in \mathfrak{C}_{k,j}} \cH^2(f_d) \leq \beta_1 \sum_{d \in \mathfrak{C}_{k,j}} \cH^1(\partial d)^2 \leq \beta_1 \left(\sum_{d \in \mathfrak{C}_{k,j}} \cH^1(\partial d)\right)^2 \leq \beta_1 \zeta^2 \]
for $k \geq k_1(j, \zeta, \rho_1)$. Consider the competitor given by replacing each $d \in \mathfrak{C}_{k,j}$ by $f_d$ (and then resolving the overlaps without increasing area by more than $\eps_k$ by \cite[Corollary 1]{AS}). By (\ref{eqn:min_seq_regularity}), Lemma \ref{lem:replacement}(\ref{eqn:almost_min_postreplace}), and the fact that $|d\phi_x|\raisebox{-.2em}{$\big\vert_{B_{\rho/r_j}^M(x)}$} \leq 1 + o_j(1)$, we have
\begin{multline*}
    r_j^{-2}\sum_{d\in \mathfrak{S}_{k,j}} \cH^2(d) \leq r_j^{-2}\sum_{d \in \mathfrak{S}_{k,j} \setminus \mathfrak{C}_{k,j}} \cH^2(d) + r_j^{-2}\sum_{d \in \mathfrak{C}_{k,j}} \cH^2(f_d) + 2cr_j^{-3}\cH^3(b_{\rho_1}) + 2\eps_k + o(r_j^{-2}).
\end{multline*}  
Rewriting, we have
\begin{equation}\label{eqn:disks_in_caps}
    \sum_{d\in \mathfrak{C}_{k,j}} \cH^2(d) \leq \beta_1\zeta^2 + \frac{2c\cH^3(b_1)}{r_j} + 2r_j^2\eps_k + o_j(1)
\end{equation}
for $k \geq k_1(j, \zeta, \rho_1)$. Taking $j \geq j_2(M, \xi)$ and $\zeta \leq \zeta_2(M, \xi)$, and then taking $k \geq k_2(j, \zeta, \rho_1, \xi)$, we can apply Lemma \ref{lem:filigree} to conclude that
\[ \lim_k \sum_{d \in \mathfrak{C}_{k,j}} \mathbf{v}(d \cap b_{\rho_1 - \xi}) = 0 \]
for all $j\geq j_2(M, \xi)$.

Step 2: We show that we can throw away the disks with nullhomotopic boundary in $c_{\rho_1, \sigma_{k,j}}$. Let $\mathfrak{N}_{k,j}$ be the set of all such disks in $\mathfrak{S}_{k,j}$. By the monotonicity of (\ref{eqn:monotonicity}), there is a uniform upper bound for the mass of $v^j$ in $b_{1}$ for all $j$ depending on $V$ and $c$.

Let $\mathfrak{N}_{k,j}^{-} \subset \mathfrak{N}_{k,j}$ be the set of disks $d \in \mathfrak{N}_{k,j}$ with $\cH^2(d) \leq \frac{\xi^2}{48\beta_1}$. By the uniform mass upper bound, there is an $m = m(V, c, \xi) \in \N$ such that
\[ \sum_{d \in \mathfrak{N}_{k,j}^{-}} \cH^2(d) \leq m\frac{\xi^2}{48\beta_1} \]
for all $j$. By a standard counting argument, we can partition $\mathfrak{N}_{k,j}^{-}$ into $2m$ disjoint sets $\{\mathfrak{N}^{-}_{k,j,i}\}_{i=1}^{2m}$ (some potentially empty) such that
\[ \sum_{d \in \mathfrak{N}^{-}_{k,j,i}} \cH^2(d) \leq \frac{\xi^2}{48\beta_1}. \]
By Lemma \ref{lem:filigree}, we have
\[ \lim_{k} \sum_{d \in \mathfrak{N}_{k,j}^{-}} \mathbf{v}(d \cap b_{\rho_1-\xi}) = \sum_{i=1}^{2m} \lim_{k} \sum_{d \in \mathfrak{N}_{k,j,i}^{-}} \mathbf{v}(d \cap b_{\rho_1-\xi}) = 0. \]

Let $\mathfrak{N}^{+}_{k,j}$ denote the set of disk $d \in \mathfrak{N}_{k,j}$ with $\cH^2(d) > \frac{\xi^2}{48\beta_1}$. By the uniform mass upper bound, we have $\#\mathfrak{N}_{k,j}^{+} \leq m$. Let $f_d$ denote the isoperimetric region in $c_{\rho_1, \sigma_{k,j}}$ for $\partial d$. Consider the competitor given by replacing $d$ with $f_d$ for all $d \in \mathfrak{N}_{k,j}^{+}$ (and then resolving overlaps without increasing area by more than $\eps_k$  by \cite[Corollary 1]{AS}). As in step 1, (\ref{eqn:min_seq_regularity}) and Lemma \ref{lem:replacement}(\ref{eqn:almost_min_postreplace}) imply
\begin{multline*}
    r_j^{-2}\sum_{d\in \mathfrak{S}_{k,j}} \cH^2(d) \leq r_j^{-2}\sum_{d \in \mathfrak{S}_{k,j} \setminus \mathfrak{N}_{k,j}^{+}} \cH^2(d) + r_j^{-2}\sum_{d \in \mathfrak{N}_{k,j}^{+}} \cH^2(f_d) + 2cr_j^{-3}\cH^3(b_{\rho_1}) + 2\eps_k + o(r_j^{-2}).
\end{multline*} 
Rewriting, we have
\begin{equation}\label{eqn:disks_in_cyl}
    \sum_{d\in \mathfrak{N}_{k,j}^{+}} \cH^2(d)
    \leq m\cH^2(c_{\rho_1, \sigma_j}) + \frac{2c\cH^3(b_1)}{r_j} + 2r_j^2\eps_k + o_j(1).
\end{equation}
Assuming $j \geq j_3(c,\rho_1, M, V, \xi)$ and then taking $k \geq k_3(j, \rho_1, M, \xi)$, we have $\# \mathfrak{N}_{k,j}^{+} = 0$.

Step 3: We obtain lower and upper bounds for the area of the disks whose boundary is homotopically nontrivial in $c_{\rho_1, \sigma_{k,j}}$. Let $\mathfrak{L}_{k,j}$ denote the set of such disks in $\mathfrak{S}_{k,j}$.

Let $d \in \mathfrak{L}_{k,j}$. Since $\partial d$ is homotopically nontrivial in $c_{\rho_1, \sigma_{k,j}}$, (\ref{eqn:noarea}) implies
\begin{equation}\label{eqn:area_lower_bd}
\cH^2(d \cap b_\rho) \geq \pi (\rho - \xi)^2
\end{equation}
for $\rho \in [\rho_1-2\xi, \rho_1]$, $j \geq j_4(\rho_1, M, V, \xi)$, and $k \geq k_4(j, \rho_1, M, \xi)$.

By the uniform mass bound on the limit, there is an $m = m(c,V)$ so that $\#\mathfrak{L}_{k,j} \leq m$ for all $j \geq j_4(c,\rho_1,M, V,\xi)$ and $k \geq k_4'(j, c,\rho_1, M, V,\xi)$.

For each $d \in \mathfrak{L}_{k,j}$, let $m_d$ be the disk with $\partial m_d = \partial d$ so that $m_d \setminus b_{\rho_1} \subset c_{\rho_1, \sigma_{k,j}}$ and $m_d \cap b_{\rho_1}$ is the flat disk $d_{k,j}^* \subset b_{\rho_1}$ with boundary $\partial d_{\rho_1, \sigma_{k,j}}^+$. Note that
\[ \cH^2(m_d) \leq \cH^2(c_{\rho_1, \sigma_{k,j}}) + \cH^2(d_{k,j}^*). \]
Consider the competitor given by replacing each $d \in \mathfrak{L}_{k,j}$ with $m_d$ (and then resolving overlaps using \cite[Corollary 1]{AS}). As in steps 1 and 2, (\ref{eqn:min_seq_regularity}) implies
\begin{multline*}
    r_j^{-2}\sum_{d\in \mathfrak{S}_{k,j}} \cH^2(d) \leq r_j^{-2}\sum_{d \in \mathfrak{S}_{k,j} \setminus \mathfrak{L}_{k,j}} \cH^2(d) + r_j^{-2}\sum_{d \in \mathfrak{L}_{k,j}} \cH^2(m_d) + 2cr_j^{-3}\cH^3(b_{\rho_1}) + 2\eps_k + o(r_j^{-2}).
\end{multline*}  
Rewriting, we have
\begin{align*}
\sum_{d\in \mathfrak{L}_{k,j}} \cH^2(d)
& \leq (\#\mathfrak{L}_{k,j})(\cH^2(d_{k,j}^*) + \cH^2(c_{\rho_1,\sigma_j})) + \frac{2c\cH^3(b_1)}{r_j} + 2r_j^2\eps_k + o_j(1).
\end{align*}
Assuming $j \geq j_5(c, \rho_1, M, V, \xi)$ and then taking $k \geq k_5(j, c, \rho_1, M, V, \xi)$, we have for every $d \in \mathfrak{L}_{k,j}$ (by (\ref{eqn:area_lower_bd}))
\begin{equation}\label{eqn:area_upper_bd}
    \cH^2(d) \leq \pi (\rho_1+\xi)^2.
\end{equation}

Step 4: We compute the density of $V$ at $x$. By (\ref{eqn:transverse}) and the above work, we have
\[ \|v^j\|(b_{\rho_1-\xi'}) = \lim_k \sum_{d \in \mathfrak{L}_{k,j}}\cH^2(d \cap b_{\rho_1 - \xi'}). \]
By (\ref{eqn:area_lower_bd}) and (\ref{eqn:area_upper_bd}), we have
\[ (\limsup_k \#\mathfrak{L}_{k,j}) \pi(\rho_1 - \xi)^2 \leq \|v^j\|(b_{\rho_1-\xi',j}) \leq (\liminf_k \#\mathfrak{L}_{k,j}) \pi(\rho_1 + \xi)^2. \]
Assuming $\xi$ is sufficiently small (depending on universal quantities), we have $\lim_k \#\mathfrak{L}_{k,j}$ exists, so we have $\#\mathfrak{L}_{k,j} = n_j \in \N$ for all $j \geq j_5(c, \rho_1, M, V, \xi)$ and $k \geq k_6(j, c, \rho_1,M,V,\xi)$. By the monotonicity of (\ref{eqn:monotonicity}), $n_j$ is decreasing for $j \geq j_6(c, \rho_1, M, V, \xi)$. By Lemma \ref{lem:rectifiability}, we have $n_j \geq 1$. Hence, we have $n_j \equiv n \geq 1$ for all $j \geq j_6'(c, \rho_1, M, V, \xi)$. Therefore, we have
\[ n - o_\xi(1) \leq \Theta^2(\|V\|, x) \leq n + o_\xi(1). \]
Since $\xi$ is arbitrary, we conclude that $\Theta^2(\|V\|,x) = n \in \N$, i.e. $V$ has integer density at all points $x$ in its support such that some tangent cone $C\in\mathrm{VarTan}(V,x)$ is supported in a 2-plane $P_C\subset T_xM$.

\emph{Flatness of Tangent Cones}. We now show that for all $x\in\mathrm{spt}\|V\|$, every tangent cone $C\in \mathrm{VarTan}(V,x)$ is supported in some 2-plane $P_C\subset T_xM$. 
\begin{lemma}[Flatness]\label{lem:flatness}
Let $U, \Sigma_k, V$ be as in the statement of Theorem \ref{thm:local_regularity}. If $x \in \mathrm{spt}\|V\|$ and $C \in \mathrm{VarTan}(V, x)$, then $\mathrm{spt}\|C\| \subset P_C$ for some 2-plane $P_C$.
\end{lemma}
\begin{proof}
Let $x \in \mathrm{spt}\|V\|$, and let $C \in \mathrm{VarTan}(V, x)$. There is a subsequence (not relabeled) and radii $r_k \to \infty$ so that $v^k := (\mu_{r_k}\circ \phi_x^{-1})_{\#}V \rightharpoonup C$. By selecting a further subsequence (not relabeled), the surfaces  $s_k := (\mu_{r_k}\circ \phi_x^{-1})(\Sigma_k)$ satisfy
\[ \mathbf{v}(s_k) \rightharpoonup C. \]
We can choose the subsequence so that $r_k^2\eps_k \to 0$.

By Sard's theorem, we can find a radius $r > 0$ so that $b_r := b_r^{T_xM}(0)$ satisfies $s_k \in \mathcal{D}^*(b_r)$ for all $k$. Note that by (\ref{eqn:min_seq_regularity}) (and the same argument as step 1 above), we have
\[ \cH^2(s_k) \leq \cH^2(s_k') + r_k^{-1}T_h(s_k, s_k') + r_k^2\eps_k + o_k(1) \]
for all $s_k' \in \mathcal{D}(b_r)[s_k]$.

We apply Lemma \ref{lem:replacement} to $s_k$ in $b_{r}$ with constant $\theta < r/2$, producing new disks $\tilde{s}_k$ so that
\[ \cH^2(\tilde{s}_k) \leq \cH^2(s_k') + r_k^{-1}T_h(\tilde{s}_k, s_k') + r_k^2\eps_k + o_k(1) \]
for all $s_k' \in \mathcal{D}(b_r)[\tilde{s}_k]$. By (\ref{eqn:replace_diff}), we have
\[ \lim_{k \to \infty} \mathbf{v}(\tilde{s}_k \cap b_{r/2}) = C \llcorner G(b_{r/2}, 2). \]

Since $\tilde{s}_k \triangle s_k' \subset b_r$ for all $s_k' \in \mathcal{D}(b_r)[\tilde{s}_k]$ by Lemma \ref{lem:replacement}, we have $T_h(\tilde{s}_k, s_k') \leq c\cH^3(b_r) < \infty$. Hence,
\[ \eps_k' \coloneqq r_k^{-1}T_h(s_k, s_k') + r_k^2\eps_k + o_k(1) \]
satisfies $\eps_k' \to 0$, and
\begin{equation}\label{eqn:min_seq_area_flatness}
    \cH^2(\tilde{s}_k) \leq \cH^2(s_k') + \eps_k'
\end{equation}
for all $s_k' \in \mathcal{D}(b_r)[\tilde{s}_k]$.

Since we have shown the disks $\tilde{s}_k$ are a minimizing sequence for area with respect to competitors $s_k'$ in the class $\mathcal{D}(b_r)[\tilde{s}_k]$, we can now conclude by following the proof strategy of \cite{AS}. 

Indeed, we observe that \cite[Theorem 3 and Corollary 2]{AS} hold for the area minimization problem restricted to $\mathcal{D}(b_r)[\tilde{s}_k]$. Case I and Case II in the proof of \cite[Theorem 3]{AS} follow directly from the argument in \cite{AS}.
As for Case III, note that the relevant step in the proof only uses disks obtained from the replacement procedure (which lie in our restricted set of competitors) and comparison to isoperimetric regions (also in the restricted set of competitors). The rest of the proofs are unchanged.
\end{proof}

Lemma \ref{lem:flatness}, together with the previous step of the proof, shows that $V$ has positive integer density at every point of its support.

\emph{Decomposition of $V$}. To study the local regularity of $V$, we return to work in $M$. We adapt the set up from the proof of integer density. Let $x \in \mathrm{spt}\|V\|$, $\xi > 0$, $r_j \to \infty$, $\sigma_j \to 0$, $\zeta > 0$, $\sigma_{k,j} \in (\sigma_j/2, \sigma_j)$ as above. We define
\begin{itemize}
    \item $B_{\rho, j} := B_{\rho/r_j}^M(x)$,
    \item $C_{\rho, \sigma, j} := (\mu_{r_j} \circ \phi_x^{-1})^{-1}(c_{\rho, \sigma}) \subset B_{\rho, j}$,
    \item $D_{\rho, \sigma, j}^{\pm} := (\mu_{r_j} \circ \phi_x^{-1})^{-1}(d_{\rho, \sigma}^{\pm}) \subset B_{\rho, j}$.
\end{itemize}
We apply Lemma \ref{lem:replacement} to $\Sigma_k \in \mathcal{D}^*(B_{\rho_1, j})$ with constant $\xi$, producing new disks $S_{k,j}$. By (\ref{eqn:replace_inequality}), we have
\begin{equation}\label{eqn:ah_min_interior_regularity}
    \cH^2(S_{k,j}) \leq \cH^2(S_{k,j}') + T_h(S_{k,j}, S_{k,j}') + \eps_k
\end{equation}
for all $S_{k,j}' \in \mathcal{D}(B_{\rho_1, j})[S_{k,j}]$. Let $\mathcal{S}_{k,j}$ be the set of disks composing $S_{k,j}$ in $B_{\rho_1,j}$. Let $\mathcal{L}_{k,j} \subset \mathcal{S}_{k,j}$ be the disks with homotopically nontrivial boundary in $C_{\rho_1, \sigma_{k,j}, j}$. We translate the estimates from the proof of integer density to $M$.

Henceforth we assume $j \geq j_6'(c, \rho_1, M,V, \xi)$ and $k \geq k_6(j, c, \rho_1, M, V, \xi)$ as above. Namely, we have $\#\mathcal{L}_{k,j} = n$ for all such $j$ and $k$.

Label the disks $D \in \mathcal{L}_{k,j}$ as $\{D_{k,j,1}, \hdots, D_{k,j,n}\}$. There is a subsequence (not relabeled, after shuffling the 1 through $n$ labels) so that $\mathbf{v}(D_{k,j,i})$ converges in the varifold sense to a limit $V_{j,i}$ for $i = 1, \hdots, n$ as $k \to \infty$, satisfying
\begin{equation}\label{eqn:mass_bds}
    \|V_{j,i}\|(B_{\rho_1-2\xi,j}) \geq \pi (\rho_1 - 4\xi)^2r_j^{-2} \ \ \text{and}\ \
    \|V_{j,i}\|(B_{\rho_1, j}) \leq \pi(\rho_1 + 2\xi)^2r_j^{-2},
\end{equation}
\begin{equation}\label{eqn:spt}
    \mathrm{spt}\|V_{j,i}\| \subset (\mu_{r_j} \circ \phi_x^{-1})^{-1}(B_{\rho_1}^{T_xM}(0) \cap \{-\sigma_j \leq x_3 \leq \sigma_j\}),
\end{equation}
and
\begin{equation}\label{eqn:limit_sum_of_disks}
    V \llcorner G(B_{\rho_1-2\xi, j}, 2) = \sum_{i=1}^n V_{j,i} \llcorner G(B_{\rho_1-2\xi, j}, 2).
\end{equation}

\emph{Reduction to Stacked Disk Minimization}.
We now reduce the problem to a stacked disk minimization problem.

Let $\mathcal{E}_{k,j} \coloneqq \mathcal{S}_{k,j} \setminus \mathcal{L}_{k,j}$. For $D \in \mathcal{E}_{k,j}$ let $F_D \subset \partial B_{\rho_1,j}$ be the isoperimetric region for $\partial D$. Let $\Lambda_D \subset B_{\rho_1,j}$ be the open set bounded by $D \cup F_D$.

Steps 1 and 2 from the proof of integer density above imply
\begin{equation}\label{eqn:extra_small_area}
    \limsup_k \sum_{D \in \mathcal{E}_{k,j}} \cH^2(D \cap B_{1/2, j}) = 0.
\end{equation}
In particular, (\ref{eqn:disks_in_caps}) and (\ref{eqn:disks_in_cyl}) imply
\begin{equation}\label{eqn:extra_small_area1}
    \cH^2(D) + \cH^2(F_D) \leq A(\xi, \zeta, j,k)
\end{equation}
with $A(\xi, \zeta, j, k) \to 0$ as $(\xi, \zeta) \to 0$ and $j, k \to \infty$. By (\ref{eqn:iso}), we have
\[ \cH^3(\Lambda_D) \leq c_1A(\xi,\zeta,j,k)^{3/2}. \]
Taking $\xi$ and $\zeta$ small and $j$ and $k$ large, the coarea formula implies
\begin{equation}\label{eqn:extra_apply_iso}
    \limsup_k \cH^2(\Lambda_D \cap \partial B_{\rho,j}) \leq \frac{1}{4}\cH^2(\partial B_{1/4,j}) \leq \frac{1}{4}\cH^2(\partial B_{\rho,j})
\end{equation}
for $\rho$ in a set of positive measure in $(1/4,1/2)$. By the coarea formula and (\ref{eqn:extra_small_area}), we have
\begin{equation}\label{eqn:extra_small_length}
    \limsup_k \sum_{D \in \mathcal{E}_{k,j}} \cH^1(D \cap \partial B_{\rho,j}) = 0
\end{equation}
for almost every $\rho \in (1/4, 1/2)$. Take $\rho \in (1/4, 1/2)$ satisfying (\ref{eqn:extra_apply_iso}) and (\ref{eqn:extra_small_length}). By the coarea formula and (\ref{eqn:extra_small_length}), we have
\begin{equation}\label{eqn:extra_corner}
    \limsup_k \sum_{D \in \mathcal{E}_{k,j}} \cH^0(D \cap \partial C_{\rho, \sigma,j}) = 0
\end{equation}
for $\sigma$ in a set of positive measure in $(\sigma_0/4, \sigma_0/2)$.

For $D \in \mathcal{E}_{k,j}$, let $B_D$ be the unique component of $D \setminus B_{\rho,j}$ containing $\partial D$. Note that $B_D$ is a punctured disk with boundary $\partial D \cup \bigcup_l \alpha_l$, where $\alpha_l$ are smooth Jordan curves in $\partial B_{\rho, j}$. Let $N_D$ be a smooth disk in $\overline{B}_{\rho_1, j} \setminus \overline{B}_{\rho, j}$ obtained by gluing $F_{\alpha_l}$ (the isoperimetric region in $\partial B_{\rho, j}$ for $\alpha_l$) to $B_D$ and then smoothing. This procedure is well-defined as long as $F_{\alpha_l}$ is glued and perturbed before $F_{\alpha_{l'}}$ if $F_{\alpha_l} \subset F_{\alpha_{l'}}$. Moreover, the resulting surface is a disk by construction.

We construct a competitor $\mathbf{S}_{k,j} \in \mathcal{D}(B_{\rho_1,j})[S_{k,j}]$ by replacing each $D \in \mathcal{E}_{k,j}$ by $N_D$. Since the surfaces $B_D$ are disjoint, we can choose the smoothing so that the surfaces $N_D$ are also disjoint, so $\mathbf{S}_{k,j}$ is well-defined.

We claim that we can replace $S_{k,j}$ by $\mathbf{S}_{k,j}$ in (\ref{eqn:ah_min_interior_regularity}) after replacing $\eps_k$ by some $\eps_{k,j}$ satisfying $\lim_k \eps_{k,j} = 0$. Note that $N_D$ is obtained from $D$ by (a) applying Lemma \ref{lem:area_comparison} to each component of $D \setminus B_{\rho,j}$ except $B_D$, (b) deleting everything in $\overline{B}_{\rho, j}$, and (c) turning $B_D$ into $N_D$ as described above. By Lemma \ref{lem:area_comparison} and (\ref{eqn:extra_small_area1}), the $\cA^h$ functional does not increase after (a). Since we only delete portions of the surface in (b), the area does not increase from (b). By (\ref{eqn:iso}), (\ref{eqn:prelim_isoperimetric}), and (\ref{eqn:extra_apply_iso}), we have (for $k$ large)
\begin{align*}
    \cH^3(\Lambda_D \cap B_{\rho,j})
    & \leq c_1(\cH^2(D \cap B_{\rho,j}) + \cH^2(\Lambda_D \cap \partial B_{\rho,j}))^{3/2}\\
    & \leq c_1(\cH^2(D \cap B_{\rho,j}) + \beta_0\cH^1(D \cap \partial B_{\rho,j})^2)^{3/2}
\end{align*}
Hence, we have
\begin{align}\label{eqn:vol_decay}
    \sum_{D \in \mathcal{E}_{k,j}} & \cH^3(\Lambda_D \cap B_{\rho,j})
    \leq c_1 \sum_{D \in \mathcal{E}_{k,j}}(\cH^2(D \cap B_{\rho,j}) + \beta_0\cH^1(D \cap \partial B_{\rho,j})^2)^{3/2}\notag\\
    & \leq c_1 \left(\sum_{D \in \mathcal{E}_{k,j}}(\cH^2(D \cap B_{\rho,j}) + \beta_0\cH^1(D \cap \partial B_{\rho,j})^2)\right)^{3/2}\\
    & \leq c_1\left(\sum_{D \in \mathcal{E}_{k,j}}\cH^2(D \cap B_{\rho,j}) + \beta_0\left(\sum_{D \in \mathcal{E}_{k,j}}\cH^1(D \cap \partial B_{\rho,j})\right)^2\right)^{3/2}\notag
\end{align}
and the right hand side goes to zero by (\ref{eqn:extra_small_area}) and (\ref{eqn:extra_small_length}). Therefore the volume change from (b) vanishes as $k \to \infty$. Finally, by (\ref{eqn:extra_small_length}) and (\ref{eqn:prelim_isoperimetric}), we have $\lim_k \sum_{\mathcal{E}_{k,j}} \cH^2(N_D \triangle B_D) = 0$, so the area added in (c) vanishes as $k \to \infty$. Hence, we have
\begin{equation}\label{eqn:ah_min_new_competitor}
    \cH^2(\mathbf{S}_{k,j}) \leq \cH^2(S_{k,j}') + T_h(\mathbf{S}_{k,j}, S_{k,j}') + \eps_{k,j}
\end{equation}
for all $S_{k,j}' \in \mathcal{D}(B_{\rho_1, j})[\mathbf{S}_{k,j}]$, where $\lim_k \eps_{k,j} = 0$.

Let $D_{k,j,i}' \in \mathcal{D}(B_{\rho,j})[D_{k,j,i}]$ be mutually disjoint for $i = 1,\hdots, n$. Let $\Omega_{k,j,i}$ be the open set in $B_{\rho_1,j}$ so that
\[ \partial [\Omega_{k,j,i}] \llcorner B_{\rho_1,j} = [D_{k,j,i}] \llcorner B_{\rho_1,j}. \]
Let $\Omega_{k,j,i}'$ be the analogous open set for $D_{k,j,i}'$. We claim that
\begin{equation}\label{eqn:stacked_disks}
    \sum_i \left(\cH^2(D_{k,j,i}) - \int_{\Omega_{k,j,i}}h\ d\cH^3\right)\leq \sum_i \left(\cH^2(D_{k,j,i}') - \int_{\Omega_{k,j,i}'}h\ d\cH^3\right) + \eps_{k,j}''
\end{equation}
for some $\lim_k \eps_{k,j}'' \to 0$.

Consider the competitor $S_{k,j}' \in \mathcal{D}(B_{\rho_1,j})[\mathbf{S}_{k,j}]$ given by replacing each $D_{k,j,i}$ by $D_{k,j,i}'$. This competitor is well-defined (i.e. there are no self-intersections) because $(\mathbf{S}_{k,j} \setminus \bigcup_i D_{k,j,i}) \cap B_{\rho, j} = \varnothing$ by definition, and $D_{k,j,i}' \setminus \overline{B}_{\rho, j}$ is a subset of $D_{k,j,i}$. Hence, (\ref{eqn:stacked_disks}) follows directly from (\ref{eqn:ah_min_new_competitor}).

We conclude with some notation. Let $B_{k,j,i}$ be the open set in $B_{\rho_1,j}$ so that
\[ \partial B_{k,j,i} \setminus D_{k,j,i} \subset \partial B_{\rho_1,j}\ \ \text{and}\ \  D_{\rho_1, \sigma_{k,j}, j}^- \subset \partial B_{k,j,i}. \]
If $B_{k,j,i} \subset B_{k,j,i'}$ for some $i \neq i'$, then we say $D_{k,j,i} < D_{k,j,i'}$. If $B_{k,j,i} = \Omega_{k,j,i}$ (as defined above), then we say $D_{k,j,i}$ is oriented up; otherwise, it is oriented down. By selecting a subsequence and relabeling, we can assume that
\[ D_{k,j,1} > D_{k,j,2} > \hdots > D_{k,j,n}, \]
and that $D_{k,j,i}$ is oriented up if and only if $D_{k',j,i}$ is oriented up for all $k'\neq k$. By the assumption that there is an open set $\Lambda_k$ with $\partial \Lambda_k \cap U = \Sigma_k \cap U$, we conclude that $D_{k,j,i}$ is oriented up if and only if $D_{k,j,i-1}$ and $D_{k,j,i+1}$ are oriented down (i.e.\ the orientation alternates with respect to the order of the disks).

\emph{Parity for Stacked Disks}. Let $\Omega_k$ and $\Omega$ be defined as in Theorem \ref{thm:compactness} with respect to the set $U$. Let $x \in \mathrm{spt}\|V\|$. We claim that
\begin{equation}\label{eqn:even_parity}
    \Theta^2(\|V\|, x) \equiv 0 \mod 2 \implies x \notin \partial^*\Omega,
\end{equation}
and
\begin{equation}\label{eqn:odd_parity}
    \Theta^2(\|V\|, x) \equiv 1 \mod 2 \implies x \in \overline{\partial^*\Omega}.
\end{equation}

Let $W$ be any neighborhood of $x$. Taking $j$ large enough, we have $B_{\rho_1, j} \subset W$. Let $\Omega_{k,j,i}$ be defined as above for the stacked disks $D_{k,j,i}$. Let $\Omega_{k,j}$ be the open subset of $B_{\rho_1,j}$ satisfying
\[ \partial[\Omega_{k,j}]\llcorner B_{\rho_1, j} = \sum_{i=1}^n [D_{k,j,i}] \llcorner B_{\rho_1, j}. \]
Let $\Omega^{j,i}$ and $\Omega^j$ be the limits of $\Omega_{k,j,i}$ and $\Omega_{k,j}$ in $k$ respectively, in the sense of Theorem \ref{thm:compactness}.

We first show that $\cH^3((\Omega^j \triangle \Omega)) \cap B_{\rho, j}) = 0$. Indeed, we have
\begin{align*}
    \cH^3((\Omega^j \triangle \Omega) \cap B_{\rho, j})
    & = \lim_k \cH^3((\Omega_{k,j} \triangle \Omega_k) \cap B_{\rho, j})\\
    & \leq \lim_k\cH^3(\beta_k) + \sum_{D \in \mathcal{E}_{k,j}} \cH^3(\Lambda_D \cap B_{\rho,j}) = 0,
\end{align*}
by Lemma \ref{lem:replacement}(\ref{eqn:replace_diff}) and (\ref{eqn:vol_decay}) (where $\beta_k$ are the bubbles called ``$B$'' in Lemma \ref{lem:replacement}). Hence, we have
\[ \partial^*\Omega \cap B_{\rho, j} = \partial^*\Omega^j \cap B_{\rho,j}, \]
so it suffices to work with $\Omega^j$.

Let $\omega^j := (\mu_{r_j} \circ \phi_x^{-1})(\Omega^j)$. Let $v_{j,i} := (\mu_{r_j} \circ \phi_x^{-1})_{\#}V_{j,i}$. By (\ref{eqn:spt}), we have
\[ \mathrm{spt}\left\|\sum_i v_{j,i}\right\| \cap b_{\rho} \subset b_{\rho} \cap \{-\sigma_j \leq x_3 \leq \sigma_j\}. \]
Let
\[ b_{\rho,j}^+ \coloneqq b_{\rho} \cap \{x_3 > \sigma_j\}\ \ \text{and}\ \  b_{\rho,j}^- \coloneqq b_{\rho} \cap \{x_3 < -\sigma_j\}. \]
Since $\overline{\partial^*\omega^j} \cap b_{\rho} \subset \mathrm{spt}\|\sum_iv_{j,i}\| \cap b_{\rho}$, we have either $b_{\rho,j}^+ \subset \omega^j$ or $b_{\rho,j}^+ \cap \omega^j = \varnothing$ (and analogously for $b_{\rho,j}^-$). We see that $D_{k,j,1}$ is oriented down if and only if $b_{\rho,j}^+ \subset \omega^j$. Similarly, $D_{k,j,n}$ is oriented up if and only if $b_{\rho,j}^- \subset \omega^j$. Note that $b_{\rho, j}^{\pm}$ contain a definite portion of $b_{2}^{\pm}$ (the upper and lower half balls of radius 2 in $T_xM$).

If $n \equiv 1 \mod 2$, then $\omega^j$ converges in $L^1$ as $j \to \infty$ (without loss of generality) to the upper half space $H^+ \subset T_xM$. Hence, every sufficiently small neighborhood $W$ of $x$ satisfies
\[ 0 < \cH^3(\Omega \cap W) < \cH^3(W), \]
so we have $x \in \overline{\partial^*\Omega}$.

If $n \equiv 0 \mod 2$, then $\omega^j$ converges in $L^1$ as $j \to \infty$ to $T_xM$ or $\varnothing$, so $x \notin \partial^*\Omega$ (i.e.\ we showed $x$ is not in the essential boundary of $\Omega$, which contains the reduced boundary (see \cite[Theorem 16.2]{maggi}).

\emph{$C^{1,\alpha}$ Regularity}. By Lemma \ref{lem:resolution_stacked_one}(\ref{eqn:resolution_stacked_one_vol1}) and (\ref{eqn:stacked_disks}), we have
\begin{equation}\label{eqn:weak_minimizing}
    \cH^2(D_{k,j,i}) \leq \cH^2(D_{k,j,i}') + c|T(D_{k,j,i}, D_{k,j,i'})| + \eps_{k,j}
\end{equation}
for any $D_{k,j,i}' \in \mathcal{D}(B_{\rho, j})[D_{k,j,i}]$.

We observe that everything we have proved thus far only requires that the function $h$ is measurable and satisfies $|h| \leq c$. We note that (\ref{eqn:weak_minimizing}) implies that $D_{k,j,i}$ is a minimizing sequence for the problem with prescribing function $-c + 2c\mathbf{1}_{\Omega_{j,i}}$. Hence, all of the above work applies to $V_{j,i}$. In particular, $V_{j,i}$ has $c$-bounded first variation and positive integer density in its support. By (\ref{eqn:mass_bds}), there is a neighborhood $W$ of $x$ so that $\Theta^2(V_{j,i}, y) = 1$ for all $y \in W\cap \mathrm{spt}\|V_{j,i}\|$.

By \cite[\S8]{allard}, we have $V_{j,i} \llcorner G(W,2) = \mathbf{v}(N_{j,i}, 1)$ for a $C^{1.\alpha}$ surface $N_{j,i} \subset W$. By (\ref{eqn:odd_parity}), we have $N_{j,i} \cap W = \partial^*\Omega^{j,i} \cap W$. Hence, $N_{j,i}$ is on one side of $N_{j,i'}$ for any $i,\ i'$, and each of these surfaces intersects tangentially at $x$. We conclude by (\ref{eqn:limit_sum_of_disks}).

\emph{Decomposition into 1-Stacks and 2-Stacks}. Take $x$ so that $h(x) \neq 0$. Suppose without loss of generality (by swapping orientations and taking $j$ sufficiently large) that $h\vert_{B_{\rho_1, j}} > 0$. We now decompose the stacked disk minimization problem (\ref{eqn:stacked_disks}) into smaller units, which we call 1-stacks and 2-stacks.

We first separate the 1-stacks. If $D_{k,j,1}$ is oriented up, then we show that
\begin{equation}\label{eqn:1-stacks}
    \cH^2(D_{k,j,1}) - \int_{\Omega_{k,j,1}}h\ d\cH^3  \leq \cH^2(D_{k,j,1}') - \int_{\Omega_{k,j,1}'}h\ d\cH^3 + \eps_{k,j}
\end{equation}
for all $D_{k,j,1}' \in \mathcal{D}(B_{\rho,j})[D_{k,j,1}]$, where $\Omega_{k,j,1}'$ is defined as above. The analogous statement holds if $D_{k,j,n}$ is oriented down by vertical reflection. The result follows immediately from Lemma \ref{lem:resolution_decompose_one}.

Now we separate the 2-stacks. If $D_{k,j,i'}$ is oriented down for some $1 \leq i' < n$, then we show that
\begin{equation}\label{eqn:2-stacks}
    \sum_{i=i'}^{i'+1} \left(\cH^2(D_{k,j,i}) -\int_{\Omega_{k,j,i}}h\ d\cH^3 \right) \leq \sum_{i=i'}^{i'+1} \left(\cH^2(D_{k,j,i}') - \int_{\Omega_{k,j,i}'}h\ d\cH^3\right) + \eps_k'
\end{equation}
for all mutually disjoint $D_{k,j,i}' \in \mathcal{D}(B_{\rho, j})[D_{k,j,i}]$ for $i = i',\ i'+1$, where $\Omega_{k,j,i}'$ is defined as above. The result follows immediately from Lemma \ref{lem:resolution_decompose_two}.

\emph{Regularity of 1-Stacks}. Suppose (\ref{eqn:1-stacks}) holds\footnote{We note that (\ref{eqn:1-stacks}) holds at any point in $V$ of density 1, as well as in the case $D_{k,j,1}$ is oriented up.}. Recall that we showed $\Theta^2(V_{j,1}, y) = 1$ for all $y \in W \cap \mathrm{spt}\|V_{j,1}\|$ and
\[ \mathrm{spt}\|V_{j,1}\| \cap W = \partial^*\Omega^{j,1} \cap W. \]
Hence, we have
\[ \|V_{j,1}\| \llcorner W = |D\mathbf{1}_{\Omega^{j,1}}| \llcorner W. \]
Inserting this equality into Theorem \ref{thm:compactness}(\ref{eqn:first_variation_formula}), we conclude that
\[ H_{V_{j,1}}\vert_{W} = h\nu_{\partial^*\Omega^{j,1}}\vert_{W}, \]
which is a $C^{1,\alpha}$ vector field. Applying the regularity theorem of Allard \cite[\S8]{allard} iteratively, we conclude that $N_{j,1}$ is a smooth surface with prescribed mean curvature $h$ in $W$ with respect to the set $\Omega^{j,1}$. Stability follows from smooth embeddedness and (\ref{eqn:variational}).

\emph{Regularity of 2-Stacks}. 
Suppose $D_{k,j,i'}$ is oriented down for $1 \leq i' < n$, so (\ref{eqn:2-stacks}) holds. For ease of notation, let $i' = 1$.

We show the mean convexity of $V_{j,1}$. Let $X$ be a smooth vector field with compact support in $B_{\rho, j}$, and let $\Phi_t$ be the flow of $X$. Without loss of generality, we suppose that
\[ \frac{d}{dt}\Big|_{t = 0} \|(\Phi_t)_{\#}V_1\|(B_{\rho, j}) < 0, \]
(i.e.\ either the derivative is zero, or we can swap $X \mapsto -X$).
By Lemma \ref{lem:resolution_stacked_one}(\ref{eqn:resolution_stacked_one_vol2}) and (\ref{eqn:2-stacks}), we have (using $h \geq 0$)
\[
    \cH^2(D_{k,j,1} \cap B_{\rho, j}) \leq \cH^2(\Phi_t(D_{k,j,1})\cap B_{\rho, j}) + c\cH^3(\Omega_{k,j,1} \setminus \Phi_t(\Omega_{k,j,1})) + \eps_k''.
\]
Following the proof of Theorem \ref{thm:compactness}, we have
\begin{equation}\label{eqn:2-stack_bdd_first_variation1}
    \|V_{j,1}\|(B_{\rho, j}) \leq \|(\Phi_t)_{\#}V_{j,1}\|(B_{\rho,j}) + c\cH^3(\Omega^{j,1} \setminus \Phi_t(\Omega^{j,1})).
\end{equation}
Hence, we have
\begin{equation}\label{eqn:2-stack_bdd_first_variation2}
    0 > \frac{d}{dt}\Big|_{t=0} \|(\Phi_t)_{\#}V_{j,1}\|(B_{\rho, j}) \geq -c\frac{d}{dt}\Big|_{t=0^+} \cH^3(\Omega^{j,1} \setminus \Phi_t(\Omega^{j,1})).
\end{equation}
By \cite[Proposition 17.8]{maggi}, we have
\[ \frac{d}{dt}\Big|_{t=0^+} \cH^3(\Omega^{j,1} \setminus \Phi_t(\Omega^{j,1})) \leq \int_{\partial^*\Omega_1} (X \cdot \nu_{\partial^*\Omega^{j,1}})_-\ d\cH^2. \]
Hence, we observe that the first variation cannot be negative if $X \cdot \nu_{\partial^*\Omega^{j,1}} \geq 0$, which proves the desired weak mean convexity. By vertical reflection, the same holds for $V_{j,2}$.

Suppose $\mathrm{spt}\|V_{j,1}\| \cap W' = \mathrm{spt}\|V_{j,2}\| \cap W'$ for some open set $W' \subset W$. As shown above, $\Theta^2(V_{j,1} + V_{j,2}, x) = 2$ for all $x \in \mathrm{spt}\|V_{j,1} + V_{j,2}\| \cap W'$. By (\ref{eqn:even_parity}), we have $\partial^*(\Omega^{j,1} \cup \Omega^{j,2}) \cap W' = \varnothing$, so in fact $\overline{\partial^*(\Omega^{j,1} \cup \Omega^{j,2})} \cap W' = \varnothing$. Since we already showed $\mathrm{spt}\|V_{j,i}\| \cap W' = \partial^*\Omega^{j,i} \cap W'$, we conclude that $\cH^3((\Omega^{j,1} \cup \Omega^{j,2}) \cap W') = \cH^3(W')$. By Theorem \ref{thm:compactness}(\ref{eqn:variational}), $V_{j,1} + V_{j,2}$ is stationary in $W'$. Then by Allard's regularity theorem \cite[\S8]{allard},
\[ (V_{j,1} + V_{j,2}) \llcorner G(W',2) \cap \mathbf{v}(N, 2), \]
where $N \subset W'$ is a smooth minimal surface. Hence, $N_{j,1} \cap W' = N_{j,2} \cap W' = N$. Stability follows from smooth embeddedness and Theorem \ref{thm:compactness}(\ref{eqn:variational}).

\emph{Maximum Principle}. 
First, suppose $\Theta^2(\|V\|, x) = 2n + 1$. Then (without loss of generality) $D_{k,j,1}$ is oriented up. By the regularity of 1-stacks, $V_{j,1}$ is the varifold associated to a smooth surface $N_{j,1}$ with mean curvature $h > 0$ pointing up. Suppose for contradiction that $n \geq 1$. Then $V_{j,2}$ is the varifold associated to a $C^{1,\alpha}$ surface $N_{j,2}$ with generalized mean curvature pointing down. Moreover, $N_{j,2}$ is below $N_{j,1}$ and intersects $N_{j,1}$ tangentially at $x$, which contradicts the strong maximum principle (see Lemma \ref{lem:maximum_principle}). Hence, if $D_{k,j,1}$ is oriented up, then $n = 1$ and $V$ is the varifold of a smooth surface of prescribed mean curvature $h$ pointing up in a neighborhood of $x$.

Second, suppose $\Theta^2(\|V\|, x) = 2n$ with $n \geq 2$. By the above paragraph, $V$ consists of $n$ 2-stacks intersecting tangentially and on one side at $x$. Then $N_{j, 2}$ and $N_{j, 2n - 1}$ are $C^{1,\alpha}$ surfaces on one side of each other, intersecting tangentially at $x$, and with generalized mean curvature pointing away from each other. By the strong maximum principle (see Lemma \ref{lem:maximum_principle}) and the regularity of 2-stacks, there is a neighborhood $W''$ of $x$ and a minimal surface $N \subset W''$ satisfying $N_{j, i} \cap W'' = N$ for all $i$. Hence, $\Theta^2(\|V\|, y) = 2n$ for all $y \in W'' \cap \mathrm{spt}\|V\|$.
\end{proof}

\section{Improvement from Free Boundary Problems}\label{sec:freeboundary}
In this section, we use some results from the theory of free boundary problems--in particular, the recent result of \cite{WZ_multiple}--to obtain the optimal regularity of the limit surfaces in Theorem \ref{thm:local_regularity}.

For the remainder of this section, let $V$ be as in Theorem \ref{thm:local_regularity}, and $x\in\mathrm{spt}\|V\|$. Also, let $W_x$ be the open neighborhood of $x$ provided by Theorem \ref{thm:local_regularity}, and let $N_1, N_2, \dots, N_{n_x}$ be the corresponding $C^{1,\alpha}$ surfaces meeting tangentially at $x$. 

The main result of this section is the following.

\begin{proposition}\label{prop:C1,1_regularity}
The surfaces $N_1, \dots, N_{n_x}$ are $C^{1,1}$ in the interior of $W_x$.
\end{proposition}

\begin{proof}
We shall follow the setup in \cite[\S 1.2]{WZ_multiple}.

By Theorem \ref{thm:local_regularity}, $V \llcorner G(W_x,2) = \sum_{l=1}^{n_x}\mathbf{v}(N_l,1)$, where $N_i$ are distinct $C^{1,\alpha}$ hypersurfaces with $\|h\|_{L^\infty(W_x)}$-bounded generalized mean curvature, such that each $N_l$ is on one side of $N_{l'}$ intersecting tangentially at $x$. Let then $n\coloneqq n_x$, and let $P\subset T_xM$ be the common tangent plane of $N_1, N_2, \dots, N_n$ at $x$. 

By shrinking $W_x$ if necessary, we can assume the surfaces $N_i$ can be represented as weakly ordered graphs over the ball $B^P_r(0)\subset P$ for some small enough $r>0$. By Allard's regularity theorem (see the proof of Theorem \ref{thm:local_regularity}), these graphs satisfy a uniform $C^{1,\alpha}$ bound.

By \cite[Remark 1.5]{WZ_multiple}, after shrinking $W_x$ again if necessary, there is an embedded minimal surface $\Gamma\subset W_x$ through $x$ with boundary in $\partial W_x$, so that we can rewrite $N_1, N_2, \dots, N_n$ as weakly ordered normal graphs over $\cB_{r'}(x)\subset\Gamma$ with uniformly bounded $C^{1,\alpha}$ norm, where $\cB_{r'}(x)$ denotes the geodesic ball in $\Gamma$ with center $x$ and radius $r'$. In particular, after choosing a unit normal field on $\Gamma$, which trivializes the normal bundle $N\Gamma$, there are $C^{1,\alpha}$ functions $u_i:\cB_{r'}(x)\to \R$ such that 
\[u_1\leq u_2\leq\dots\leq u_n,\]
and 
\[N_i=\mathrm{graph}\,u_i\]
for $i=1, \dots, n$.

Notice that, when restricted to $W_x$, the varifold $V$ is induced by the union of graphs of the functions $u_i$, and the Caccioppoli set $\Omega$ has boundary given by $\sum_{i=1}^n(-1)^{i-1}[\mathrm{graph}\, u_i]$ in the sense of currents. By the proof of Theorem \ref{thm:local_regularity}, the pair $(V, \Omega)$ is $\cA^h$-stationary in $W_x$ in the sense of \cite{WZ_multiple}.

Therefore, the first conclusion in \cite[Theorem 1.4]{WZ_multiple} applies, and this ends the proof of the proposition.
\end{proof}

As mentioned in the Introduction, for the curvature estimates in \S \ref{sec:curvature_estimates} we are going to need a $C^{1,1}$ estimate for the functions $u_1, \dots, u_n$. This also follows from \cite[Theorem 1.4]{WZ_multiple}, after a simple rescaling.
Indeed, after rescaling the metric in $W_x$ by a constant factor, we can assume the rescaled graph functions $\tilde{u}_i$ are defined on the ball $\tilde{\cB}_{2}(x)$ of radius 2, where for the rest of this section we shall use a tilde to denote quantities and balls with respect to the rescaled metric. Similarly, the prescribing function must be rescaled appropriately, so that 
the corresponding pair $(\tilde{V}, \tilde{\Omega})$ is $\cA^{\tilde{h}}$-stationary.
By these rescalings, we can assume that  
\[\|\tilde{u}_i\|_{C^{1,\alpha}(\tilde{\cB}_2(x))} < 1 \]
for all $i=1, \dots, n$, and 
\[\|\tilde{h}\|_{C^1(\tilde{\cB}_1(x)\times (-1,1))} < 1,\]
where $\tilde{\cB}_1(x)\times (-1,1))$ denotes the width 1 tubular neighborhood of $\tilde{\cB}_1(x)$ in Fermi coordinates with respect to $\Gamma$ in the rescaled metric.
Again, by Theorem \ref{thm:local_regularity} each graph has $\|\tilde{h}\|_{L^\infty(\tilde{\cB}_1(x)\times (-1,1))}$-bounded first variation. Therefore, the second conclusion of \cite[Theorem 1.4]{WZ_multiple} applies, and we obtain the following $C^{1,1}_{\mathrm{loc}}$ estimate:

\begin{corollary}\label{cor:estimate}
\begin{equation}
\sum_{i=1}^n\|\tilde{u}_i\|_{C^{1,1}(\tilde{\cB}_{1/2}(x))}\leq C\left(\sum_{i=1}^n\|\tilde{u}_i\|_{C^{1,\alpha}(\tilde{\cB}_1(x)}+\|\tilde{h}\|_{C^1(\tilde{\cB}_1(x)\times (-1,1))}\right)
\end{equation} 
for some constant $C \geq 1$ depending only on $n, \alpha$, and the metric $g$ in $W_x$.
\end{corollary}
\noindent Naturally, this implies a $C_\mathrm{loc}^{1,1}$-estimate for the original functions $u_i$.

If $h(x)\neq 0$, we can prove a more precise regularity result. 
By Theorem \ref{thm:local_regularity}, if $h(x)\neq 0$, then $n_x\in\{1\}\cup 2\N$ and
\begin{itemize}
    \item if $n_x=1$, then $V \llcorner G(W_x,2) = \mathbf{v}(N, 1)$, where $N \subset W_x$ is a smooth stable surface with prescribed mean curvature $h$;
    \item if $n_y = 2n$ for all $y \in \mathrm{spt}\|V\| \cap W_x'$ for an open set $W_x' \subset W_x$ containing $x$, then $V \llcorner G(W'_x,2) = \mathbf{v}(N, 2n)$, where $N\subset W_x'$ is a smooth stable minimal surface.
\end{itemize}
Hence, in both cases, the support of $V$ in a neighbourhood of $x$ coincides with a smooth surface.

We now deal with the case where $n_x = 2n$, and there is a sequence $x_j \to x$ with $x_j \in \mathrm{spt}\|V\|$ and $n_{x_j} \neq 2n$. By Theorem \ref{thm:local_regularity}, then $n=1$ and the surfaces $N_1,N_2\subset W_x$ intersect tangentially at $x$. Let $T$ denote the touching set $T=N_1\cap N_2\subset W_x$ and let $\Gamma$ be its boundary $\partial T\cap W_x$ in $W_x$.

\begin{proposition}
The boundary $\Gamma$ of the touching set $T$ is locally contained in a $C^{1,\alpha}$ curve and satisfies $\cH^1(\Gamma)<+\infty$
\end{proposition}

\begin{proof}
This follows directly from the results in Section 4 of \cite{SilvestreTM}. Indeed, the difference $\tilde{u}_1-\tilde{u}_2$ is a solution to an obstacle problem to which the  results of Caffarelli (\cite{Caf77}, \cite{Caf80} and \cite{Caf98}) apply. 
\end{proof}

\part{Minimization over isotopy classes}\label{part:isotopy}
In this part we follow the arguments of \cite{MSY}, \cite{Colding_DeLellis}, and \cite{DeLellis_Pell}, upgrading from the regularity of minimizers of disks to the regularity of minimizers over isotopy classes.

\section{Minimization Problem}
Let $(M, g)$ be a smooth $3$-dimensional Riemannian manifold. Let $h : M \to \R$ be a smooth real valued function, and let $c \coloneqq \sup_M |h|$.

As before, by rescaling $M$, we can assume that the conditions of \S\ref{sec:local_control} hold with $\rho_0 = 1$.

Let $U \subset M$ be an open set with $C^1$ boundary.

Let $\mathcal{O}(U)$ be the set of open sets $\Omega \subset M$ with smooth boundary so that $\partial \Omega$ has transverse (or empty) intersection with $U$.

In this part, we consider the following question.

\begin{question}\label{question:main_isotopy}
Consider a sequence $\{\Omega_k\}_{k \in \N} \subset \mathcal{O}(U)$ and $\Sigma_k := \partial \Omega_k \cap U$ satisfying $\sup_k \cH^2(\Sigma_k) < \infty$, $\sup_k \mathrm{genus}(\Sigma_k) < \infty$, and
\begin{equation}\label{eqn:min_seq_setup_isotopy}
    \cA^h(\Omega_k) \leq \cA^h(\phi(1, \Omega_k)) + \eps_k
\end{equation}
for all $\phi \in \mathcal{I}(U)$ with $\eps_k \to 0$. Does (a subsequence of) $\Sigma_k := \partial \Omega_j \cap U$ converge to a regular limit with prescribed mean curvature $h$?
\end{question}

We first observe that all of the conclusions of Theorem \ref{thm:compactness} hold, where $\mathbf{v}(\Sigma_k) \rightharpoonup V$ and $\Omega_k \cap U \to \Omega$, by following the same proof.

\section{Gamma-Reduction}
The key insight of \cite{MSY} is the introduction of a procedure, called $\gamma$-reduction, to cut away small necks that either create genus or separate large components. This procedure enables the reduction of the full isotopy minimization problem to the disk minimization problem, solved in Part \ref{part:disks} above.

\subsection{Definition of gamma-reduction}
Let $\delta > 0$ be the constant from Lemma \ref{lem:thin_isoperimetric_MSY}, and let $0 < \gamma < \delta^2/9$. Let $U \subset M$ be an open set, and let $\Sigma \subset U$ be a smooth embedded oriented surface with $\partial \Sigma \subset \partial U$. Let $\mathcal{C}(\Sigma, U)$ be the set of smooth embedded oriented surfaces $\Sigma' \subset U$ with $\partial [\Sigma'] = \partial [\Sigma]$.

\begin{definition}\label{def:gamma_reduction}
For $\Sigma_1,\ \Sigma_2 \in \mathcal{C}(\Sigma, U)$, we write
\[ \Sigma_2 \overset{(\gamma, U)}{\ll} \Sigma_1 \]
and say $\Sigma_2$ is a \emph{$\gamma$-reduction} of $\Sigma_1$ if the following hold:
\begin{enumerate}
    \item $\Sigma_2$ is obtained from $\Sigma_1$ by surgery in $U$; namely,
    \begin{itemize}
        \item $\overline{\Sigma_1 \setminus \Sigma_2} =: A$ is diffeomorphic to the standard closed annulus,
        \item $\overline{\Sigma_2 \setminus \Sigma_1} =: D_1 \cup D_2$ is diffeomorphic to the disjoint union of two standard closed disks,
        \item there is an open set $Y \subset U$ homeomorphic to the standard 3-ball with $\partial Y = A \cup D_1 \cup D_2$ and $Y \cap (\Sigma_1 \cup \Sigma_2) = \varnothing$.
    \end{itemize}
    \item $\cH^2(A) + \cH^2(D_1) + \cH^2(D_2) < 2\gamma$.
    \item If $\Gamma$ is the connected component of $\Sigma_1$ containing $A$, then for each component of $\Gamma \setminus A$ we have either
    \begin{itemize}
        \item it is a disk of area at least $\delta^2/2$,
        \item it is not simply connected.
    \end{itemize}
\end{enumerate}
\end{definition}

\begin{definition}\label{def:strong_gamma_reduction}
For $\Sigma_1,\ \Sigma_2 \in \mathcal{C}(\Sigma, U)$, we write
\[ \Sigma_2 \overset{(\gamma, U)}{<} \Sigma_1 \]
and say $\Sigma_2$ is a \emph{strong $\gamma$-reduction} of $\Sigma_1$, if there is an isotopy $\psi \in \mathcal{I}(U)$ satisfying
\begin{enumerate}
    \item $\Sigma_2 \overset{(\gamma, U)}{\ll} \psi(1, \Sigma_1)$,
    \item $\cH^2(\psi(1, \Sigma_1) \triangle \Sigma_1) < \gamma$.
\end{enumerate}
\end{definition}

We say a surface is \emph{$(\gamma, U)$-irreducible} (resp.\ \emph{strongly $(\gamma, U)$-irreducible}) if it admits no $(\gamma, U)$ reduction (resp.\ strong $(\gamma, U)$-reduction).

\subsection{Structure of gamma-irreducible surfaces}
We adapt \cite[Theorem 2]{MSY} to our setting, allowing us to decompose $\gamma$-irreducible surfaces in small balls into a disjoint union of disks.

\begin{theorem}\label{thm:gamma_reduction}
Suppose $\cH^3(U) \leq \eta/8$. Let $U' \subset\subset U$ be an open set so that $\overline{U}'$ is diffeomorphic to $\overline{\B}$. Suppose $\Omega \subset M$ satisfies
\begin{itemize}
    \item
        \[ \cA^h(\Omega) - \inf_{\phi \in \mathcal{I}(U)}\cA^h(\phi(1,\Omega)) =: E(\Omega) \leq 3\gamma/16. \]
    \item $\Sigma \coloneqq \partial \Omega \cap U$ is strongly $(\gamma, U)$-irreducible,
    \item $\Sigma$ has transverse intersection with $\partial U'$.
\end{itemize}
Let $\Gamma_1, \hdots, \Gamma_q$ be the components of $\Sigma \cap \partial U'$, and let $F_j$ denote the isoperimetric region in $\partial U'$ with boundary $\Gamma_j$. Suppose further that
\[ \sum_{j=1}^q \cH^2(F_j) \leq \gamma/8.  \]
Let $\Sigma_0$ denote the union of all closed components $\Lambda$ of $\Sigma$ contained an an open set $K_{\Lambda} \subset U$ diffeomorphic to $\B$ so that $\partial K_{\Lambda} \cap \Sigma = \varnothing$.

Then $\cH^2(\Sigma_0) \leq 2E(\Omega)$, and there are disjoint closed disks $D_1, \hdots$, $D_p$ satisfying
\begin{enumerate}
    \item $D_j \subset \Sigma \setminus \Sigma_0$,
    \item $\partial D_j \subset \partial U'$,
    \item
        \[ \left(\bigcup_{j=1}^p D_j\right) \cap U' = (\Sigma \setminus \Sigma_0) \cap U', \]
    \item there is an isotopy $\phi \in \mathcal{I}(U)$ such that $\phi(t,x) = x$ for all $(t, x) \in [0,1] \times W$, where $W$ is a neighborhood of $(\Sigma \setminus \Sigma_0) \setminus \bigcup_{j=1}^p(D_j \setminus \partial D_j)$, so that
        \[ \bigcup_{j=1}^p (\phi(1, D_j) \setminus \partial D_j) \subset U' \]
    \item
        \[ \sum_{j=1}^p (\cH^2(D_j) - c\cH^3(\Omega^j)) \leq \sum_{j=1}^q \cH^2(F_j) + E(\Omega), \]
        where $\Omega^j \subset U$ is the region bounded by $D_j \cup F_{\partial D_j}$, which exists by the existence of the isotopy in (4).
\end{enumerate}
\end{theorem}

\begin{remark}\label{rem:bubbles}
Let $K_1, \hdots, K_s \subset U$ be open 3-balls with $\Sigma_0 \subset \bigcup_{j=1}^s K_j$ and $\partial K_j \cap \Sigma = \varnothing$.

For each $\eps > 0$, there is an open set $\bigcup_{j=1}^s K_j \subset W \subset U$ and an isotopy $\phi\in \mathcal{I}(U)$ satisfying
\begin{itemize}
    \item $W \cap (\Sigma \setminus \Sigma_0) = \varnothing$,
    \item $\phi(t, x) = x$ for all $(x,t) \in (M \setminus W) \times [0, 1]$,
    \item $\phi(t, W) \subset W$,
    \item $\cH^2(\phi(1, \Sigma_0)) < \eps$,
    \item the union of the regions bounded by each component of $\phi(1, \Sigma_0)$ (given by Lemma \ref{lem:thin_isoperimetric_MSY}) has volume bounded by $\eps$,
    \item and each component of $\phi(1, \Sigma_0)$ has diameter less than $\eps$.
\end{itemize}
Let $\tilde{\Omega}$ be the region given by deleting $\Sigma_0$ in $\Sigma$. Each component that we collapse decreases $\cA^h$ by Proposition \ref{prop:iso+PMC_controls_area}, and we have
\[ \cH^2(\Sigma_0) \leq 2E(\Omega) \]
and
\[ \cA^h(\tilde{\Omega}) - \inf_{\phi \in \mathcal{I}(U)} \cA^h(\phi(1,\tilde{\Omega})) \leq E(\Omega). \]
Hence, the general case of Theorem \ref{thm:gamma_reduction} follows from the special case where $\Sigma_0 = \varnothing$, which we henceforth assume in the proof below.
\end{remark}

\begin{remark}\label{rem:ac_to_vol}
Since $\Omega^j \subset U$ and $\cH^3(U) \leq \eta$, we have by Proposition \ref{prop:iso+PMC_controls_area} that
\[ \cH^3(\Omega^j) \leq c_1^{2/3}\eta^{1/3}(\cH^2(D_j) + \cH^2(F_{\partial D_j})). \]
By our choice of $\eta$, we have
\[ c\cH^3(\Omega^j) \leq \frac{1}{2}(\cH^2(D_j) + \cH^2(F_{\partial D_j})). \]
Then conclusion (5) of Theorem \ref{thm:gamma_reduction} implies
\begin{equation}\label{eqn:area_gamma_red}
    \sum_{j=1}^p \cH^2(D_j) \leq 3\sum_{j=1}^q \cH^2(F_j) + 2E(\Omega).
\end{equation}
\end{remark}

\begin{proof}[Proof of Theorem \ref{thm:gamma_reduction}]
The proof follows exactly as in \cite[Theorem 2]{MSY}, with the additional step of keeping track of volumes.

We induct on $q$, where the case $q = 0$ is trivial.

Assume $\Omega$ satisfies
\begin{itemize}
    \item $\Sigma_0 = \varnothing$,
    \item $\sum_{j=1}^q \cH^2(F_j) \leq \gamma/8$,
    \item $E(\Omega) \leq 3\gamma/8 - \frac{3}{2}\sum_{j=1}^q \cH^2(F_j)$,
    \item $\Sigma$ is strongly $\tilde{\gamma}$ irreducible for $\tilde{\gamma} = \gamma/4 + 4\sum_{j=1}^q \cH^2(F_j) + \frac{4}{3}E(\Omega)$.
\end{itemize}

Assume as an inductive hypothesis that the conclusions of the theorem are true for $\hat{\Omega} = \phi(1,\Omega)$ for some $\phi \in \mathcal{I}(U)$ if it satisfies the above assumptions with $q-1$ in space of $q$ and $\hat{\Omega}$ in place of $\Omega$.

Without loss of generality (by relabeling), we have $F_q \cap \Gamma_j = \varnothing$ for all $j \neq q$. Since $\Sigma$ is strongly $\tilde{\gamma}$-irreducible and $\cH^2(F_q) < \tilde{\gamma}$, there is a disk $D \subset \Sigma$ so that $\partial D = \Gamma_q$ and $\cH^2(D) < \delta^2/2$. Then $D \cup F_q$ is homeomorphic to $S^2$ and $\cH^2(D \cup F_q) <  \delta^2/2 + \delta^2/2 = \delta^2$, so Lemma \ref{lem:thin_isoperimetric_MSY} implies that $D \cup F_q = \partial K$ for $K \subset U$ an open 3-ball. Let $\Lambda$ be the component of $\Sigma$ containing $D$. Since $\Sigma_0 = \varnothing$ by assumption, we have $\Lambda \setminus D \not\subset K$. Hence, we have $(\Sigma \setminus D) \cap K = \varnothing$ (because $\Sigma \setminus D$ does not intersect $F_q$).

Let $\Sigma_* = (\Sigma \setminus D) \cup F_q$ and $\Omega_*$ the corresponding region. As in \cite[Theorem 2]{MSY}, for every $\eps > 0$, there is a continuous isotopy $\alpha$ taking $\Sigma_*$ to a smooth surface $\hat{\Sigma}_*$ and $\Omega_*$ to $\hat{\Omega}_*$ satisfying
\begin{itemize}
    \item $\hat{\Omega}_* = \phi(1, \Omega)$ for some $\phi \in \mathcal{I}(U)$,
    \item $\hat{\Sigma}_* \cap \partial U' = \bigcup_{j=1}^{q-1} \Gamma_j$,
    \item $\cH^2(\hat{\Sigma}_* \triangle \Sigma) < \cH^2(D) + \cH^2(F_q) + \eps$,
    \item $\cH^3(\hat{\Omega}_* \triangle \Omega) < \cH^3(K) + \eps/(2c)$,
    \item $\cH^2(\hat{\Sigma}_*) < \cH^2(\Sigma) + \cH^2(F_q) - \cH^2(D) + \eps/2$,
    \item $\cH^2(\hat{\Sigma}_*) \geq \cH^2(\Sigma_*)$.
\end{itemize}
Using $\cH^3(U) \leq \eta/8$, Proposition \ref{prop:iso+PMC_controls_area} implies $c\cH^3(K) \leq \frac{1}{4}(\cH^2(F_q) + \cH^2(D))$. We assume $\eps \leq \cH^2(F_q)/4$. Observe that since $\Omega$ and $\hat{\Omega}_*$ are in the same isotopy class, we have
\begin{align*}
    E(\hat{\Omega}_*) & < E(\Omega) + c\cH^3(K) + \cH^2(F_q) - \cH^2(D) + \eps\\
    & \leq E(\Omega) + \frac{3}{2}\cH^2(F_q) - \frac{3}{4}\cH^2(D).
\end{align*}
By \cite[Remark 3.11]{MSY}, we have that $\hat{\Sigma}_*$ is strongly $\gamma^*$-irreducible, where
\begin{align*}
    \gamma^* & = \gamma/4 + 4\sum_{j=1}^q \cH^2(F_j) + \frac{4}{3}E(\Omega) - \cH^2(D) - 2\cH^2(F_q)\\
    & = \gamma/4 + 4\sum_{j=1}^{q-1} \cH^2(F_j) + \frac{4}{3}E(\Omega) + 2\cH^2(F_q) - \cH^2(D)\\
    & \geq \gamma/4 + 4\sum_{j=1}^{q-1} \cH^2(F_j) + \frac{4}{3}E(\hat{\Omega}_*).
\end{align*}
Furthermore, we have
\begin{align*}
    E(\hat{\Omega}_*) & \leq E(\Omega) + \frac{3}{2}\cH^2(F_q)\\
    & \leq 3\gamma/8 - \frac{3}{2}\sum_{j=1}^q \cH^2(F_j) + \frac{3}{2}\cH^2(F_q)\\
    & = 3\gamma/8 - \frac{3}{2}\sum_{j=1}^{q-1} \cH^2(F_j).
\end{align*}
Hence, $\hat{\Omega}_*$ satisfies the inductive hypothesis. Then there are pairwise disjoint disks $\tilde{\Delta}_1, \hdots, \tilde{\Delta}_p$ satisfying
\begin{itemize}
    \item $\tilde{\Delta}_j \subset \hat{\Sigma}_*$,
    \item $\partial \tilde{\Delta}_j \subset \partial U'$,
    \item
        \[ \left(\bigcup_{j=1}^p \tilde{\Delta}_j\right) \cap U' = \hat{\Sigma}_* \cap U', \]
    \item there is an isotopy $\tilde{\psi}\in \mathcal{I}(U)$ such that $\tilde{\psi}(t, x) = x$ for all $(t, x) \in [0,1] \times \tilde{W}$, where $\tilde{W}$ is a neighborhood of $\hat{\Sigma}_* \setminus \bigcup_{j=1}^p(\tilde{\Delta}_j \setminus \partial \tilde{\Delta}_j)$, so that
        \[ \bigcup_{j=1}^p (\tilde{\psi}(1, \tilde{\Delta}_j) \setminus \partial \tilde{\Delta}_j) \subset U', \]
    \item
        \[ \sum_{j=1}^p (\cH^2(\tilde{\Delta}_j) - c\cH^3(\tilde{\Omega}^j)) \leq \sum_{j=1}^{q-1} \cH^2(F_j) + E(\hat{\Omega}_*). \]
\end{itemize}
Letting $\Delta_j \coloneqq (\alpha^{-1})(1, \tilde{\Delta}_j)$ (i.e.\ the time-reversal of the smoothing isotopy $\alpha$ for $\Sigma_*$), and taking $\psi \coloneqq \tilde{\psi} * \alpha$ (i.e.\ the concatenation of isotopies), we have that
\begin{itemize}
    \item $\Delta_j \subset \Sigma_*$,
    \item $\partial \Delta_j \subset \partial U'$,
    \item
        \[ \left(\bigcup_{j=1}^p \Delta_j\right) \cap U' = \Sigma_* \cap U', \]
    \item there is an isotopy $\psi\in \mathcal{I}(U)$ such that $\psi(t, x) = x$ for all $(t, x) \in [0,1] \times W'$, where $W'$ is a neighborhood of $\Sigma_* \setminus \bigcup_{j=1}^p(\Delta_j \setminus \partial \Delta_j)$, so that
        \[ \bigcup_{j=1}^p (\psi(1, \Delta_j) \setminus \partial \Delta_j) \subset U', \]
    \item
        \[ \sum_{j=1}^p (\cH^2(\Delta_j) - c\cH^3(\tilde{\Omega}^j)) \leq \sum_{j=1}^{q-1} \cH^2(F_j) + E(\hat{\Omega}_*). \]
\end{itemize}

Let $\beta$ be a continuous isotopy so that $\beta(t, K) \subset K$, $\beta\vert_{\{t\} \times \Sigma \setminus D} = \mathbf{1}_{\Sigma \setminus D}$, and $\beta(1, D) = F_q$. Consider the following two cases:

\emph{Case 1}: $F_q \subset \bigcup_{j=1}^p \Delta_j$. Select the disks $D_1, \hdots, D_p$ by taking $D_{j_0} = (\Delta_{j_0} \setminus F_q) \cup D$ for the unique $j_0$ so that $F_q \subset \Delta_{j_0}$, and $D_j = \Delta_j$ for $j \neq j_0$. Define an isotopy by $\tilde{\phi} = \psi * \beta$. By smoothing $\tilde{\phi}$, we obtain the desired isotopy $\phi$. We compute
\begin{align*}
    \sum_{j=1}^p (\cH^2(D_j) - c\cH^3(\Omega^j))
    & \leq \sum_{j=1}^p (\cH^2(\Delta_j) - c\cH^3(\tilde{\Omega}^j))
    + \cH^2(D) - \cH^2(F_q) - c\cH^3(K) + \eps\\
    & \leq \sum_{j=1}^{q-1} \cH^2(F_j) + E(\hat{\Omega}_*)
    + \cH^2(D) - \cH^2(F_q) - c\cH^3(K) + \eps\\
    & \leq \sum_{j=1}^{q-1} \cH^2(F_j) + E(\Omega) + 2\eps\\
    & \leq \sum_{j=1}^{q} \cH^2(F_j) + E(\Omega) + 2\eps.
\end{align*}
Since $\eps > 0$ is arbitrary, the desired inequality holds.

\emph{Case 2}: $F_q \not\subset \bigcup_{j=1}^p \Delta_j$. Define the disks $D_1, \hdots, D_{p+1}$ by $D_j = \Delta_j$ for $j < p+1$ and $D_{p+1} = D$. Define an isotopy $\tilde{\phi} = \hat{\beta} * \psi * \beta$, where $\hat{\beta} \in \mathcal{I}(U)$ satisfies $\hat{\beta}(t, x) = x$ for all $(x, t) \in (\Sigma \setminus D) \times [0,1]$ and $\hat{\beta}(1, F_q)$ is a disk $\tilde{D} \subset U'$ with $\partial \tilde{D} = \partial D$, $\tilde{D} \cap \partial U' = \partial \tilde{D}$, and $\tilde{D} \cap \psi(t, \Sigma_*) = \Gamma_q$ for all $t \in [0, 1]$. We claim that in this case there is a neighborhood $W$ of $\partial D = \Gamma_q$ so that $W \cap D \subset U'$. Otherwise we would have $W$ with $W \supset \partial D$ and $W \cap (\Sigma \setminus D) \subset U'$, which would imply $F_q \subset \bigcup_{j=1}^p \Delta_j$, a contradiction. By smoothing $\tilde{\phi}$, we obtain the required isotopy. We compute
\begin{align*}
    \sum_{j=1}^{p+1} (\cH^2(D_j) - c\cH^3(\Omega^j))
    & \leq \sum_{j=1}^p (\cH^2(\Delta_j) - c\cH^3(\tilde{\Omega}^j)) + \cH^2(D) - c\cH^3(K) + \eps\\
    & \leq \sum_{j=1}^{q-1} \cH^2(F_j) + E(\hat{\Omega}_*) + \cH^2(D) - c\cH^3(K) + \eps\\
    & \leq \sum_{j=1}^{q-1} \cH^2(F_j) + E(\Omega) + \cH^2(F_q) + 2\eps\\
    & = \sum_{j=1}^{q} \cH^2(F_j) + E(\Omega) + 2\eps.
\end{align*}
Since $\eps > 0$ is arbitrary, the desired inequality holds.
\end{proof}

\subsection{Gamma-reduction of minimizing sequence}\label{subsec:gamma_min}

We now modify the initial minimizing sequence $\{\Sigma_k\}$ to a new sequence $\{\tilde{\Sigma}_k\}$ with $\lim \mathbf{v}(\tilde{\Sigma}_k) = \lim \mathbf{v}(\Sigma_k) = V$ so that each $\tilde{\Sigma}_k$ is strongly $(\gamma,U)$-irreducible for some $\gamma$.

For $q \in \N$ and $\gamma \in (0, \delta^2/9)$, let $k_q(\gamma)$ be the largest integer such that there exist $\Sigma_q^{(j)} \in \mathcal{C}(\Sigma_q, U)$ for $j = 1, \hdots, k_q(\gamma)$ with
\[ \Sigma_q^{(k_q(\gamma))} \overset{(\gamma,U)}{<} \hdots \overset{(\gamma,U)}{<} \Sigma_q^{(2)} \overset{(\gamma,U)}{<} \Sigma_q^{(1)} = \Sigma_q. \]

We claim that $k_q(\gamma)$ is bounded independent of $q$ and $\gamma$. Indeed, there is an upper bound on the topology of $\Sigma_q$ and $\cH^2(\Sigma_q)$ independent of $q$ and $W$ by assumption. Then the uniform bound on $k_q(\gamma)$ follows from \cite[(3.9)]{MSY}.

We also note that $k_q(\gamma) \leq k_q(\gamma')$ for $\gamma < \gamma'$, which follows by definition.

Let $l(\gamma)$ be the nonnegative integer defined by
\[ l(\gamma) \coloneqq \limsup_{q\to\infty} k_q(\gamma). \]
$l$ is an nondecreasing integer-valued function of $\gamma$, so there is $\gamma_0 \in (0, \delta^2/9)$ so that
\[ l(\gamma) = l(\gamma_0) \ \text{for\ all}\ \gamma \in (0, \gamma_0]. \]
For each $n \geq \gamma_0^{-1}$, there is a $q_n \geq n$ so that $k_{q_n}(1/n) = l(1/n) = l(\gamma_0)$. We set
\[ \tilde{\Sigma}_n \coloneqq \Sigma_{q_n}^{(k_{q_n}(1/n))}. \]
Hence, $\tilde{\Sigma}_n$ is strongly $(\gamma_0, U)$-irreducible for all sufficiently large $n$.

By \cite[(3.10)]{MSY}, we have
\[ \cH^2(\tilde{\Sigma}_n \triangle \Sigma_{q_n}) \leq C/n \]
for some $C$ independent of $n$. Hence, $\lim \mathbf{v}(\tilde{\Sigma}_n) = \lim \mathbf{v}(\Sigma_k)$. 

Let $\tilde{\Omega}_n \subset M$ be the open set corresponding to $\tilde{\Sigma}_n$ (which is obtained from $\Omega_{q_n}$ in the obvious way at each $(\gamma, U)$-reduction). Moreover, since $\Sigma_{q_n}$ can be recovered from $\tilde{\Sigma}_n$ be deleting disks and adding thin tubes, we have
\[ \cA^h(\tilde{\Omega}_n) \leq \inf_{\phi \in \mathcal{I}(U)} \cA^h(\phi(1,\tilde{\Omega}_n)) + \eps_n' \]
where $\eps_n' \to 0$ as $n \to \infty$.

\section{Interior Regularity}\label{sec:MSYintreg}
We prove the interior regularity part of Theorem \ref{thm:main_isotopy_regularity}. To clarify, the ``sufficiently small'' assumption on $U$ is $\cH^3(U) \leq \eta/8$, where $\eta$ is defined in (\ref{eqn:eta_assumption}).

\begin{proof}[Proof of Theorem \ref{thm:main_isotopy_regularity} Interior Regularity]
By \S\ref{subsec:gamma_min}, it suffices to consider a minimizing sequence $\{\Sigma_k\}$ so that $\Sigma_k$ is strongly $(\gamma, U)$-irreducible. By Remark \ref{rem:bubbles}, we can suppose there are no closed components of $\Sigma_k$ in $U$ contained in an open 3-ball whose boundary does not intersect $\Sigma_k$.

Let $x_0 \in \text{spt}\|V\|$. Let $\rho_0(x_0) \coloneqq d(x_0, \partial U)$. By the monotonicity of (\ref{eqn:monotonicity}), there is a constant $c_2$ so that
\begin{equation}\label{eqn:monotonicity_MSY}
    \sigma^{-2}\|V\|(B_\sigma(x_0)) \leq c_2\rho^{-2}\|V\|(B_\rho(x_0))
\end{equation}
for any $\sigma < \rho \leq \rho_0(x_0)$.

By the coarea formula,
\[ \int_{\rho-\sigma}^\rho \cH^1(\Sigma_k \cap \partial B_s(x_0))\ ds \leq \cH^2(\Sigma_k \cap (\overline{B}_\rho(x_0) \setminus B_{\rho-\sigma}(x_0))) \]
for almost every $\rho \in (0, \rho_0(x_0))$, $\sigma \in (0,\rho)$. Taking $\sigma = \rho/2$, (\ref{eqn:monotonicity_MSY}) implies
\[ \rho^{-1}\int_{\rho/2}^\rho \cH^1(\Sigma_k \cap \partial B_s(x_0))\ ds \leq C\rho \]
for all sufficiently large $k$, where $C$ depends on $\mu$, $\rho_0(x_0)$, and $\|V\|(U)$. Then there is a sequence $\rho_k \in (3\rho/4, \rho)$ so that $\Sigma_k$ has transverse intersection with $\partial B_{\rho_k}(x_0)$ and
\[ \cH^1(\Sigma_k \cap \partial B_{\rho_k}(x_0)) \leq C\rho \leq C\eta_1\rho_0(x_0) \]
for all sufficiently large $k$, provided $\rho \leq \eta_1\rho_0(x_0)$ for some $\eta_1 \in (0,1)$ to be determined. For $\eta_1$ sufficiently small and $k$ sufficiently large, we can apply Theorem \ref{thm:gamma_reduction} to $\Sigma_k$ with $U' = B_{\rho_k}(x_0)$. Hence, there are disks $D_k^{(1)}, \hdots, D_k^{(q_k)} \subset \Sigma_k$ and isotopies $\phi^{(k)} \in \mathcal{I}(U)$ so that for $k$ sufficiently large, we have
\begin{itemize}
    \item $\partial D_k^{(l)} \subset \partial B_{\rho_k}(x_0)$,
    \item $\Sigma_k \cap B_{\rho_k}(x_0) = (\bigcup_{l=1}^{q_k} D_k^{(l)}) \cap B_{\rho_k}(x_0)$,
    \item $\phi^{(k)}(1, D_k^{(l)}) \setminus \partial D_k^{(l)} \subset B_{\rho_k}(x_0)$,
    \item $\sum_{l=1}^{q_k} \cH^2(D_k^{(l)}) \leq \tilde{C}\eta_1^2\rho_0(x_0)^2$ (see (\ref{eqn:area_gamma_red})).
\end{itemize}

For $\eta_1$ sufficiently small and $k$ sufficiently large, we can apply Lemma \ref{lem:replacement} to produce disks $\tilde{D}_k^{(l)}$ with $\partial \tilde{D}_k^{(l)} = \partial D_k^{(l)}$, $\tilde{D}_k^{(l)} \setminus \partial D_k^{(l)} \subset B_{\rho_k}(x_0)$, and
\[ \cH^2(\tilde{D}_k^{(l)}) \leq \cH^2(D_k^{(l)}) + T_h(\tilde{D}_k^{(l)}, D_k^{(l)}). \]

By taking a subsequence, we can assume that $\rho_k \to \rho_{\infty}$. Take an arbitrary $\alpha > 0$ small. We can construct oriented smooth disks $\tilde{\Sigma}_k \in \mathcal{D}^*(B_{\rho_{\infty} - \alpha}(x_0))$ so that (for all $k$ sufficiently large)
\begin{itemize}
    \item $\tilde{\Sigma}_k \cap B_{\rho_{\infty} - \alpha}(x_0) = \bigcup_l \tilde{D}_k^{(l)} \cap B_{\rho_{\infty} - \alpha}(x_0)$,
    \item $\partial \tilde{\Sigma}_k \cap \overline{B}_{\rho_{\infty}-\alpha}(x_0) = \varnothing$,
    \item $\tilde{\Sigma}_k \setminus B_{\rho_{\infty} - \alpha}(x_0)$ is connected.
\end{itemize}
Indeed, we can inductively add bridges outside $B_{\rho_{\infty}-\alpha}(x_0)$ joining the disks $\tilde{D}_k^{(l)}$.
Then $\tilde{\Sigma}_k$ satisfies the hypotheses of Theorem \ref{thm:local_regularity} with the set $B_{\rho_{\infty}-\alpha}(x_0)$. Hence, the interior regularity of Theorem \ref{thm:main_isotopy_regularity} follows.
\end{proof}

\section{Boundary Regularity}
We prove the boundary regularity part of Theorem \ref{thm:main_isotopy_regularity}. We follow the proof of \cite{DeLellis_Pell}.

\subsection{Wedge property}
\begin{definition}\label{def:meets_transversally}
Let $U = B_\rho(x) \subset M$ for $\rho \ll 1$ and let $\sigma \subset \partial U$ be a disjoint union of finitely many smooth Jordan curves. An open set $A \subset U$ \emph{meets $\partial U$ in $\sigma$ transversally} if there is a positive angle $\theta_0$ so that
\begin{itemize}
    \item $\partial A \cap \partial U \subset \sigma$,
    \item for every $p \in \partial A \cap \partial U$, if we choose coordinates so that $T_p \partial U = \{z = 0\}$ and $T_p\sigma = \{y = z = 0\}$, then every point $q = (q_1, q_2, q_3) \in A$ satisfies $\frac{q_3}{q_2} \geq \tan(1/2 - \theta_0)$.
\end{itemize}
\end{definition}

\begin{lemma}[Wedge Property]\label{lem:wedge}
There is a $r_0 < 1$ so that the following holds for $\rho < r_0$. Let $V$ and $\Sigma := \partial \Omega_0 \cap U$ as in Theorem \ref{thm:main_isotopy_regularity}. Then there exists an open set $A \subset U$ which meets $\partial U$ in $\partial \Sigma$ transversally and satisfies $\mathrm{spt}\|V\| \subset \overline{A}$.
\end{lemma}

We use the following result, see \cite[Lemma 5.3]{DeLellis_Pell}.
\begin{proposition}\label{prop:convex_hull}
If $\beta \subset \partial B_1(0) \subset \R^3$ is a disjoint union of finitely many $C^2$ Jordan curves, then the convex hull of $\beta$ meets $\partial B_1(0)$ in $\beta$ transversally.
\end{proposition}

\begin{proof}[Proof of Lemma \ref{lem:wedge}]
Consider the rescaled exponential coordinates induced by the chart $f_{x,\rho} : \overline{B}_{\rho}(x) \to \overline{B}_1(0)$ given by $f_{x,\rho}(z) = (\exp_x^{-1}(z))/\rho$. We apply Proposition \ref{prop:convex_hull} and let $B$ be the convex hull of $\beta = f(\partial \Sigma)$.

Fix $\xi \in \beta$ and let $\pi_1$ and $\pi_2$ be the halfplanes delimiting the wedge of $B$. We assume without loss of generality that $\pi_1 = \{x_3 \leq a\}$ for some $a \in (0, 1)$.

Let $C_t = (0, 0, -t)$, and let $S_t = B_2(C_t)$. Assuming $\rho$ is sufficiently small, (\ref{eqn:bounded_first_variation}) implies $\|\delta (f_{\#}V)\| \leq \|f_{\#}V\|$. By \cite[Theorem 7]{WhiteMax}, $\text{spt}\|V\|$ is contained is $\overline{S}_{t_0}$, where $t_0$ is such that $\partial S_{t_0} \cap \partial B_1(0) \subset \pi_1$. The same argument applies to $\pi_2$ and $\xi$ is arbitrary, so the conclusion follows.
\end{proof}

\subsection{Tangent cones at the boundary}
Let $x \in \partial U$. As in the previous section, consider the chart $f_{x,\rho} : B_{\rho}(x) \to B_1(0)$ given by $f_{x,\rho}(z) = (\exp_x^{-1}(z))/\rho$. Fix coordinates in $\R^3$ so that $f_{x,\rho}(U \cap B_{\rho}(x))$ converges to the half ball $B_1^+(0) = B_1(0) \cap \{x_1 > 0\}$ as $\rho \to 0$.

We first follow the ideas of \cite[Lemma 6.4]{DeLellis_Ramic}.

\begin{lemma}\label{lem:boundary_monotonicity_prelim}
$\|V\|(\partial \Sigma) = 0$.
\end{lemma}
\begin{proof}
Let $V^r = V \llcorner G_2(\overline{U} \setminus \partial \Sigma)$ and $V^s = V - V^r$.

Let $X$ be a compactly supported vector field on $\overline{U}$ that vanishes on $\partial \Sigma$. Let $\phi_{\delta} \in C^{\infty}_c(\overline{U}\setminus \partial \Sigma)$ satisfy
\begin{itemize}
    \item $\phi_{\delta} = 1$ outside $B_{2\delta}(\partial \Sigma)$,
    \item $|\nabla \phi_{\delta}| \leq C \delta^{-1}$.
\end{itemize}
Since $X$ vanishes on $\partial \Sigma$, we have $|X| \leq C\delta$ in $B_{2\delta}(\partial \Sigma)$. Hence, $|\nabla (\phi_{\delta}X)| \leq C$. Then by dominated convergence and the fact that $V$ has bounded first variation in $\overline{U} \setminus \partial \Sigma$, we have
\begin{align*}
    \delta V^r(X)
    & = \int_{\overline{U}} \mathrm{div}_\pi X(p)\ dV^r(p, \pi)\\
    & = \lim_{\delta \to 0} \int_{\overline{U}} \mathrm{div}_\pi (\phi_\delta X)(p)\ dV^r(p, \pi)\\
    & = \lim_{\delta \to 0} \int_{\overline{U}} \mathrm{div}_\pi (\phi_\delta X)(p)\ dV(p, \pi)\\
    & = \lim_{\delta \to 0} \int_{\overline{U}} H_V \cdot (\phi_\delta X)(p)\ dV(p, \pi)\\
    & = \lim_{\delta \to 0} \int_{\overline{U}} H_V \cdot X(p)\ dV(p, \pi) = \delta V(X).
\end{align*}
Hence, $\delta V^s = 0$.

Let $W$ be a small tubular neighborhood of $\partial \Sigma$. For each $p \in W \setminus \partial \Sigma$, let $q$ be the unique nearest point in $\partial \Sigma$. Let $X(p)$ be the tangent vector at $p$ to the minimizing geodesic joining $p$ to $q$ parametrized on $[0,1]$. Set $X(p) = 0$ for $p \in \partial \Sigma$. If $e_1, e_2, e_3$ is a smooth orthonormal frame over $\partial \Sigma$ with $e_1$ tangent to $\partial \Sigma$, we have $\nabla_{e_1}X = 0$, $\nabla_{e_2}X = -e_2$, and $\nabla_{e_3}X = -e_3$ along $\partial \Sigma$. For any 2-plane $\pi$, we have
\[ \mathrm{div}_\pi X(p) = \sum_{i=1}^2 g(\nabla_{f_i}X, f_i) \]
for $\{f_i\}$ an orthonormal basis for $\pi$. Hence, $\mathrm{div}_\pi X(p) \leq -1$. Then
\begin{align*}
    0 = \delta V^s(X)
    & = \int_{\overline{U}} \mathrm{div}_\pi X(p)\ dV^s(p, \pi) \leq -\|V^s\|(\overline{U}) = - \|V\|(\partial \Sigma),
\end{align*}
so $\|V\|(\partial \Sigma) = 0$.
\end{proof}

By Lemma \ref{lem:boundary_monotonicity_prelim}, we can apply Allard's boundary monotonicity formula \cite[3.4(2)]{Allard_boundary}. Then the varifolds $V_{x,\rho} \coloneqq (f_{x,\rho})_{\#}V$ have uniformly bounded mass with respect to $\rho$. Hence, if we take a sequence $\rho_j \to \infty$, there is a subsequence (not relabeled) so that $V_{x,\rho_j} \to W$. We note by \cite[Theorem 42.7]{simon_old} that $W$ is stationary in $B_1^+(0)$.

By Lemma \ref{lem:wedge}, there is a positive angle $\theta_0$ so that (after changing coordinates) $\|W\|$ is supported in $\{|x_2| \leq x_1\tan(\theta_0)\}$. Then $\text{spt}\|W\| \cap \{x_1 = 0\} = \{x_1 = x_2 = 0\} =: l$.

By the same argument as Lemma \ref{lem:boundary_monotonicity_prelim}, we have $\|W\|(l) = 0$. Hence, by Allard's boundary monotonicity formula \cite[3.4(2)]{Allard_boundary}, we have
\begin{equation}\label{eqn:boundary_monotonicity}
    \|W\|(B_\rho(0)) = \pi\theta(\|V\|, x)\rho^2
\end{equation}

By Allard's reflection principle \cite[3.2]{Allard_boundary}, we have $W' \coloneqq W + r_{\#}W$ is stationary. By (\ref{eqn:boundary_monotonicity}) and \cite[\S5.1, Corollary 2]{allard}, $W'$ is a cone, which implies $W$ is a cone.

By Allard's boundary regularity theorem (see \cite[\S4]{Allard_boundary}), it suffices to show that any such $W$ is a half-disk of the form
\[ P_\theta \coloneqq \{y_2 = y_1\tan(\theta)\} \cap B_1^+(0) \]
for some $\theta \in (-\pi/2, \pi/2)$.

By the classification of one dimensional stationary integral varifolds from \cite{Allard_Almgren}, we have
\[ W = \sum_{i=1}^N k_iP_{\theta_i} \]
for $\theta_i \in [-\theta_0, \theta_0]$ and $k_i \in \N$. It therefore remains to show $N = 1$ and $k_1 = 1$.

\subsection{Reduction to the area minimization case} Consider a subsequence so that
\[ S_n \coloneqq f_{x_0, \rho_n}(\Sigma_{k_n} \cap B_{\rho_n}(x_0)) \]
converges to $W$. We further suppose that $\rho_n^{-2}\eps_{k_n} \to 0$. For any isotopy $\phi$ supported in $B_1^+(0)$, let $\tilde{\phi}$ be the isotopy of $B_{\rho_n}(x_0) \cap U$ given by conjugating with the chart $f_{x_0,\rho_n}$. We have
\begin{align*}
    & \cH^2(\Sigma_{k_n} \cap B_{\rho_n}(x_0))
    - \int_{\Omega_{k_n} \cap B_{\rho_n}(x_0) \cap U} h\ d\cH^3\\
    & \hspace{2cm} \leq \cH^2(\tilde{\phi}(1, \Sigma_{k_n} \cap B_{\rho_n}(x_0))) - \int_{\tilde{\phi}(1, \Omega_{k_n} \cap B_{\rho_n}(x_0) \cap U)} h\ d\cH^3 + \eps_{k_n}.
\end{align*}
Since
\begin{align*}
    & \cH^2(\Sigma_{k_n} \cap B_{\rho_n}(x_0))
    = \rho_n^2 \cH^2(S_n) + o(\rho_n^2),\\
    & \cH^2(\tilde{\phi}(1, \Sigma_{k_n} \cap B_{\rho_n}(x_0)))
    = \rho_n^2\cH^2(\phi(1, S_n)) + o(\rho_n^2),\\
    & \cH^3(B_{\rho_n}(x_0) \cap U)
    = \rho_n^3\cH^3(B_1^+(0)) + o(\rho_n^3) = o(\rho_n^2),
\end{align*}
we have
\[ \cH^2(S_n) \leq \cH^2(\phi(1, S_n)) + \eps_n', \]
where $\eps_n' \to 0$. Hence, $S_n$ is $\eps_n'$-area minimizing in its isotopy class, so the rest of the proof follows exactly as in \cite[\S8]{DeLellis_Pell}. Therefore, the boundary regularity of Theorem \ref{thm:main_isotopy_regularity} follows.

\part{Estimates and compactness for minimizers}\label{part:compactness}
We develop a compactness theory to prove an analogue of the curvature estimates for stable surfaces that works in our lower regularity setting.

\section{$C^{1,1}$ Compactness}

Let $U \subset M$ be a precompact open set with $C^1$ boundary. Recall from \S\ref{sec:local_control} that the map $\phi_x : \overline{B}_{\rho_0}^{T_xM}(0) \to \overline{B}_{\rho_0}^M(x)$ is the diffeomorphism defined by the exponential map for $x \in M$.

Let $\Sigma^2$ be a compact orientable surface (possibly with boundary) and let $f : \Sigma \to M^3$ be 2-sided $C^{1,1}$ immersion such that $f(\Sigma)\subset U$ and $f(\partial\Sigma)\subset\partial U$. Let $\nu:\Sigma\to TM$ denote a (Lipschitz) choice of unit normal vector field for the immersion. We endow $\Sigma$ with the pullback metric $f^*g$ and for $p\in\Sigma$, we denote by $\cB^\Sigma_r(p)$ the image $f(B^\Sigma_r(p))\subset M$ of the ball of radius $r$ centered at $p$ in $\Sigma$.

Note that, since $f$ is a $C^{1,1}$ immersion, for any $p\in\Sigma$ there is $0<\rho(p)\leq\rho_0$ such that if $\Sigma_p$ is the component of $f^{-1}(B^M_{\rho(p)}(f(p)))$ which contains $p$, then 
\begin{itemize}
    \item $\Sigma_p\cap\partial\Sigma=\varnothing$,
    \item $f$ is injective on $\Sigma_p$,
    \item and $\phi^{-1}_{f(p)}(f(\Sigma_p))$ can be represented as the graph of a $C^{1,1}$ function $u_p$ over an open set $B_p\subset df_p(T_{f(p)}\Sigma)\subset T_{f(p)}M$ containing 0.
\end{itemize}

Then, for any $p \in \Sigma$ we can define the \emph{pointwise Lipschitz constant} $\Lip_p\nu$ of $\nu$ to be the  following quantity
\begin{equation}
     \Lip_p\nu\coloneqq\limsup_{B_p\ni \xi\to 0}\frac{|\nu_p(0)-\nu_p(\xi)|}{|\xi|}
\end{equation}
where 
\begin{itemize}
    \item $\nu_p:B_p\to T_{f(p)}M$ is the unit normal to the graph of $u_p$ in $(T_{f(p)}M, \phi_{f(p)}^*g)$,% given by $$\nu_p=\frac{(-\nabla u_p,1)}{\sqrt{1+|\nabla u_p|^2}}$$ in geodesic normal coordinates,
    \item gradients and distances are taken with respect to the pullback metric on $\R^3$.
\end{itemize}

\begin{remark}\label{rmk:extend_Lip}
It is easy to see that for every $p\in\Sigma$ there is $0<r(p)\leq\rho(p)$ such that for any $\xi,\eta\in \phi_x^{-1}(\cB^\Sigma_{r(p)}(p))$ we have $$|\nu_p(\xi)-\nu_p(\eta)|\leq 2\Lip_p\nu\,|\xi-\eta|.$$ To simplify the notation, we redefine $\rho(p)$ to be the same as $r(p)$.
\end{remark}

We start by proving a simple technical fact, which plays the same role in the proof of compactness as \cite[Lemma 2.4]{CM}.

\begin{lemma}\label{lem:small_curv_graph}
Let $\Sigma^2$ be a compact orientable surface (possibly with boundary) and let $f : \Sigma \to M^3$ be 2-sided $C^{1,1}$ immersion such that $f(\Sigma)\subset U$ and $f(\partial\Sigma)\subset\partial U$. Let $\nu:\Sigma\to TM$ be a (Lipschitz) choice of unit normal vector field for the immersion such that $\sup_{p\in\Sigma}\Lip_p\nu\leq C_0$. Then, for any $K\subset\subset U$ there is $r_K=r_K(\rho_0, \dist_M(K,\partial U), C_0)>0$ such that if $x=f(p)\in f(\Sigma)\cap K$ and $r \leq r_K$, then
\begin{enumerate}
    \item $\cB^\Sigma_{2r}(p)$ can be written as the graph of a $C^{1,1}$ function $u$ over an open neighborhood of $0$ in $df_p(T_p\Sigma) \subset T_xM$, such that $|\nabla u|\leq 1$ and $\Lip\nabla u\leq 2\sqrt{2} C_0$,
    \item the connected component of $f^{-1}(B^M_{r}(x))$ containing $p$ is contained in $B^\Sigma_{2r}(p)$.
\end{enumerate}
\end{lemma}

\begin{proof}
Since $K\subset\subset U$, then $\delta_K\coloneqq\mathrm{dist}_M(K,\partial U)>0$. Hence $\dist_{f(\Sigma)}(f(\Sigma)\cap K, f(\partial\Sigma))\geq \delta_K>0$. Let $r\leq\min\{\rho_0/2, \delta_K/2\}>0$ and let $p\in\Sigma$ be such that $x=f(p)\in f(\Sigma)\cap K$. Then $\cB^\Sigma_r(p)\subset B^M_r(x)\subset\subset U$ stays a fixed positive distance away from $\partial U$. 

Now, let $y=f(q)\in \cB_r^{\Sigma}(p)$ and let $\tilde{\gamma}:[0,1]\to\Sigma$ be a curve from $p$ to $q$ in $\Sigma$ such that $\gamma = f\circ\tilde{\gamma}$ is a minimizing constant speed geodesic from $x$ to $y$ in $f(\Sigma)$.  

By compactness, we can pick $0=t_0<t_1<\dots<t_N<t_{N+1}=1$ so that the segment $\overline{x_jx_{j+1}}=\gamma([t_j,t_{j+1}])\subset \cB^\Sigma_r(p)$ (which is a minimizing geodesic from $x_j=f(p_j)$ to $x_{j+1}=f(p_{j+1})$) lies in $f(\Sigma_{p_j})$ for $j=0,\dots, N$. Hence, by \eqref{eqn:prelim_dexp_small} and Remark \ref{rmk:extend_Lip},
\begin{align*}
    |d\phi_x^{-1}\nu(p_j)-d\phi_x^{-1}\nu(p_{j+1})| &\leq 4|d\phi_{x_j}^{-1}\nu(p_j)-d\phi_{x_j}^{-1}\nu(p_{j+1})|\\
    &\leq 8\,(\Lip_{p_j}\nu)\,d_\Sigma(x,y)|t_j-t_{j+1}|,
\end{align*}
and, using the uniform pointwise Lipschitz bound,
\begin{align*}
    |d\phi_x^{-1}\nu(p)-d\phi_x^{-1}\nu(q)| &\leq \sum_{j=0}^N|d\phi_x^{-1}\nu(p_j)-d\phi_x^{-1}\nu(p_{j+1})|\\
    &\leq 8C_0\, d_{f(\Sigma)}(x,y) \sum_{j=0}^N|t_j-t_{j+1}|\\
    &= 8C_0\, d_{f(\Sigma)}(x,y)\\
    &\leq 8C_0 r.
\end{align*}

In particular, this proves that for any $y=f(q)\in \cB^\Sigma_r(p)$, we have $|d\phi_x^{-1}\nu(p)-d\phi_x^{-1}\nu(q)|\leq 8C_0r$. By taking $r\leq 2r_K\coloneqq\min\{\rho_0/2, \delta_K/2, \frac{1}{16C_0}\}$, we have that $$|d\phi_x^{-1}\nu(p)-d\phi_x^{-1}\nu(q)|\leq \frac{1}{2}$$ for all $y=f(q)\in \cB^\Sigma_r(p)$, which implies that the whole of $\cB^\Sigma_r(p)$ can be written as the graph of a $C^{1,1}$ function $u$ over an open subset of $df_p(T_p\Sigma)\subset T_xM$.

Then it is immediate to see (as in \cite[Lemma 2.4]{CM}) that $u$ satisfies $|\nabla u|\leq 1$ and $\Lip\nabla u\leq 2\sqrt{2} C_0$.

The second statement follows from the first as in \cite[Lemma 2.4]{CM}.
\end{proof}

The next lemma is essentially a geometric Arzelà-Ascoli theorem, and it follows indeed from the usual Arzelà-Ascoli compactness theorem. The argument is well-known and varitions of it appear in several difference sources and contexts (see for example \cite[Theorem 11.8]{ExtFlows}). We shall now outline it for the sake of completeness.

\begin{lemma}\label{lem:Geo_AA}
Let $\{\Sigma_i\}_{i \in \N}$ be a sequence of compact orientable surfaces and let $f_i : \Sigma_i \to U$ be a sequence of 2-sided $C^{1,1}$ immersions with unit normal vector field $\nu_i$ in a precompact open set $U \subset M$ with $f(\partial\Sigma_i)\subset\partial U$. If the immersions have locally bounded area, i.e. for every $K\subset\subset U$ there is a finite $c_K$ such that  $\sup_i\cH^2(f_i(\Sigma_i)\cap K)\leq c_K$, and
\[ C_0 := \sup_i \sup_{p \in \Sigma_i}\Lip_p\nu_i < \infty, \]
then there is a subsequence (not relabeled) so that $f_i$ converges in $C^{1,\alpha}_{\text{loc}}$ to a $C^{1,1}$ immersion $f : \Sigma \to U$ with $\sup_{p \in \Sigma} \mathrm{Lip}_p \nu \leq C_0$.
\end{lemma}
\begin{proof}[Sketch of proof.]
% Note that, by Lemma \ref{lem:small_curv_graph}, there is a positive constant $r_K$ such that for any $i$, if $x\in f_i(\Sigma_i)\cap K$, then for all $r\leq r_K$, any component of $f_i^{-1}(B^M_{r}(x))$ is contained in $B^{\Sigma_i}_{2r}(p_i)$ for some $p_i \in f_i^{-1}(x)$, and its image under $f$ is the graph of a $C^{1,1}$ function satisfying explicit $C^{1,1}$-bounds depending on $C_0$ only.

Let $K\subset\subset U$. We can cover the compact set $\overline{K}$ with finitely many balls $\{B_k=B^M_{r_K/2}(z_k)\}_{k=1}^N$. Suppose $p_i \in \Sigma_i$ satisfies $f_i(p_i) = x_i \in B_k$. By Lemma \ref{lem:small_curv_graph}, the component of $f_i^{-1}(B^M_{r_K/2}(x_i))$ is contained in $B_{r_K}^{\Sigma_i}(p_i)$ and its image under $f_i$ is the graph of a $C^{1,1}$ function satisfying explicit $C^{1,1}$-bounds depending only on $C_0$ (not on $i$). This image is the graph of a function $u_i$ such that $|\nabla u_i|\leq 1$---in particular, it has area bounded from below by some uniform constant $\kappa r_K^2$. By the uniform upper bounds on the area of $f_i(\Sigma_i)$ in compact sets, we conclude that $f_i^{-1}(B_k)$ is contained in at most $N_0$ balls of radius $r_K$ in $\Sigma_i$, where $N_0$ is a constant depending only on $K$, the area bounds, and $C_0$ (not on $i$ or $k$).

%For any $i$, if $ f_i(\Sigma_i)\cap B_k\neq\varnothing$, then there is some $x_i\in f_i(\Sigma_i)\cap B_k$, and any component of $f_i^{-1}(B_k)$ containing $p_i \in f^{-1}(x_i)$ has its image under $f$ in $B^M_{r_K}(z_k)$. This image is the graph of a function $u_i$ such that $|\nabla u_i|\leq 1$, so, in particular, it has area bounded below by some uniform constant $\kappa r_K^2$. By the uniform upper bound on $\cH^2(\Sigma_i)$, we conclude that for any $k=1,\dots, N$ and $i\in\N$, if $B_k\cap f_i(\Sigma_i)\neq\varnothing$, then $f_i(\Sigma_i)\cap B^M_{r_K}(z_k)$ contains a union of a finite number $N_{k,i}$ of $C^{1,1}$ graphs, and $N_{k,i}$ is bounded above by a universal constant $N_0$ depending on $K$ and the upper bounds on area and the Lipschitz constant.

For any $k\in\{1, \dots, N\}$, we can select a set of at most $N_0$ points $\{p_l^{i,k}\}_l \subset \Sigma_i$ so that $f_i^{-1}(B_k)$ is covered by $\{B^{\Sigma_i}_{r_K}(p_l^{i,k})\}_l$. Up to a subsequence we may assume as $i\to \infty$ that $f(p_l^{i,k}) \to x_l^k$ and $df_{p_l^{i,k}}(T_{p_l^{i,k}}\Sigma_i) \to P_l^k$, a 2-dimensional subspace of $T_{x_l}M$. Note that the corresponding graphical functions satisfy uniform $C^{1,1}$ bounds, so by Arzelà-Ascoli, up to a subsequence, they converge (in $C^{1,\alpha}$ for any $0<\alpha<1$) to some $C^{1,1}$ function $u_l^k$ defined on $B^{T_{x_l^k}M}_{r_K}(0)\cap P_l^k$ with the same $C^{1,1}$ bound.

%such that $f_i(\Sigma_i)\cap B_k\neq\varnothing$ for infinitely many $i$'s we can select points $x_i$ in the intersections and the corresponding tangent planes $T_{x_i}f(\Sigma_i)$. Up to a subsequence, we may assume $x_i\to x\in B_k$ and $T_{x_i}f(\Sigma_i)\to P$, a two-dimensional subspace of $T_xM$. Note that the corresponding functions $u_i$ are uniformly bounded and satisfy $|\nabla u_i|\leq 1$, $\Lip\nabla u_i\leq 2\sqrt2 C_0$, then, by Arzelà-Ascoli, up to a subsequence, they converge (in $C^{1,\alpha}$ for any $0<\alpha<1$) to some $C^{1,1}$ function $u$ defined on $B^{T_xM}_{r_K}\cap P$ with the same $C^{1,1}$ bound.

We then repeat this argument for all $k=1, \dots, N$ and take a diagonal subsequence. It is clear that the graphs we obtained match up on the overlaps between balls (because the surfaces $f_i(\Sigma_i)$ do), so we have produced our immersed limit surface restricted to $K\subset\subset U$. 

Lastly, we take an exhaustion of $U$ by compact sets, repeat this procedure inductively and check that the limit surfaces coincide on the overlaps. 
\end{proof}

\section{Curvature Estimates}\label{sec:curvature_estimates}
We now use these compactness ideas to prove the curvature estimates.

For convenience, we denote by $\mathcal{P}$ the set of pairs $(V, \Omega)$ where $V$ is a varifold and $\Omega \subset M$ is a set of finite perimeter with the property that there is a sequence $\Omega_j \subset M$ with smooth boundary satisfying
\[ \mathbf{1}_{\Omega_j} \xrightarrow{L^1} \mathbf{1}_{\Omega},\ \ D\mathbf{1}_{\Omega_{j}} \wto D\mathbf{1}_{\Omega},\ \ \mathbf{v}(\partial \Omega_{j}) \rightharpoonup V. \]
We let $\mathcal{P}_c$ denote the pairs $(V, \Omega) \in \mathcal{P}$ so that $V$ has $c$-bounded first variation.

For $(V, \Omega) \in \mathcal{P}$, we define
\[ \cA^h(V, \Omega) := \|V\|(M) - \int_{\Omega} h\ d\cH^3. \]

\begin{definition}
A pair $(V, \Omega) \in \mathcal{P}$ is a \emph{weakly stable $h$-surface in $U$} if
\begin{itemize}
    \item $(V, \Omega)$ has the regularity of Theorem \ref{thm:main_isotopy_regularity} in $U$,
    \item for any $\phi \in \mathcal{I}(U)$, we have
    \begin{align*}
        & \left.\frac{d}{dt}\right\vert_{t = 0}  \cA^h(\phi(t,\cdot)_{\#}V, \phi(t, \Omega)) = 0,\\
        & \left.\frac{d^2}{dt^2}\right\vert_{t = 0} \cA^h(\phi(t,\cdot)_{\#}V, \phi(t, \Omega)) \geq 0.
    \end{align*}
\end{itemize}
\end{definition}

We check that weakly stable $h$-surfaces satisfy the usual stability inequality.

\begin{proposition}\label{prop:second_variation}
If $(V, \Omega)$ is a weakly stable $h$-surface in $U$ with corresponding immersion $f : \Sigma \to U$, then for any $\psi \in C_c^{\infty}(\Sigma)$ which is the pullback by $f$ of a smooth function on $U$, we have
\begin{equation}\label{eqn:second_variation}
    \int_{\Sigma} |\nabla^{\Sigma}\psi|^2 - (|A^{\Sigma}|^2+\mathrm{Ric}_M(\nu,\nu))\psi^2\ d\cH^2
    \geq \int_{\partial^*\Omega} (H^{\Sigma}h + \partial_{\nu}h - |H^{\Sigma}|^2)\psi^2\ d\cH^2.
\end{equation}
\end{proposition}
\begin{proof}
This formula follows from the standard computation of the second variation of area and potential energy (see \cite[Proposition 2.5]{BdCE}), combined with the observation that all the integrands are well defined measurable and integrable functions since the immersion is $C^{1,1}$.
\end{proof}

The next result is the $C^{1,1}$ analogue of the usual curvature estimates for stable minimal (or prescribed mean curvature) surfaces (see e.g. \cite[Theorem 3.6]{ZZPMC}).

\begin{theorem}\label{thm:main_estimate}
Let $U \subset M$ be an open set and $c > 0$. Let $(V, \Omega) \in \mathcal{P}_c$ be a weakly stable $h$-surface in $U$ for a function with $\|h\|_{C^1} \leq c$ and $\|V\|(U) \leq m_0$. Let $f : \Sigma \to U$ be the $C^{1,1}$ immersion inducing the varifold $V$ in $U$. Then there is a constant $C = C(U, c, m_0) > 0$ so that
\[ \sup_{p \in \Sigma}\ \min\{1,d_\Sigma(p, f^{-1}(\partial U))\} \ \Lip_p\nu \leq C. \]
\end{theorem}
\begin{proof}
We argue by contradiction, using a standard point-picking argument (cf. \cite[Lecture 3]{WhiteNotes}).

Suppose there is a sequence of such immersions $f_i:\Sigma_i\to M$ so that 
\[\sup_{p \in \Sigma_i}\ \min\{1,d_{\Sigma_i}(p, f_i^{-1}(\partial U))\} \ \Lip_p\nu_i=R_i\to \infty.\]
Let us denote the local Lipschitz constant $\Lip_p\nu_i$ by $l_i(p)$.

Now pick $p_i\in\Sigma_i$ so that 
\begin{equation}\label{eqn:point_pick}
    \min\{1, d_{\Sigma_i}(p_i, f_i^{-1}(\partial U))\} \ l_i(p_i)\geq\frac{1}{2}R_i.
\end{equation}
We shall rescale the metric on $U$ by $l_i(p_i)\to \infty$ and obtain a new metric $\tilde{g}_i=l_i(p_i)g$, so that, with respect to the new metric, $\tilde{l}_i(p_i)=1$. 
By \eqref{eqn:point_pick}, the rescalings $(U,\tilde{g}_i)$ converge locally smoothly to $(\R^3, g_{\mathrm{Euc}})$, and, under the rescaling, $$d_{\tilde\Sigma_i}(p_i, f_i^{-1}(\partial U))\geq\frac{1}{2}R_i.$$
Assume, without loss of generality that $f_i(p_i)\to 0\in\R^3$. For $p,p_i\in\Sigma_i$, if (under the rescaling) $d_{\tilde{\Sigma}_i}(p, p_i)\leq \frac{1}{4}R_i$, then, since the quantity $d_{\Sigma_i}(\cdot, f_i^{-1}(\partial U))\ l_i(\cdot)$ is scale invariant, $\tilde{l}_i(p)\leq \frac{R_i}{R_i/2-R_i/4}= 4$.
Thus,
\[\sup_{p\in B^{\tilde{\Sigma}_i}_{R_i/4}(p_i)}\tilde{l}_i(p)\leq 4.\]

In any compact set, the mass bound $m_0$ and the mononicity formula give a uniform mass bound for the rescaled immersions. Recalling that $R_i\to\infty$, we conclude that, by Lemma \ref{lem:Geo_AA}, (up to a subsequence) the rescaled immersions converge in $C^{1, \alpha}_\text{loc}$ to a complete, 2-sided $C^{1,1}$ immersion $f:\Sigma\to \R^3$ with Lipschitz unit normal vector field $\nu$ such that $\sup_\Sigma\,\Lip_p\nu\leq 4$.

\emph{Claim 1}: $f$ is a $C^{1,1}$ immersion with locally ordered sheets. Indeed, the $C^{1,\alpha}_{\text{loc}}$ graphical convergence preserves the local ordering of sheets of the immersions $f_i$.

\emph{Claim 2}: The image of $f$ is a smooth embedded minimal surface. Since each sheet of the rescaled immersion in a small ball has $c/R_i$ bounded first variation, the lower semicontinuity of the first variation (see \cite[Theorem 42.7]{simon_old}) implies that each sheet of $f$ in a small ball is stationary. By Allard's regularity theorem \cite{allard}, $f$ is a smooth minimal immersion. Since the sheets are ordered locally, the strong maximum principle for minimal surfaces implies that each component of the image of $f$ is a smoothly embedded minimal surface.

\emph{Claim 3}: $f$ is a stable minimal immersion. By claim 2, $f$ is a covering of a smooth minimal embedding, so we can assume without loss of generality that $f$ is an embedding. We consider the limit of (\ref{eqn:second_variation}), where we fix a function $\psi$ on $\Sigma$ with compact support and transplant it to $\Sigma_i$ using the graphical identification. Since the rescalings satisfy $|H^{\Sigma_i}| \leq c/R_i$ and $|\nabla h_i| \leq c/R_i$, the right hand side of (\ref{eqn:second_variation}) vanishes in the limit. Since the convergence is $C^{1,\alpha}_{\text{loc}}$, every term on the right hand side converges to the corresponding term for $\Sigma$ except the $|A^{\Sigma_i}|^2$ term. Since the ambient space converges locally smoothly to Euclidean $\R^3$, \cite[Theorem 4.3]{Langer} implies that the term $\int_{\Sigma_i} |A^{\Sigma_i}|^2u^2 \ d\cH^2$ is lower semicontinuous (we have local weak convergence in $W^{2,2}$ by the uniform $C^{1,1}$ bounds). Hence, $\Sigma$ satisfies the stability inequality in $\R^3$.

By the stable Bernstein theorem in $\R^3$ due to \cite{FCS}, \cite{dCP}, and \cite{Pog}, $\Sigma$ is flat. By Corollary \ref{cor:estimate}, $\Sigma_i$ converges in $C^{1,1}_{\text{loc}}$ to the flat limit. However, we chose points so that $\text{Lip}_{p_i}\nu_i \geq 1$ where $p_i \to 0$, which yields a contradiction.
\end{proof}

As usual, this estimate directly implies $C^{1,\alpha}_{\text{loc}}$ pre-compactness for weakly stable $h$-surfaces (cf. \cite[Theorem 3.6]{ZZPMC}).

\part{Application to existence in the 3-sphere}\label{part:min-max}

\section{Min-Max Setup}\label{sec:min-max_setup}
We set up an appropriate setting to carry out the min-max procedure of \cite{Colding_DeLellis} for the prescribed mean curvature functional.

Let $v > 0$, $\kappa > 0$, and $c > 0$ satisfy
\begin{equation}\label{eqn:constraints}
    \frac{8\pi}{(1 + \kappa)^2 + c^2/4} > 4\pi + vc.
\end{equation}
Consider a smooth metric $g$ on $S^3$ so that
\begin{itemize}
    \item $R_g \geq 6$,
    \item $\Ric(g) > 0$,
    \item $\mathrm{Vol}(S^3, g) \leq v$,
    \item $(S^3, g)$ isometrically embeds in $\R^N$ with $|A_{S^3}| \leq 1 + \kappa$.
\end{itemize}
Note that the round metric satisfies these conditions with $v = 2\pi^2$, $\kappa = 0$, and $c = 0.547$. Moreover, there is an explicit $C^2$ graphical neighborhood of the standard embedding $\mathbf{S}^3 \subset \R^4 \subset \R^N$ and an explicit constant $c > 0$ for which these conditions are satisfied (by rescaling). Let $h : S^3 \to \R$ be a smooth function with $|h| \leq c$. Moreover, we assume that $h$ satisfies the conditions (\dag) or (\ddag) from \cite{ZZPMC}. By \cite[Proposition 3.8 and Corollary 3.18]{ZZPMC}, the set of functions satisfying these conditions is open and dense, all positive functions satisfy (\ddag), and any such function $h$ has the unique continuation property for surfaces of prescribed mean curvature $h$.

Let $g_0$ be the round metric on $S^3$. Let $x_0 \in S^3$. Let $\{\Omega_t\}_{t \in [0, 1]}$ be given by
\[ \Omega_t := B_{\pi t}^{g_0}(x_0). \]
Note that we use the round metric in particular only to define the standard sweepout; everything else uses the general metric $g$. Let $\mathcal{M}$ be the set of smooth maps $\psi : [0, 1] \times [0, 1] \times S^3 \to S^3$ so that
\begin{itemize}
    \item $\psi(s, t, \cdot)$ is a diffeomorphism of $S^3$,
    \item $\psi(0, t, \cdot) = \mathrm{id}$.
\end{itemize}
We define the min-max width $\omega_1(g,h)$ by
\[ \omega_1(g,h) := \inf_{\psi \in \mathcal{M}} \sup_{t\in[0, 1]} \cA^h(\psi(1, t, \Omega_t)). \]

We say $\{\Omega_t^j = \psi_j(1, t, \Omega_t)\}_{j\in\N}$ for $\psi_j \in \mathcal{M}$ is a \emph{minimizing sequence} if
\[ \lim_{j \to \infty}\sup_t \cA^h(\Omega_t^j) = \omega_1(g, h). \]
Given a minimizing sequence, we say $\{\Omega_{t_j}^j\}$ is a \emph{min-max sequence} if
\[ \lim_{j \to \infty} \cA^h(\Omega_{t_j}^j) = \omega_1(g, h). \]

By the standard compactness theory (see \cite[Theorem 12.26]{maggi} and \cite[Chapter 1, Theorem 4.4]{simon_old}), for any min-max sequence $\{\Omega_{t_j}^j\}$, there is a set $\Omega \subset S^3$ and a varifold $V$ satisfying
\[ \mathbf{1}_{\Omega_{t_j}^j} \xrightarrow{L^1} \mathbf{1}_{\Omega},\ \ D\mathbf{1}_{\Omega_{t_j}^j} \wto D\mathbf{1}_{\Omega},\ \ \mathbf{v}(\partial \Omega_{t_j}^j) \rightharpoonup V, \]
and
\begin{equation}\label{eqn:achieves_width}
\|V\|(S^3) - \int_{\Omega} h\ d\cH^3= \omega_1(g,h).
\end{equation}

We show the following theorem.

\begin{theorem}\label{thm:min-max}
Suppose $(S^3, g)$ and $h$ satisfy the constraints (\ref{eqn:constraints}) and $\int_{S^3} h\ d\cH^3_g \geq 0$. Then there is an open set $\Omega \subset S^3$ satisfying
\begin{itemize}
    \item $\cA^h(\Omega) = \omega_1(g, h)$,
    \item $\partial \Omega$ is a disjoint union of smooth embedded spheres with prescribed mean curvature $h$ with respect to $\Omega$.
\end{itemize}
\end{theorem}

Since reversing orientations corresponds to reversing the sign of $h$, Theorems \ref{thm:PMC_in_sphere} and \ref{thm:CMC_in_sphere} follow from Theorem \ref{thm:min-max}.

\section{Density Estimates in the 3-Sphere}
We exhibit an upper and lower bound for the mass of a min-max varifold for the $\cA^h$ functional with a point of density at least two. We show that these two bounds are contradictory in the case of our chosen setup.

We first show an upper bound, using the upper bound for $c$, the upper bound for the volume of $(S^3, g)$, and the lower bound for the scalar curvature of $g$.

\begin{proposition}\label{prop:min_max_area_upper_bound}
Suppose $(S^3, g)$ satisfies the hypotheses of \S\ref{sec:min-max_setup}. Suppose $(V, \Omega) \in \mathcal{P}$ satisfies $\|V\|(S^3) - \int_{\Omega} h \ d\cH^3 \leq \omega_1(g,h)$. Then
\[ \|V\|(S^3) \leq 4\pi + v c. \]
\end{proposition}
\begin{proof}
Let $h_+ := \min\{0, h\}$ and $h_- := \min\{0, -h\}$. We compute
\begin{align*}
    \|V\|(S^3) - \int_{S^3} h_+\ d\cH^3
    & \leq \|V\|(S^3) - \int_{\Omega} h\ d\cH^3\\
    & \leq \omega_1(g, h)\\
    & \leq \omega_1(g, 0) + \int_{S^3} h_-\ d\cH^3\\
    & \leq 4\pi + \int_{S^3} h_-\ d\cH^3,
\end{align*}
where the inequality $\omega_1(g, 0) \leq 4\pi$ follows from the assumptions $\Ric(g) > 0$ and $R_g \geq 6$ by \cite[Theorem 1.1]{MNrigidity}. Since $|h| \leq c$ and $\cH^3(S^3,g) \leq v$, the conclusion follows.
\end{proof}

To prove the lower bound, we recall a monotonicity formula for the Willmore energy (modified for our setting).

\begin{lemma}[{\cite[(1.4)]{Simon_Willmore} and \cite[\S6]{Topping1998}}]\label{lem:willmore_monotone}
If $V$ is an integer rectifiable 2-varifold in $\R^N$ with bounded first variation and finite mass, then
\[ \pi\Theta^2(V,x_0) \leq \frac{1}{16}\int_{\R^4} |H_V|^2\ d\|V\| \]
for any $x_0 \in \R^N$.
\end{lemma}
\begin{proof}
We note that the proof given in the first part of the proof of Lemma 1 from \cite{Topping1998} applies to varifolds (where our convention for mean curvature differs from Topping's by a factor of 2). The only point where the $C^1$ assumption is used is in showing that the rightmost term in \cite[(21)]{Topping1998} vanishes as $\sigma \to 0$. Instead, we use the more trivial bound:
\begin{align*}
\int_{B_{\sigma}} (\sigma^{-2} - \rho^{-2}) X \cdot H_V\ d\|V\|
& \leq ((\sigma^{-2} - \rho^{-2}) \sigma c)\|V\|(B_{\sigma})\\
& \leq c\sigma \frac{\|V\|(B_{\sigma})}{\sigma^2} = O(\sigma)
\end{align*}
as $\sigma \to 0$ by the monotonicity formula and assumption of bounded support and finite mass.
\end{proof}

We apply this monotonicity formula to obtain the mass lower bound.

\begin{proposition}\label{prop:min_max_area_lower_bound}
Suppose $(S^3, g)$ satisfies the hypotheses of \S\ref{sec:min-max_setup}. If $V$ is a varifold in $(S^3,g)$ with $c$-bounded first variation and $\Theta^2(V,x) \geq 2$ for some $x \in S^3$, then
\[ \|V\|(S^3) \geq \frac{8\pi}{(1+\kappa)^2+c^2/4}. \]
\end{proposition}
\begin{proof}
Using the isometric embedding of $(S^3,g)$ into $\R^N$, we view $V$ as a 2-varifold in $\R^N$. By the bound on the second fundamental form of $S^3 \subset \R^N$, we have
\[ |H_V^{\R^N}|^2 \leq 4(1+\kappa)^2 + c^2. \]

By Lemma \ref{lem:willmore_monotone} and the assumption of a point of density at least 2, we have
\begin{align*}
    2\pi \leq \frac{1}{16}\int_{\R^N} |H_V^{\R^N}|^2\ d\|V\| \leq \frac{4(1+\kappa)^2 + c^2}{16}\|V\|(S^3),
\end{align*}
which yields the desired inequality.
\end{proof}

Combining the mass upper and lower bounds, we deduce the desired conclusion about the density of special varifolds appearing in the min-max construction.

\begin{corollary}\label{cor:density_one}
Suppose $(S^3,g)$ satisfies the hypotheses of \S\ref{sec:min-max_setup}. Suppose $(V, \Omega) \in \mathcal{P}$ satisfies $\|V\|(S^3) - \int_{\Omega} h\ d\cH^3 = \omega_1(g, h)$. Further suppose that the varifold $V$ has integer density everywhere and $c$-bounded first variation in $(S^3,g)$. Then $\Theta^2(V, x) = 1$ for all $x \in \mathrm{spt}\|V\|$.
\end{corollary}
\begin{proof}
Suppose otherwise for contradiction. Then by Proposition \ref{prop:min_max_area_lower_bound} and Proposition \ref{prop:min_max_area_upper_bound}, we must have
\[ 4\pi + vc \geq \frac{8\pi}{(1+\kappa)^2 + c^2/4}. \]
This inequality contradicts (\ref{eqn:constraints}).
\end{proof}

\section{Min-Max Proof}

We show that each step in \cite{Colding_DeLellis} and \cite{DeLellis_Pell} can be modified to work in our setting.

\subsection{Nontrivial width}
The fact that $\omega_1(g,h) > 0$ follows immediately from the isoperimetric inequality. By the assumption $\int_{S^3} h\ d\cH^3 \geq 0$, $\omega_1(g,h)$ cannot be achieved by a trivial solution (meaning a min-max sequence whose limit satisfies $V = 0$ and $\Omega \in \{S^3, \varnothing\}$).

\subsection{Pull-tight}
For completeness, we recall the ideas of \cite[\S4]{ZZCMC}, \cite[\S5]{ZZPMC} in the setting of the pull-tight argument of \cite{Colding_DeLellis} (specifically, we follow \S2 of the errata for \cite{Colding_DeLellis}).

Let $\mathbf{F}$ denote the varifold distance on $\mathcal{V}_2(S^3, g)$, the space of 2-varifolds in $(S^3, g)$.

Take $L = 2(\omega_1(g, h)+c\mathrm{Vol}_g(S^3))\footnote{Recall that $|h| \leq c$.}$, and let $A^L = \{V \in \mathcal{V}_2(S^3,g) \mid \|V\|(S^3) \leq 2L\}$, which is compact in $\mathbf{F}$.

Let $A^c_{\infty} = \{V \in A^L \mid |\delta V(X)| \leq c \int_{S^3}|X|d\mu_V \ \text{for all smooth vectors fields}\ X\}$. By \cite[Lemma 4.1]{ZZCMC}, $A^c_{\infty}$ is a compact subset of $A^L$ in the $\mathbf{F}$-topology.

For $k \in \N$, define $A_k = \{V \in A^L \mid 2^{-k-2} \leq \mathbf{F}(V, A^c_{\infty}) \leq 2^{-k+1}\}$.

For each $V \in A_k$ (by the definition of $A^c_{\infty}$), there is a smooth vector field $X_V$ satisfying
\[ \delta V(X_V) + c\int_{S^3} |X_V|d\mu_V < 0 \ \ \text{and} \ \ \|X_V\|_{C^k} \leq \frac{1}{k} \]
(the second condition can be guaranteed by scaling).

For each $V \in A_k$, there is an $\mathbf{F}$-metric ball $B_{\rho_V}^{\mathbf{F}}(V)$ so that
\[ \delta W(X_V) + c\int_{S^3} |X_V|d\mu_W \leq \frac{1}{2}\left(\delta V(X_V) + c\int_{S^3} |X_V|d\mu_V\right) < 0 \]
for all $W \in B_{\rho_V}(V)$.

By compactness, cover $A_k$ by finitely many such metric balls $B_{\rho_i}^{\mathbf{F}}(V_i)$. Let $\{\psi_i\}$ be a continuous partition of unity subordinate to the covering $\{B_{\rho_i}^{\mathbf{F}}(V_i)\}$, and define
\[ V \in A_k \mapsto H_V^k = \sum_i \psi_i(V)X_{V_i}. \]

Now choose $\Psi_k$ a compactly supported continuous function on $A_k$ with $0 \leq \Psi_k \leq 1$ and $\Psi_k \equiv 1$ on $\{V \in A^L \mid 2^{-k-1} \leq \mathbf{F}(V, A^c_{\infty}) \leq 2^{-k}\} \subset A_k$. On $A^L \setminus A^c_{\infty}$, define
\[ H_V = \frac{\sum_k\Psi_k(V)H_V^k}{\sum_k \Psi_k(V)}. \]
We can extend the map $V \mapsto H_V$ to $A^c_{\infty}$ with $H_V = 0$ for $V \in A^c_{\infty}$. Since $\|H_V\|_{C^k} \leq \frac{1}{k}$ if $\mathbf{F}(V, A^c_{\infty}) \leq 2^{-k+1}$ by construction (by the choice of normalization of $X_V$), the extension is a continuous map to the $C^{\infty}$ topology on the space of smooth vector fields. Moreover, $H_V$ satisfies
\[ \delta V(H_V) + c\int_{S^3} |H_V|d\mu_V < 0 \]
for all $V \in A^L \setminus A^c_{\infty}$.

The rest of the proof now follows \cite[Errata Proposition 2.1]{Colding_DeLellis} exactly (normalizing the speed of the isotopies generated by the vector fields $H_V$ and smoothing the new sweepout), with the additional observation that the first variation of $\cA^h$ satisfies
\[ \delta \cA^h_{\Omega}(X) = \delta|\partial \Omega|(X) - \int_{S^3} h\langle X, \nu_{\partial \Omega}\rangle d\mu_{|\partial \Omega|} \leq \delta |\partial \Omega|(X) + c \int_{S^3}|X| d\mu_{|\partial \Omega|} \]
(for a sweepout $\{\Omega_t\}$, the isotopies generated by the vector fields $H_{|\partial \Omega_t|}$ will be applied to the sets $\Omega_t$).

Hence, there is a minimizing sequence $\{\Omega_t^j\}$ with the property that every min-max sequence $\{\Omega_{t_j}^j\}$ satisfies $\mathbf{v}(\partial \Omega_{t_j}^j) \rightharpoonup V$ for $V \in A^c_{\infty}$.

\subsection{Almost-minimizing}
We say $\Omega\subset S^3$ is $\eps$-almost minimizing in $U \subset S^3$ if there is no isotopy $\phi \in \mathcal{I}(U)$ satisfying
\[ \cA^h(\phi(t, \Omega)) \leq \cA^h(\Omega) + \eps/8, \]
\[ \cA^h(\phi(1, \Omega)) \leq \cA^h(\Omega) - \eps. \]
We say a sequence $\{\Omega^j\}$ is almost minimizing in $U$ if $\Omega^j$ is $\eps_j$ almost minimizing in $U$ with $\eps_j \to 0$.

With this definition, we see that the argument of \cite[\S5]{Colding_DeLellis} is general and does not depend on the specific choice of functional. Hence, there is a function $r: S^3 \to \R^+$ and a min-max sequence $\{\Omega^j\}$ satisfying the conclusions of (2) so that $\{\Omega^j\}$ is almost minimizing in any annulus centered at any $x \in S^3$ with outer radius at most $r(x)$.

To allow application of our regularity theory, we assume (without loss of generality) that $\cH^3(B_{r(x)}(x)) \leq \eta/8$ for all $x \in M$.

\subsection{Regularity of weak replacements}
\begin{definition}\label{def:weak_replacement}
Let $(V, \Omega) \in \mathcal{P}_c$. A pair $(V', \Omega') \in \mathcal{P}_c$ is a \emph{weak replacement for $(V, \Omega)$ in $U \subset S^3$} if
\begin{itemize}
    \item $V'\llcorner G(M\setminus \overline{U},2) = V \llcorner G(M\setminus \overline{U}, 2)$ and $\Omega \setminus \overline{U} = \Omega' \setminus \overline{U}$,
    \item $|\|V\|(S^3) - \|V'\|(S^3)| \leq c\cH^3(U)$,
    \item $\|V'\|(S^3) - \int_{\Omega'}h\ d\cH^3 \leq \|V\|(S^3) - \int_{\Omega}h\ d\cH^3$,
    \item $(V', \Omega')$ is in the $C^{1,\alpha}_{\text{loc}}$ graphical closure of the set of weakly stable $h$-surfaces in $U$.
\end{itemize}
\end{definition}

\begin{lemma}\label{lem:regularity_weak_replacement}
If there is a function $r : S^3 \to \R$ so that $(V, \Omega) \in \mathcal{P}_c$ has a weak replacement in any annulus centered at any $x \in S^3$ with outer radius at most $r(x)$, then $V$ is integer rectifiable and any tangent cone to $V$ at any $x \in \mathrm{spt}\|V\|$ is an integer multiple of a plane.
\end{lemma}
\begin{proof}
We observe that the same proof as \cite[Lemma 6.4]{Colding_DeLellis} works in our setting.

\emph{Density lower bound}. The monotonicity formula implies there is a constant $c_2$ so that
\[ \frac{\|V\|(B_{\sigma}(x))}{\pi \sigma^2} \leq \frac{c_2\|V\|(B_{\rho}(x))}{\pi \rho^2} \]
for any varifold $V$ with $c$-bounded first variation, $x \in S^3$, and $\sigma < \rho \ll 1$. Let $r < r(x)$ sufficiently small. Let $(V', \Omega')$ be a replacement for $(V, \Omega)$ in $A(x, r, 2r)$. By the maximum principle of \cite{WhiteMax}, $V'$ is not identically 0 on $A(x, r, 2r)$. Since $V' \llcorner G(A(x, r, 2r), 2)$ is the varifold of a $C^{1,1}$ immersion, there is a $y \in A(x, r, 2r)$ with $\Theta^2(V', y) \geq 1$. Then we have
\[ \frac{\|V\|(B_{4r}(x))}{16\pi r^2} \geq \frac{\|V'\|(B_{4r}(x))}{16\pi r^2} - Cr \geq \frac{\|V'\|(B_{2r}(y))}{16\pi r^2} - Cr \geq \frac{c_2}{4} - Cr. \]
Hence, we have a uniform positive lower bound for the density by taking $r$ sufficiently small. By \cite[Theorem 5.5]{allard}, $V$ is rectifiable.

\emph{Tangent cones}. Let $x \in \mathrm{spt}\|V\|$. Let $C$ be a tangent cone to $V$ at $x$, and let $\rho_n \to 0$ satisfy
\[ (f_{x, \rho_n})_{\#}V \rightharpoonup C. \]
Let $(V_n', \Omega_n')$ be a replacement for $(V, \Omega)$ in $A(x, \rho_n/4, 3\rho_n/4)$. Let $W_n' = (f_{x,\rho_n})_{\#}V_n'$. Up to a subsequence (not relabeled), we have $W_n' \rightharpoonup C'$, where $C'$ is a stationary varifold. By the definition of the replacement, we have
\[ C' \llcorner G(B_1(0) \setminus A(0, 1/4, 3/4),2) = C \llcorner G(B_1(0) \setminus A(0, 1/4, 3/4), 2), \]
\begin{equation}\label{eqn:mass_equal}
\|C'\|(B_{\rho}(0)) = \|C\|(B_{\rho}(0)) \ \ \forall\ \ \rho \in (0,1/4) \cup (3/4, 1).
\end{equation}
Since $C$ is a cone, (\ref{eqn:mass_equal}) implies
\[ \frac{\|C'\|(B_{\sigma}(0))}{\pi \sigma^2} = \frac{\|C'\|(B_{\rho}(0))}{\pi \rho^2} \]
for all $\rho \in (0,1/4) \cup (3/4, 1)$. By stationarity and monotonoicity, we conclude that $C'$ is a cone. Since $C$ and $C'$ agree on $B_{1/4}(0)$, we have $C' = C$. By Theorem \ref{thm:main_estimate}, the convergence in $A(0, 1/4, 3/4)$ is $C^{1,\alpha}_{\text{loc}}$ graphical. By the same argument as in the proof of Claims 1 and 2 in Theorem \ref{thm:main_estimate} (i.e. the mean curvature of each sheet vanishes in the limit, and the ordering of sheets is preserved so we can apply the strong maximum principle), $C'$ is the varifold of a smooth embedded minimal surface with integer multiplicity on the annulus $A(0, 1/4, 3/4)$. Then $C'$ is (an integer multiple of) a smooth embedded minimal cone in $\R^3$, and hence a plane with integer multiplicity.
\end{proof}

\subsection{Regularity of replacements}
\begin{definition}\label{def:strong_replacement}
Let $(V, \Omega) \in \mathcal{P}_c$. A pair $(V', \Omega') \in \mathcal{P}_c$ is a \emph{replacement for $(V, \Omega)$ in $U \subset S^3$} if
\begin{itemize}
    \item $V'\llcorner G(M\setminus \overline{U},2) = V \llcorner G(M\setminus \overline{U}, 2)$,
    \item $|\|V\|(S^3) - \|V'\|(S^3)| \leq c\cH^3(U)$,
    \item $\|V'\|(S^3) - \int_{\Omega'}h\ d\cH^3 \leq \|V\|(S^3) - \int_{\Omega}h\ d\cH^3$,
    \item $V' \llcorner G(U, 2)$ is an embedded smooth stable surface with prescribed mean curvature $h$ with respect to $\Omega'$.
\end{itemize}
\end{definition}
\begin{remark}
Observe that any replacement is also a weak replacement.
\end{remark}

\begin{definition}
The pair $(V, \Omega) \in \mathcal{P}_c$ has the \emph{good replacement property} if
\begin{enumerate}
    \item there is a positive function $r : U \to \R^+$ such that there is a replacement $(V', \Omega')$ for $(V, \Omega)$ in any annulus centered at $x$ with outer radius at most $r(x)$ for all $x \in U$,
    \item there is a positive function $r' : U \to \R^+$ such that there is a replacement $(V'', \Omega'')$ for the replacement $(V', \Omega')$ from (1) in any annulus centered at $x$ with outer radius at most $r(x)$ and any annulus centered at $y$ with outer radius at most $r'(y)$ for any $y \in U$,
    \item there is a positive function $r'': U \to \R^+$ such that there is a replacement $(V''', \Omega''')$ for the replacement $(V'', \Omega'')$ from (2) in any annulus centered at $z$ with outer radius at most $r''(z)$ for any $z \in U$.
\end{enumerate}
\end{definition}

\begin{lemma}\label{lem:regularity_strong_replacement}
If $(V, \Omega) \in \mathcal{P}_c$ has the good replacement property, then $V$ is an embedded smooth surface with prescribed mean curvature $h$ with respect to $\Omega$.
\end{lemma}
\begin{proof}
Since unique continuation and curvature estimates hold, the proof is the same as the proof of \cite[Proposition 6.3]{Colding_DeLellis}. The only exception is the argument showing $\Sigma \subset \mathrm{spt}\|V\|$ on $B_\rho(x) \setminus \{x\}$, which in \cite{Colding_DeLellis} uses that $\|V\|(S^3) = \|V'\|(S^3)$. Instead, we follow the argument from \cite[Theorem 6.1 Claim 5]{ZZCMC} and \cite[Theorem 7.1 Claim 5]{ZZPMC}, which applies to our setting.
\end{proof}

\subsection{Construction of replacements}
The construction in \cite[\S7]{Colding_DeLellis} is general and applies to our setting (see also \cite[\S6]{ZZPMC}). We need only make three observations.

First, to go from the regularity of minimizers in isotopy classes to the regularity of minimizers in constrained isotopy classes, we need a squeezing lemma like \cite[Lemma 7.4]{Colding_DeLellis}. We note that the same proof applies in our setting, as volume changes after squeezing will be arbitrarily small.

Second, since the weak replacement is constructed as the limit of solutions to constrained minimization problems for $\cA^h$, the inequality
\begin{equation}\label{eqn:hypothesis_for_cor}
\|V'\|(S^3) - \int_{\Omega'}h\ d\cH^3 \leq \|V\|(S^3) - \int_{\Omega}h\ d\cH^3 = \omega_1(g, h)
\end{equation}
for the weak replacement $(V', \Omega')$ of the pair $(V, \Omega)$ from (\ref{eqn:achieves_width}) follows immediately.

Finally, we observe that the replacements constructed in this way are a priori only weak replacements. However, the pull-tight argument, Lemma \ref{lem:regularity_weak_replacement}, (\ref{eqn:hypothesis_for_cor}), and Corollary \ref{cor:density_one} imply that the weak replacements we construct are in fact strong replacements under the assumptions of Theorem \ref{thm:min-max}.

\subsection{Genus bound}
We follow the argument of \cite{DeLellis_Pell}.

The analogue of \cite[Proposition 3.2]{DeLellis_Pell}, i.e. regularity of the constrained isotopy minimization problem, follows from the above interior and boundary regularity theorems, combined with the observation that the squeezing argument at the boundary from \cite[Lemma 6.1]{DeLellis_Pell} works in our setting (after squeezing, volume changes will be arbitrarily small).

The proof of \cite[Proposition 2.1]{DeLellis_Pell} in \cite[\S4]{DeLellis_Pell} follows the same argument combined with the following observations. First, the min-max limit $V$ is a smooth embedded surface of prescribed mean curvature and satisfies unique continuation. Hence, when we solve local constrained minimization problems and take limits using the compactness theory developed in Part \ref{part:compactness}, the apriori $C^{1,1}$ local minimizers converge in $C^{1,\alpha}_{\text{loc}}$ to the smooth limit $V$. Since local minimizers are smooth where they have density 1, we can always take the local constrained minimizers to be smooth stable embedded surfaces of prescribed mean curvature. Second, the estimate \cite[Lemma 4.2]{DeLellis_Pell} holds for smooth surfaces with bounded mean curvature, where the constant $C$ depends on the mean curvature bound. The proof follows by keeping track of the nonvanishing left hand side in the first variation formula \cite[(A.2)]{DeLellis_Pell}.

Equipped with Simon's Lifting Lemma \cite[Proposition 2.1]{DeLellis_Pell}, the rest of the proof is identical.

\appendix
\section{Maximum Principle}\label{sec:max_principle}

A maximum principle is required to patch together the limits of decomposed stacked disk minimization problems. Since we obtain lower regularity than in the minimal surface setting, we need a maximum principle in the setting of $C^{1,\alpha}$ surfaces.

We adopt the following definition from \cite{ZZCMC}.

\begin{definition}
Let $N_i$ ($i=1,2$) be connected embedded two-sided $C^1$ surfaces in a connected open subset $W\subset M$ with $\partial N_i\cap W=\varnothing$. Let $\nu_i$ be the unit normal to $N_i$. We say that $N_2$ \emph{lies on one side of} $N_1$ if $N_1$ divides $W$ into two connected components $W_1\cup W_2=W\setminus N_1$, where $\nu_1$ points into $W_1$ and either:
\begin{itemize}
    \item $N_2\subset\overline{W_1}$, which we write as $N_1\leq N_2$; or
    \item $N_2\subset\overline{W_2}$, which we write as $N_1\geq N_2$.
\end{itemize}
\end{definition}

\begin{lemma}[Maximum Principle]\label{lem:maximum_principle}
Let $N_1$ and $N_2$ be $C^{1,\alpha}$ surfaces in a connected open set $W$ so that $N_1$ lies on one side of $N_2$, $N_1 \geq N_2$, $N_2 \geq N_1$, and $N_1$ and $N_2$ intersect tangentially at $x \in W$. If each $N_i$ has nonnegative mean curvature in the weak sense with respect to the normal $\nu_i$, then $N_1 = N_2$ and each is a smooth minimal surface.
\end{lemma}
\begin{proof}
Since the unit normals point in opposite directions, being smooth and minimal follows if we show $N_1 = N_2$.

Suppose for contradiction $N_1 \neq N_2$. Without loss of generality, there is a sequence $x_j \to x$ with $x_j \in N_1 \setminus N_2$. Let $P \subset T_xM$ be the common tangent plane of $N_1$ and $N_2$ at $x$, and let $\nu=\nu_1(x)$ be a unit normal to $P$, so that $\nu_2(x)=-\nu$. Let $r_0 \in (0, 1)$ be sufficiently small so that $\phi_x^{-1}(N_1 \cap B_r^M(x))$ and $\phi_x^{-1}(N_2 \cap B_{r_0}^M(x))$ are $C^{1,\alpha}$ graphs over $B \subset P$, given by $C^{1,\alpha}$ functions $u_1 \geq u_2$ and let $\rho \in (0, r_0/2)$. Let $B_r$ denote $B^{\R^2}_{r}(0)$.

We shall use the same setup as in the proof of Proposition \ref{prop:C1,1_regularity}. We let
\begin{itemize}
    \item $\tilde{g}$ be the rescaled pullback metric $(\phi_x \circ \mu_{\rho})^*g$, so that
\begin{equation*}
    \|\tilde{g} - g_{\text{Euc}}\|_{C^{\infty}(B_2 \times [-1, 1])} \leq \delta(\rho),
\end{equation*}
where $\delta(\rho) \to 0$ as $\rho \to 0$; 
    \item $\tilde{u}_i:B_2\to \R$ be defined by $\tilde{u}_i(x)=\rho^{-1}u_i(\rho x)$ so that 
\begin{equation*}
    \|\tilde{u}_i\|_{C^{1}(B_2)} \leq \delta(\rho);
\end{equation*}
    \item $A(x,z,p)$ be such that with the rescaled metric $\tilde{g}$ on $\R^3$, we have $\cH^2(\text{graph }u\vert_{B_r})=\int_{B_r}A(x,u,Du)\ dx$ for all $C^1$ functions $u:B_2\to \R$ and $r\in(0,2)$;
    \item $Q$ be the mean curvature operator for graphs over $B_2$ with respect to $\tilde{g}$.
\end{itemize}

Then $Q(u) \leq 0$ (resp. $\geq 0$) in the weak sense in $B_1$ is equivalent to (in our choice of sign convention)
\[ \int_{B_1} (A_p(x,u,Du) \cdot D\phi + A_z(x,u,Du)\phi) \leq 0 \ (\text{resp. }\geq 0) \]
for all $\phi \in W^{1,2}_0(B_1)$ with $\phi \geq 0$.

By assumption, we have $Q(\tilde{u}_1)\geq 0$ and $Q(\tilde{u}_2)\leq 0$ in the weak sense. If we let $v=\tilde{u}_1-\tilde{u}_2\geq 0$, then $v\in C^{1,\alpha}$ is a non-negative weak supersolution of a linear elliptic equation with continuous coefficients (depending on $\tilde{u}_1$ and $\tilde{u}_2$) as in \cite[Theorem 10.7]{GT}, i.e.
\begin{equation*}
    \int_{B_1} (a^{ij}D_jv + b^iv)D_i\phi - (-b^iD_iv + cv)\phi\geq 0
\end{equation*}
for all non-negative $\phi\in W_0^{1,2}(B_1)$, where $\tilde{u}_t=t\tilde{u}_1+(1-t)\tilde{u}_2$ and
\begin{align*}
    a^{ij} & \coloneqq \int_0^1 A_{p_ip_j}(x, \tilde{u}_t, D\tilde{u}_t)\ dt\\
    b^i & \coloneqq \int_0^1 A_{p_iz}(x, \tilde{u}_t, D\tilde{u}_t)\ dt\\
    c & \coloneqq -\int_0^1 A_{zz}(x, \tilde{u}_t, D\tilde{u}_t)\ dt.
\end{align*}

Since $N_1$ and $N_2$ intersect tangentially at $x$, $v(0)=0$, and $\inf_{B_{1/4}}v=0$. Hence, the weak Harnack inequality for supersolutions \cite[Theorem 8.18]{GT} implies $\|v\|_{L^1(B_{1/2})}=0$. As $v$ is $C^{1,\alpha}$, this means $v\vert_{B_{1/2}}\equiv 0$, contradicting the choice of the sequence $x_j\to x$. 
\end{proof}

% REFERENCES
\bibliographystyle{amsalpha}
\bibliography{bib}

\end{document}